\documentclass[a4paper, 12pt]{article}
\usepackage{graphicx}
\usepackage{pdfpages}
\usepackage[utf8]{inputenc}
\usepackage{tikz}
\usepackage{tikz-cd}
\usepackage{xstring}
\usepackage[normalem]{ulem}
\usepackage{cancel}
\usetikzlibrary{matrix,arrows,decorations.pathmorphing}
\usetikzlibrary{arrows.meta, positioning, shapes, calc}
\tikzset{commutative diagrams/.cd}
\usepackage{pst-node}

\usepackage[english]{babel}
\usepackage{amssymb,amsmath,amsfonts,bbm,pifont,upgreek,bbold,accents}  
\usepackage[colorlinks=true]{hyperref}
\hypersetup{urlcolor=blue, citecolor=red}
\usepackage{cancel}

\font\teneufm=eufm10
\font\seveneufm=eufm7
\font\fiveeufm=eufm5
\newfam\eufmfam
\textfont\eufmfam=\teneufm
\scriptfont\eufmfam=\seveneufm
\scriptscriptfont\eufmfam=\fiveeufm

\newcommand\beq[1]{\begin{equation}\label{#1} }
\newcommand{\eeq}{\end{equation} }

\newcommand\beqa[1]{\begin{eqnarray} \label{#1}}
\newcommand{\eeqa}{\end{eqnarray} }
\newcommand{\beqano}{\begin{eqnarray*} }
\newcommand{\eeqano}{\end{eqnarray*} }

\renewcommand{\theequation}{\arabic{section}.\arabic{equation}}

\newtheorem{theorem}{Theorem}
\newtheorem{example}[theorem]{Example}
\newtheorem{definition}[theorem]{Definition}
\newtheorem{proposition}[theorem]{Proposition}
\newtheorem{lemma}[theorem]{Lemma}
\newtheorem{sublemma}[theorem]{Sublemma}
\newtheorem{remark}[theorem]{Remark}
\newtheorem{notationalremark}[theorem]{Notational Remark}
\newtheorem{corollary}[theorem]{Corollary}
\newtheorem{assumption}[theorem]{Assumption}
\newtheorem{claim}[theorem]{Claim}
\newtheorem{conjecture}[theorem]{Conjecture}

\newtheorem{tools}{$\negsp\negsp$}[subsection]

\newcommand\thm[1]{\begin{theorem}\label{#1}}
\newcommand\thmtwo[2]{\begin{theorem}[#1]\label{#2}}
\newcommand\ethm{\end{theorem} }
\newcommand\dfn[1]{\begin{definition}\label{#1} \rm}
\newcommand\dfntwo[2]{\begin{definition}[#1]\label{#2} \rm}
\newcommand\edfn{\end{definition} }
\newcommand\pro[1]{\begin{proposition}\label{#1}}
\newcommand\protwo[2]{\begin{proposition}[#1]\label{#2}}
\newcommand\epro{\end{proposition} }
\newcommand\lem[1]{\begin{lemma}\label{#1}}
\newcommand\lemtwo[2]{\begin{lemma}[#1]\label{#2}}
\newcommand\elem{\end{lemma} }
\newcommand\sublem[1]{\begin{sublemma}\label{#1}}
\newcommand\sublemtwo[2]{\begin{sublemma}[#1]\label{#2}}
\newcommand\esublem{\end{sublemma} }
\newcommand\rem[1]{\begin{remark}\label{#1} \rm}
\newcommand\erem{\end{remark} }
\newcommand\notrem[1]{\begin{notationalremark}\label{#1} \rm}
\newcommand\enotrem{\end{notationalremark} }
\newcommand\cor[1]{\begin{corollary}\label{#1}}
\newcommand\cortwo[2]{\begin{corollary}[#1]\label{#2}}
\newcommand\ecor{\end{corollary} }
\newcommand\asmp[1]{\begin{assumption}\label{#1}}
\newcommand\asmptwo[2]{\begin{assumption}[#1]\label{#2}}
\newcommand\easmp{\end{assumption} }
\newcommand\clm[1]{\begin{claim}\label{#1}}
\newcommand\eclm{\end{claim} }

\expandafter\chardef\csname pre amssym.def
at\endcsname=\the\catcode`\@
\catcode`\@=11
\def\undefine#1{\let#1\undefined}
\def\newsymbol#1#2#3#4#5{\let\next@\relax
 \ifnum#2=\@ne\let\next@\msafam@\else
 \ifnum#2=\tw@\let\next@\msbfam@\fi\fi
 \mathchardef#1="#3\next@#4#5}
\def\mathhexbox@#1#2#3{\relax
 \ifmmode\mathpalette{}{\m@th\mathchar"#1#2#3}%
 \else\leavevmode\hbox{$\m@th\mathchar"#1#2#3$}\fi}
\def\hexnumber@#1{\ifcase#1 0\or 1\or 2\or 3\or 4\or 5\or 6\or 7\or
8\or
 9\or A\or B\or C\or D\or E\or F\fi}
\ifcase\@ptsize
 \font\tenmsb=msbm10
 \font\sevenmsb=msbm7
 \font\fivemsb=msbm5
\or
 \font\tenmsb=msbm10 scaled \magstephalf
 \font\sevenmsb=msbm7 scaled \magstephalf
 \font\fivemsb=msbm5  scaled \magstephalf
\or
 \font\tenmsb=msbm10 scaled \magstep1
 \font\sevenmsb=msbm7 scaled \magstep1
 \font\fivemsb=msbm5 scaled \magstep1
\fi
\newfam\msbfam
\textfont\msbfam=\tenmsb
\scriptfont\msbfam=\sevenmsb
\scriptscriptfont\msbfam=\fivemsb
\edef\msbfam@{\hexnumber@\msbfam}
\def\Bbb#1{\fam\msbfam\relax#1}
\def\widehat#1{\setboxz@h{$\m@th#1$}%
 \ifdim\wdz@>\tw@ em\mathaccent"0\msbfam@5B{#1}%
 \else\mathaccent"0362{#1}\fi}
\def\widetilde#1{\setboxz@h{$\m@th#1$}%
 \ifdim\wdz@>\tw@ em\mathaccent"0\msbfam@5D{#1}%
 \else\mathaccent"0365{#1}\fi}

\def\RIfM@{\relax\ifmmode}
\def\nonmatherr@#1{\errmessage{\string#1\space allowed only in math mode}}
\def\Bbb{\RIfM@\expandafter\Bbb@\else
 \expandafter\nonmatherr@\expandafter\Bbb\fi}
\def\Bbb@#1{{\Bbb@@{#1}}}
\def\Bbb@@#1{\fam\msbfam\relax#1}
\def\setboxz@h{\setbox\z@\hbox}
\def\wdz@{\wd\z@}
\catcode`\@=\csname pre amssym.def at\endcsname
\newcommand{\negsp}{\hspace{-.09truecm}}  

\newcommand{\integer}{{\Bbb Z}   }
\newcommand{\complex}{{\Bbb C}   }
\newcommand{\rational}{{\Bbb Q}  }

\renewcommand{\Im}{{\, \rm Im\, }}

\newcommand{{\cH}}{{\cal H} }

\newcommand{\Aut}{\rm Aut}
\newcommand{\Pic}{\rm Pic}
\newcommand{\Num}{\rm Num}
\newcommand{\Div}{\rm Div}
\newcommand{\Rat}{\rm Rat}
\newcommand{\Deg}{\rm deg}
\newcommand{\Gr}{\rm Gr}
\newcommand{\Dim}{\rm dim}
\newcommand{\Ann}{\rm Ann}
\newcommand{\Ker}{\rm Ker}
\newcommand{\Det}{\rm det}
\newcommand{\Fix}{\rm Fix}
\begin{document}

\begin{titlepage}
    \begin{center}
    \includegraphics[height=1.0cm]{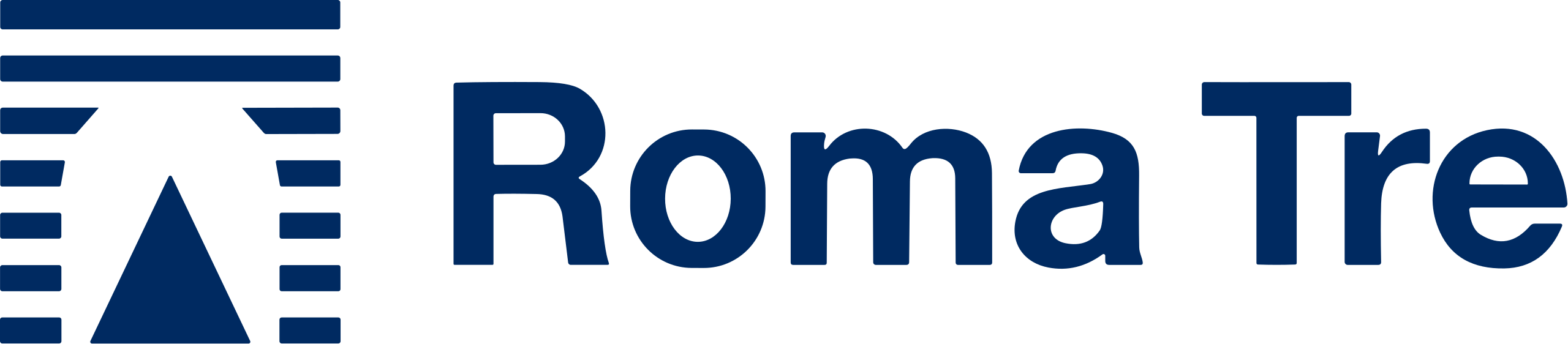}
        \vspace*{1.75cm}

        \large

        {CORSO DI DOTTORATO DI RICERCA IN MATEMATICA}

        \vspace*{1cm}
        
        {XXXVII CICLO DEL CORSO DI DOTTORATO}

         \vspace*{1cm}
         
         \LARGE
         
        \textbf{Coble surfaces: projective models and automorphisms with related topics}

        \normalsize
        \vspace*{4cm}

        \begin{minipage}[t]{0.47\textwidth}
	       {Relatore:} \vspace{0.3em} \\
              {\large \textbf{Prof. Alessandro Verra}} \vspace{1em}  \\
              {Coordinatore:} \vspace{0.3em} \\
              {\large \textbf{Prof. Alessandro Giuliani}}
        \end{minipage}
        \hfill
        \begin{minipage}[t]{0.47\textwidth}\raggedleft
	       {Dottorando:} \hspace{-0.9em} \vspace{0.3em} \\
              {\large \textbf{Federico Pieroni}} 
        \end{minipage}

        \vfill
        19 Novembre 2025
            
    \end{center}
\end{titlepage}


{\small
\tableofcontents}

\newpage
\section*{Abstract}
\addcontentsline{toc}{section}{Abstract}
Our goal will be the study of complex Coble surfaces. Originally, they were introduced by A.B. Coble in \cite{OriginalCoble1919} in $1919$, in the following way. Consider the complex projective plane $\mathbb{P}^2$, which we will define more precisely in the next Introduction. Inside $\mathbb{P}^2$, an irreducible curve of degree $d$ is the zero locus of an irreducible homogeneous polynomial $F = F(X_0, X_1, X_2)$ of degree $d$ in three complex variables. A singular point for such a curve is a point where all the partial derivatives $\frac{\partial F}{\partial X_0}, \frac{\partial F}{\partial X_1}, \frac{\partial F}{\partial X_2}$ simultaneously vanish. In the original definition, a Coble surface is obtained picking an irreducible curve $\overline C \subset \mathbb{P}^2$ of degree $6$ with $10$ singular points $p_1, \ldots, p_{10}$, and by performing what is called the blow - up of $\mathbb{P}^2$ at these points. The result is a smooth surface $X$, originally called a Coble surface. By construction, this surface possesses a holomorphic map $p : X \to \mathbb{P}^2$, which is almost everywhere an isomorphism. Indeed, the restriction on the open subsets $p : X \setminus p^{-1} \{p_1, \ldots, p_{10}\} \to \mathbb{P}^2 \setminus \{p_1, \ldots, p_{10}\}$ is invertible, so that we can think of $p^{-1}$ as a meromorphic inverse to $p$. This surface $X$ has a nice property, that is, the automorphism group $\Aut  (X)$ of holomorphic invertible transformations of $X$ in itself is infinite, see for example \cite{CantatDolg2012}, \cite{OriginalCoble1919}. For each of these transformations $f : X \to X$, one can consider the composition $p \circ f \circ p^{-1}$, which is a meromorphic map of $\mathbb{P}^2$ in itself, undefined at $p_1, \ldots, p_{10}$. Coble surfaces were introduced precisely with the aim to study subgroups of meromorphic transformations of $\mathbb{P}^2$ in itself, which was a classical topic in Algebraic Geometry in the early $20$-th century. Moreover, the fact that $\Aut  (X)$ is infinite is remarkable, since a conjecture of Coble himself, later proved by Hirschowitz in \cite{Hirschowitz1988}, stated that, if one starts with a set of $n \ge 9$ points $p_1, \ldots, p_n$ in $\mathbb{P}^2$, which do not satisfy any nontrivial algebraic constraint, then the blow up of $\mathbb{P}^2$ at these points admits only the identity automorphism. However, Coble surfaces do not contradict Hirschowitz theorem. Indeed, fixed a $10$-ple $(p_1, \ldots, p_{10})$, we are asking that there exists a homogeneous polynomial $F(X_0, X_1, X_2)$ of degree $6$ with $\frac{\partial F}{\partial X_i} (p_j) = 0$. One can easily show that this is actually a non trivial algebraic condition, which is not automatically satisfied by all the $10$-ples $p_1, \ldots, p_{10}$. Those which actually satisfy this condition can be used to give rise to a Coble surface.\\
Nowadays, the definition of a Coble surface has been modified in the following way. To any surface $X$, one can associate an abelian group, called the Picard group $\Pic (X)$ of $X$. By definition, a divisor $D$ on $X$ is a formal linear combination $D = n_1 C_1 +\cdots + n_r C_r$, where $C_i \subset X$ are curves, and the coefficients $n_i \in \integer$ are integers. The set of divisors is an abelian group, with the addition operation. A divisor is said effective if $n_i \ge 0$ for all indices $i$. We want to be able to deform divisors, so we consider equivalent two divisors $D_1, D_2$ if there exists a $1$-parameter deformation family $X_{\lambda} \subset X$, with $\lambda \in \mathbb{P}^1 = \complex \cup \{\infty\}$ which starts with $X_0 = D_1$ and it terminates with $X_{\infty} = D_2$. This equivalence relation is called rational equivalence, and the Picard group $\Pic (X)$ is the group of divisors modulo rational equivalence. A crucial role in the Picard group is played by canonical divisors. A divisor $D = \sum n_i C_i - \sum_j m_j C_j'$ is canonical if there exists a meromorphic differential form $\omega$ of degree $2$, with zeroes of order $n_i$ along $C_i$ and poles of order $m_j$ along $C_j'$. All canonical divisors lie in the same class of rational equivalence, denoted by $K_X$, the canonical class of $X$. This class contains effective divisors if and only if $X$ admits holomorphic $2$-differential forms.\\
The modern definition for a Coble surface is stated purely in terms of the canonical divisor class $K_X$. Indeed, we require three conditions to a Coble surface $X$:\\
i) $X$ must be rational, that is, there must exist open subsets $U \subset X, V \subset \mathbb{P}^2$ which are isomorphic.\\
ii) The anti - canonical divisor class $- K_X$ does not contain any effective divisor.\\
iii) The anti - bicanonical divisor class $- 2 K_X$ contains a unique effective divisor $C = C_1 + \cdots + C_n$, called the Coble curve of $X$.\\
The reason of this definition is that it closely resembles the definition of Enriques surfaces. This is a very well known family of surfaces, first introduced in the late $19$-th century by Enriques. For an Enriques surface, the requirement $i)$ on rationality is dropped, while conditions $ii), iii)$ are required respectively for the divisors $K_X, 2 K_X$, rather than their opposites.\\
Going back on Coble surfaces, the original definition given by Coble corresponds exactly to the irreducibility of the curve $C$ introduced in point iii).\\
In the first part of this work, we will discuss general properties of Coble surfaces, with a special focus on the case when $C$ is irreducible. We will underline many analogies between the behaviour of these surfaces and of Enriques surfaces. As an example, a very useful tool to study Enriques surfaces are isotropic sequences, see for example Knutsen \cite{Knutsen2020}. By definition, these are sequences of isolated elliptic curves ${\cal E}_1, \ldots, {\cal E}_r$ with intersection product ${\cal E}_i {\cal E}_j = 1 - \delta_{i, j}$. On a Coble surface, we will show that any isotropic sequence ${\cal E}_1, \ldots, {\cal E}_r$ of length $r \le 8$ can be extended to a maximal sequence ${\cal E}_1, \ldots, {\cal E}_{10}$. Going on, we can study lots of models for projective Coble surfaces. Among these, we cite the Bordiga - Coble model, which will be a special Bordiga surface which is also a Coble surface. Another interesting model, which underlines the link between Enriques surfaces and Coble surfaces, will be the case of Coble quintics in $\mathbb{P}^3$. These will be quintic surfaces containing a tetrahedron, and they will be nodal along three concurrent edges. We will show that the desingularization of such a surface is actually a Coble surface in the classical sense, and vice versa, any Coble surface admits such a polarization $|H|$. The similarities will become heavier once we will pass to the anti - adjoint linear system $|H - K_X|$. Using this we will come back to the classical construction of an Enriques sextic, which is nodal along all the six edges, with the extra condition that one of the vertices will be a $4th$-ple point.\\
The following sections will be dedicated to study of automorphisms on Coble surface, with a detailed attention to those with irreducible Coble curve $C$. For those surfaces, indeed, there exist a well - defined restriction homomorphism $\rho : \Aut  (X) \to \Aut  (C) \simeq PGL (2)$, and there exists an open conjecture about who $ker (\rho), Im (\rho)$ might
 be. We will show the construction made by Pompilj in $1937$, in an attempt to prove $ker (\rho) \ne 1$, and we will follow Coble counter - proof, which proved how this happens only for some families of Coble surfaces. Going on, we will focus on the case of involutions, proving that on an unnodal Coble surface $X$ with irreducible boundary $C$, any involution $i: X \to X$ is the lift of a Bertini involution. We will conclude this work by defining the coincidence loci for families of involutions, showing that they are always $2$-codimensional inside the Severi variety of $10$-nodal plane sextics.

\renewcommand{\theequation}{\arabic{equation}}
 
\newpage
\section*{Introduction}
\addcontentsline{toc}{section}{Introduction}
\paragraph*{General setup:}
The main topic of Algebraic Geometry is the study of algebraic varieties. By definition, one considers the complex projective space $\mathbb{P}^n (\complex) = \mathbb{P}^n$, which is the quotient $$\mathbb{P}^n := (\complex^{n + 1} \setminus \{0\}) / \simeq$$ where the equivalence relation $\simeq$ between nonzero vectors $v, w \in \complex^{n + 1} \setminus \{0\}$ is given by $$v \simeq w \quad {\rm\, if\, and\, only\, if\,} \quad w = \lambda v$$ for some scalar $\lambda \ne 0$. If $(X_0, \ldots, X_n) \in \complex^{n + 1}$ is any vector different from zero, its equivalence class with respect to $\simeq$ will be denoted as $[X_0, \ldots, X_n]$. Over the space $\mathbb{P}^n$, one puts the Zariski topology, where the closed subsets have the form $$V (F_1, \ldots, F_r) := \{x \in \mathbb{P}^n \quad {\rm\, such\, that\,} \quad F_1 (x) = \cdots = F_r (x) = 0\} \subset \mathbb{P}^n$$ for a finite collection of homogeneous polynomials $F_1, \ldots, F_r \in \complex[X_0, \ldots, X_n]$. The Zariski - closed subsets of $\mathbb{P}^n$ are called algebraic varieties. If $X \subset \mathbb{P}^n, Y \subset \mathbb{P}^m$ are varieties, a regular morphism $f : X \to Y$ will be a holomorphic function which locally has a polynomial structure, that is, it can be written as $$f (x) := [H_0 (x), \ldots, H_m (x)]$$ for a $(m + 1)$-ple of homogeneous polynomials $H_0, \ldots, H_m \in \complex[X_0, \ldots, X_n]$ with the same degree and with no common zeroes. We will consider mainly curves and surfaces, which are varieties of dimension $1$ and $2$ respectively. The simplest examples are given respectively by the projective line $\mathbb{P}^1$ and the projective plane $\mathbb{P}^2$.\\
The closest examples are given by rational varieties. An $n$-dimensional variety $X$ is said rational if there exist open subsets $U \subset X, V \subset \mathbb{P}^n$ which are isomorphic via regular functions $f : U \to V, g : V \to U$. Roughly speaking, this means that $X$ admits a meromorphic polynomial parametrization from $\mathbb{P}^n$, which is almost everywhere an isomorphism. In dimension $1$, one can show that the only rational smooth curve is actually $\mathbb{P}^1$ itself, see for example \cite{Shafarevich2013}. In dimension $2$, this is no longer true: there are smooth rational surfaces which are not isomorphic to $\mathbb{P}^2$. During the $20$-th century, Castelnuovo, Enriques and later Kodaira provided a classification of all families of complex surfaces. A key ingredient for their work is the notion of vector bundle. The idea is the following: given a smooth complex variety $X$, one can attach to any point $x \in X$ a $\complex$-vector space $V_x$, whose rank $r$ does not vary as $x$ moves in $X$. The union $V := \bigcup_{x \in X} V_x$ is still a variety, equipped with a natural surjective morphism $p : V \to X$, $p (v) := x$ if $v \in V_x$. As an example, one can simply consider $X \times \complex^r$, which attaches to any $x \in X$ the ``constant'' vector space $\complex^r$. Other natural constructions are given by the tangent vector bundle ${\cal T}_X \to X$, which is given by the glueing of the tangent spaces $T_x X, x \in X$, or its dual cotangent bundle $\Omega_X^1 := ({\cal T}_X)^*$, or its external $p$-th power $\Omega_X^p := \bigwedge^p \Omega_X^1$, for $1 \le p \le {\rm\, dim\,} X$. For any vector bundle $p : V \to X$, one can look for its sections, that is, regular functions $\sigma : X \to V$ such that $\sigma (x) \in V_x$ for any $x \in X$. The set of such sections is a vector space itself, denoted by $H^0 (X, V)$. In the examples above, a section of $X \times \complex^r \to X$ is just a $r$-ple of regular functions $X \to \complex$, and by a compactness argument, each of these must be a constant, so that $H^0 (X, \complex^r) = \complex^r$. Sections of the tangent bundle ${\cal T}_X$ are regular vector fields on $X$, and they may exist or not. A section of $\Omega_X^p$ is a holomorphic differential form of degree $p$.\\
A special role is played by line bundles, which are vector bundles of rank $1$. The reason is the following: given two line bundles $p_1 : L_1 \to X, p_2 : L_2 \to X$, one can construct the tensor line bundle $L_1 \otimes L_2$, which attaches to any point $x \in X$ the space $(L_1 \otimes L_2)_x := (L_1)_x \otimes (L_2)_x$, which has still rank $1$. 
The trivial line bundle $X \times \complex$ plays the role of zero element with respect to this operation, and any line bundle $L \to X$ satisfies admits an inverse element, namely the dual bundle $L^* \to X$, which satisfies $L \otimes L^* \simeq X \times \complex$. Hence the set of line bundles over a fixed variety $X$, equipped with the tensor product, has a structure of an abelian group. From now on, a line bundle $L \to X$ will be denoted by ${\cal O}_X (L)$, so that the tensor product between ${\cal O}_X (L_1), {\cal O}_X (L_2)$ will become simply ${\cal O}_X (L_1 + L_2)$, using the additive notation which comes from the commutativity of tensor product. The dual of ${\cal O}_X (L)$ will be ${\cal O}_X (- L)$, and the trivial line bundle $X \times \complex$ will be just ${\cal O}_X$. For any variety $X$, the top wedge power $\Omega_X^{{\rm\, dim\,} X}$ is called the canonical line bundle, and it is denoted as ${\cal O}_X (K_X)$. The abelian group of line bundles on any variety $X$ is called the Picard group, and it is denoted as $\Pic (X)$.\\
In the Castelnuovo - Enriques - Kodaira classification of surfaces, the canonical bundle ${\cal O}_X (K_X) = \Omega_X^2$ of a smooth surface $X$ plays a central role. As an example, Castelnuovo showed that the rationality of $X$ can be stated just in terms of ${\cal O}_X, \Omega_X^2$.\\
The aim of this Ph.D. Thesis is the study of linear systems on Coble surfaces. Historically, they were introduced for the following reason: a classical problem in Algebraic Geometry in the early $20$-th century was the study of birational transformations of the plane $\mathbb{P}^2$, which are meromorphic transformations of $\mathbb{P}^2$ in itself, with meromorphic inverse. The group of such transformations is denoted by $Cr (2)$ or $Cr (\mathbb{P}^2)$, the Cremona group in dimension $2$. Since this is a very large group, an idea was to look at some of its subgroups, in the following way. Starting from a finite set of points $p_1, \ldots, p_n$ in the plane $\mathbb{P}^2$, one can always build the blow up of $\mathbb{P}^2$ centered at $p_1, \ldots, p_n$. By construction, this will be a smooth algebraic surface $X = Bl_{p_1, \ldots, p_n} \mathbb{P}^2$, with a regular morphism $p : X \to \mathbb{P}^2$, with the property that $p^{-1} (p_i) \simeq \mathbb{P}^1$ for all the points $p_i$'s, and the restriction on the complementary open subsets $p: X \setminus p^{-1} \{p_1, \ldots, p_n\} \to \mathbb{P}^2 \setminus \{p_1, \ldots, p_n\}$ is invertible. Once we have our blow up $X$, any regular invertible automorphism $f : X \to X$ will induce a meromorphic transformation on $\mathbb{P}^2$, by considering the composition $\phi := p \circ f \circ p^{-1} : \mathbb{P}^2 \dasharrow \mathbb{P}^2$. The discontinuous arrow in this notation underlines that the composition is defined only as a meromorphic function, since the involved function $p^{-1}: \mathbb{P}^2 \dasharrow X$ exists only over the open subset $\mathbb{P}^2 \setminus \{p_1, \ldots, p_n\}$. In this way, the group $\Aut  (X)$ of regular transformation of $X$ in itself is identified with the subgroup of $Cr (\mathbb{P}^2)$ of meromorphic transformation which have indeterminacies at most at prescribed points $p_1, \ldots, p_n \in \mathbb{P}^2$. There is not too much restriction in this, since one can show that the indeterminacy locus of a meromorphic transformation $\phi: \mathbb{P}^2 \dasharrow \mathbb{P}^2$ is always a finite set, see for example \cite{Shafarevich2013}, so you can always consider its blow up to resolve the indeterminacies of $\phi$. As the finite set $\{p_1, \ldots p_n\}$ changes, the blown up surface $X$ changes too, so it has different automorphism group $\Aut  (X)$, and we will be able to describe different subgroups of $Cr (\mathbb{P}^2)$.\\
However, this strategy has some limits: if one picks too many points $p_1, \ldots, p_n$, the surface $X$ may admit just the identity $\mathbb{1}_X$ as automorphism into itself, so the induced subgroup in $Cr (\mathbb{P}^2)$ would be trivial. In the early $20$-th century this was just a conjecture, but Hirschowitz actually showed in \cite{Hirschowitz1988} that for almost every choice of $n \ge 9$ points in the plane, the blow up $X = Bl_{p_1, \ldots, p_n} \mathbb{P}^2$ does not admit any automorphism. Coble surfaces proved to be an exception to this fact. Indeed, they were originally defined by A. B. Coble in \cite{OriginalCoble1919} as a blow up of $\mathbb{P}^2$ in $10$ points, which must be singular points for a curve $\overline C \subset \mathbb{P}^2$ of degree $6$, and nonetheless they admit nontrivial automorphisms.\\
Nowadays, there is another definition for Coble surfaces, purely expressed in terms of the previously mentioned line bundle ${\cal O}_X (K_X)$, and it extends the original one given by Coble. The new definition, which we will se later, underlines how Coble surface provide a bridge between the world of rational surfaces, where they belong, and the world of Enriques surfaces. This is a class of non - rational surfaces known since the late $19$-th century, and it is object of an extremely rich literature, see for example \cite{Alexeev2023}, \cite{Ciliberto2023}, \cite{VerraReye1993}, \cite{Cossec1985}, \cite{EnriquesI2025}, \cite{Cossec1983}, \cite{Moduli2013}, \cite{EnriquesII2025}, \cite{ABriefIntro2016}, \cite{Knutsen2020}, \cite{Gebhard2024}, \cite{Schaffler2024}, \cite{Reye1882}, \cite{Verra2023}, \cite{Vinberg1983}.
\paragraph*{Overview on the work:}
The main results we will show are the following: in an unnodal Coble surface $X$ with irreducible boundary $C$, any isotropic sequence ${\cal E}_1, \ldots, {\cal E}_r$ with $r \le 8$ can be extended to a maximal isotropic sequence ${\cal E}_1, \ldots, {\cal E}_{10}$ (Theorem \ref{ndindex}, page 45). Under the same assumption, there are no involutions $i : X \to X$ which fix the curve $C$ (Proposition \ref{nifc}, page 69), and all the involutions of $\Aut  (X)$ are lifts of Bertini involutions (Theorem \ref{solobert}, page 70).\\ 
In Section $1$, we will give a brief introduction on the general setting, showing the tools we will use in all the work. A deep attention will be put upon the theory of divisors, remembering the construction of the Picard group $\Pic (X)$ for any smooth surface $X$. This group comes together a bilinear symmetric pairing $\Pic (X) \times \Pic (X) \to \integer$, which is just an application to surfaces of the much larger intersection theory on any kind of smooth variety. A very heavy role will be played by the canonical class divisor $K_X$. We will give remarks on some classical Theorems, which use this intersection product to establish some criteria for divisors to be effective, such as Riemann - Roch Theorem, or ample divisors, like the Mukai - Moishezon Criterion or Reider Theorem.\\
Following this way, we will briefly remember some concrete application of the study of $\Pic (X)$ to rational surfaces. Among rational surfaces, a special attention will be dedicated to Del Pezzo surfaces, and the definition of some automorphisms on them, like those induced by De - Jonquieres, Geiser or Bertini involutions of $\mathbb{P}^2$. We will also give some basic information on the construction of the Hirzebruch surfaces $\mathbb{F}_n$, since they will be needed in some parts of this Thesis.\\
Next part will be dedicated to giving some pre-requisites on Enriques surfaces. These were first introduced during the birational classification work by Enriques and Castelnuovo, as a negative example to the following question: is the rationality of a surface $X$ an equivalent condition to the cohomological properties $h^1 ({\cal O}_X) = h^0 ({\cal O}_X (K_X)) = 0$ ? It is well - known that this question has negative answer, since these cohomological conditions are satisfied by Enriques surfaces, which yet fail to be rational. We will briefly remember their birational model, which were originally discovered by Enriques. Namely he proved that sextic surfaces in $\mathbb{P}^3$ with double points along the six edges of a tetrahedron are birational models of what we call nowadays as Enriques surfaces. Other needed tools will be the Picard group $\Pic (X)$ and the numerical class group $\Num (X)$ of a general Enriques surface $X$. This will involve the definition of the lattice $\mathbb{E}_{10}$, which will show some analogies with main object of the thesis, the Coble surfaces.\\
\\
In Section $2$ we will finally give the definition of a Coble surface: it is a complex smooth rational surface $X$ which satisfies $h^0 ({\cal O}_X (- K_X)) = 0$ and $h^0 ({\cal O}_X (- 2 K_X)) = 1$. For us, the Coble curve, or boundary curve, of a Coble surface $X$ will be the unique effective divisor $C \in |- 2 K_X|$, with irreducible decomposition $C = C_1 + \cdots + C_n$, with $C_i \cap C_j = \emptyset$ for any $i \ne j$. It is well known that each of these surfaces comes from blowing up the projective plane $\mathbb{P}^2$ at a suitable number of points, placed in non-general position. We will provide examples of Coble surfaces with reducible boundary, and it was proved by Cossec, Dolgachev and Liedtke in \cite{EnriquesI2025} that $n \le 10$ always.\\
Most of the work of this Thesis will be carried over Coble surfaces $X$ with irreducible boundary, that is, $n = 1$. In this case, the birational description of $X$ becomes much easier, since every such surface is obtained via the blow up of $10$ points in $\mathbb{P}^2$ which are nodes for a reduced irreducible plane sextic curve $\overline C$, and conversely, every such blow up is a Coble surface. In this case, the strict transform of $\overline C$ becomes the Coble curve $C$. We will underline some analogies between Enriques surfaces and Coble surfaces with irreducible boundary. As an example, at a level of Picard group, the former satisfy $\Pic (X) = \mathbb{E}_{10} \oplus \integer_2 K_X$, while the latter satisfy $\Pic (X) = \mathbb{E}_{10} \oplus \integer K_X$.\\
Another nice analogy we will find is the following: it is well known that, for an Enriques surface $X$, you can build an isotropic sequence, that is, a sequence ${\cal E}_1, \ldots, {\cal E}_r$, where ${\cal E}_i \in \Pic (X)$ are the classes of isolated elliptic curves with intersection products ${\cal E}_i {\cal E}_j = 1 - \delta_{i, j}$, and that the maximum achievable value for the length $r$ is $r \le 10$. On a Coble surface $X$ we will show that this upper bound is always achieved, namely, any isotropic sequence ${\cal E}_1, \ldots, {\cal E}_r$ with $r \le 8$ can be extended to an isotropic sequence ${\cal E}_1, \ldots, {\cal E}_{10}$. A fundamental help in the proof will come from the fact that Coble curves always admit $(-1)$-curves, which is absolutely forbidden in Enriques surface.\\
We will conclude Section $2$ with a brief glimpse on why it it reasonable to think that the moduli space of Coble curves with irreducible boundary has dimension $9$. To do this, we will observe that the rationality of the sextic Coble curve $\overline C$ provides a morphism $\mathbb{P}^1 \to \mathbb{P}^2$ defined by a net in $H^0 ({\cal O}_{\mathbb{P}^1} (6))$, and hence a suitable parameter space for Coble curves is given by the Grassmannian $G = Gr (3, H^0 ({\cal O}_{\mathbb{P}^1} (6)))$, which is $12$-dimensional. Then we will use the GIT-quotient $G // PGL (2)$.\\
\\
Section $3$ will be dedicated to the study of projective models of Coble surfaces with irreducible boundary curve $C \in |- 2 K_X|$. The first one will be the Bordiga-Coble model, since it is a Bordiga surface, which is also a Coble surface, since its anti - bicanonical class contains a smooth rational curve, of degree $4$ and self intersection $-4$. Actually, Bordiga-Coble surfaces are special Bordiga surfaces. Indeed, a general Coble surface is embedded in $\mathbb{P}^4$ via the linear system of quartic curves passing through the $10$ base points, and the image is a Bordiga surface, as mentioned.\\
Another interesting case is given by case of quintic Coble surfaces: by definition, these surfaces are the normalization of a quintic surface $\overline X \subset \mathbb{P}^3$ of the form \begin{eqnarray*}\overline X &:=& \{[X_0, X_1, X_2, X_3] \in \mathbb{P}^3 \quad {\rm\, such\, that\,} \nonumber\\
&&\alpha X_0 X_2^2 X_3^2 + \beta X_0 X_1^2 X_3^2 + \gamma X_0 X_1^2 X_2^2 + X_1 X_2 X_3 q = 0\} \end{eqnarray*} where $\alpha, \beta, \gamma$ are nonzero constants, and $q$ is a quadric in $\mathbb{P}^3$. The similarities with the Enriques sextic are evident. We will prove that on any Coble surface there is always a big polarization $H$ with $H^2 = 5$ which realizes such a representation, and conversely, the normalization of any such quintic surface is actually a Coble surface. The similarities with the Enriques case become even stronger if we pass to the anti - adjoint linear system $|H - K_X|$, which realizes $X$ as an actual Enriques sextic in $\mathbb{P}^3$, with the additive condition that it has a $4$-ple singular point at a vertex of the coordinate tetrahedron.\\
We go on, introducing two other models for Coble surfaces, both in $\mathbb{P}^3$. The first one has degree $3$, and it will be interesting since it will be the first concrete example of a nodal Coble surface, that is, a Coble surface containing a $(-2)$-curve. The second one has degree $4$ and contains a double line, and we consider it as a $K3$-Coble surface, since it can be thought as a limit of smooth quartics in $\mathbb{P}^3$, which are special cases of $K3$ surfaces.\\
\\
From Section $4$ on, we start to consider biregular automorphisms of Coble surfaces. The reason why we do this is the following: any automorphism $f : X \to X$ from a Coble surface $X$ to itself must preserve the canonical divisor $K_X$, that is, $f_* (K_X) = K_X$ at a level of divisors on $X$. Multiplying by $-2$, the consequence is that $f (C) = C$, that is, $C$ is set - theoretically preserved by $f$. When $C$ is irreducible, this fact becomes strongly relevant, since it provides a restriction morphism between groups $\rho: \Aut  (X) \to \Aut  (C) \simeq PGL (2)$. Coble posed the question about which subgroups $Ker (\rho), Im (\rho)$ can be, and in particular, he conjectured that $\ker (\rho) = \mathbb{1}_X$ for a general Coble surface $X$.\\
We will reconstruct the attempt given by G. Pompilj in $1937$ in his article ``Sulle trasformazioni cremoniane del piano che posseggono una curva di punti uniti'' (see \cite{Pompilj1940}). His aim was to provide a counter - example to Coble conjecture, and to do so, he proceeded in the following way. He started with an irreducible curve $\overline C \subset \mathbb{P}^2$ of degree $6$, with $10$ nodes in points $p_1, \ldots, p_7, A, B, C \in Sing(\overline C)$. Then he considered the plane birational Bertini involution $i_A : \mathbb{P}^2 \dasharrow \mathbb{P}^2$, associated to the $8$-ple $p_1, \ldots, p_7, A$, and he observed that $i_A$ becomes regular when you pass to the blow - up $X$, which is our Coble surface. The same holds by symmetry for the Bertini involutions $i_B, i_C$. He then considered a triple of automorphisms $T_A, T_B, T_C : X \to X$, defined as $T_A := i_B \circ i_C$, $T_B := i_C \circ i_A$, $T_C := i_A \circ i_B$. The automorphism $T_A$ can also be defined by looking at the elliptic pencil $\pi_A : X \to \mathbb{P}^1$, whose fibers are the strict transforms of sextic curves with nodes at $p_1, \ldots, p_7, A$. Indeed, $T_A$ preserves each fiber of $\pi_A$, as it acts simply as the addition of a divisor of degree $0$ on each smooth fiber of geometric genus $1$. Again, the same holds for $T_B, T_C$ with respect to the elliptic fibrations $\pi_B, \pi_C$ similarly defined. Pompilj claimed that the composition $R := T_C \circ T_B \circ T_A = (i_A \circ i_B \circ i_C)^2$ acts as the identity on the curve $C \in |- 2 K_X|$. We will show the path followed by Coble in \cite{Coble1939}, where he proved that the condition ${(i_A \circ i_B \circ i_C)^2}|_C = \mathbb{1}_C$ is not automatically true, as Pompilj claimed, but it is a divisorial condition in a suitable parameter space.\\
\\
In the final Section $5$, we will focus our attention on the biregular involutions defined on a Coble $X$, still under the assumption that $C \in |-2 K_X|$ is irreducible. We start by a very simple observation: no involution $i$ on such a Coble surface $X$ can satisfy the Coble conjecture $i|_C = \mathbb{1}_C$. To prove this, we will use a result from Bayle - Beauville in \cite{Bayle2000}, which explicitly enlists all possible minimal pairs $(X, i)$, with $X$ a smooth rational surface and $i : X \to X$ a minimal involution, that is, an involution which is not induced by another pair $(X', i')$ by blowing up an $i'$-symmetric finite set. It will be a straight - forward proof to show that none of the minimal models $(X, i)$ admits $X$ to be a Coble surface. Together with the hypothesis that $i|_C = \mathbb{1}_C$, this will lead us to the existence of a $(-1)$-curve $E \subset X$ such that $i (E) = E$. But then we will invoke Castelnuovo's Contractibility Criterion, and a result from Dolgachev - Zhang in \cite{ZhangDolgachev2001} on smoothness of fixed loci to find a contradiction.\\
As we already mentioned, we will also prove that any involution $i : X \to X$ is a lift of a Bertini involution, under the assumption that the Coble curve $C$ is irreducible and the surface is unnodal. Again, we will invoke the same result from \cite{Bayle2000}, to  We will do this by a straight - forward procedure, excluding one by one all the cases provided by \cite{Bayle2000}. To do this, a crucial help will come from the quintic model in $\mathbb{P}^3$ we talked about in Section $3$. In all the steps of the proof, we will show that the hypothesis of irreducibility for $|- 2 K_X|$ cannot be removed, providing examples of how all the cases actually happens when $|- 2 K_X|$ is reducible. We will also point out that the converse is not true, showing that the Bertini involution admits lifts also on Coble surfaces with reducible boundary.\\
The final part of this Section will be dedicated to repeating the same reasoning, but for families. We will build a suitable quasi - projective parameter space $\tilde V$ of triples $\tilde V := \{(X, E, H)\}$, where $X$ is a Coble surface, $E \subset X$ is a $(-1)$-curve, and $|H| : X \to \mathbb{P}^2$ is a polarization of degree $1$ which contracts $E$. In other words, $|H|$ realizes $X$ as a blow up of $\mathbb{P}^2$ in $10$ points, and $E$ is a marked choice of one of the $10$ exceptional divisors. This space will come together with a family $\pi: {\cal Y} \to \tilde V$, a divisor ${\cal E} \subset {\cal Y}$, and a line bundle ${\cal H} \in \Pic ({\cal Y})$, such that the pre - image of a point $\pi^{-1} (X, E, H)$ is $X$ itself, the restriction ${\cal E}|_{\pi^{-1} (X, E, H)} = E$ and the polarization ${\cal H}|_{\pi^{-1} (X, E, H)}$ is equal to $H$. On the family ${\cal Y}$ we will talk about rationally determined automorphism ${\cal G} : {\cal Y} \to {\cal Y}$, to denote biregular tranformations which live in every $\Aut  (X)$, while $X$ moves in a family of Coble surfaces. The aim is to construct, for a rationally determined involution ${\cal I}: {\cal Y} \to {\cal Y}$, what is the coincidence locus of ${\cal I}$. By definition, this will be the set of triples $(X, E, H)$ such that $i_{C \cap E} = \mathbb{1}_{C \cap E}$, with $i := {\cal I}|_X$. The aim of this part will be to study the geometry of this loci; invoking the classification result from Section $4$, we will show that the coincidence loci are always $2$-codimensional in $\tilde V$.\\\\
The final Appendix contains some un - finished computations, showing an alternative strategy to classify involutions on unnodal Coble surfaces with irreducible Coble curve. These calculations were made before the proof of the Proposition \ref{nifc} and Theorem \ref{solobert} in Section $3$. The idea was to take any involution $i : X \to X$ such that $i|_C = \mathbb{1}_C$, an exceptional curve $E \subset X$, and look at the behaviour of the linear system $E + i (E)$. However, none of the proof of Section $3$ requires these computations, which until now remain suspended.

\newpage
\section{Generalities on surfaces}
We work on the complex field $\complex$. A surface (respectively, a curve) will be a Zariski - closed subvariety of some projective space $\mathbb{P}^N$, of dimension $2$ (respectively $1$). Unless otherwise specified, we will assume that surfaces and curves are smooth and irreducible.\\
For any smooth irreducible variety $X$, we will denote by $\Pic (X)$ the Picard group of $X$, that is, the abelian group of line bundles on $X$, with group operation given by the tensor product.
\subsection{Line bundles, divisors and intersection theory on surfaces}
We refer to \cite{3264AndAllThat2016} for the contents of this subsection.\\
If $X$ is smooth and irreducible, the group $\Div (X)$ of divisors on $X$ is the free abelian group generated by closed subvarieties of $X$ of codimension $1$. The subgroup $\Rat (X) \subset \Div (X)$ is the group of divisors of rational functions on $X$. The Chow group of $X$ in codimension $1$ is the quotient $CH^1 (X) = \Div (X) / \Rat (X)$.\\
We will repeatedly use the existence of an isomorphism $$CH^1 (X) \simeq \Pic (X).$$ With some ambiguity, a divisor $D$ will be both an element in $\Div (X)$ and its class in $CH^1 (X)$. The corresponding line bundle in $\Pic (X)$ will be denoted by ${\cal O}_X (D)$. The $i$-th cohomology space of the line bundle ${\cal O}_X (D)$ will be denoted by $H^i ({\cal O}_X (D))$, with the space $H^0 ({\cal O}_X (D))$ consisting of global regular sections (provided they exist) of ${\cal O}_X (D)$. The dimension ${\rm dim}_\complex H^i ({\cal O}_X (D))$ will be denoted by $h^i ({\cal O}_X (D))$.\\
We will denote by $K_X \in \Pic(X)$ the (isomorphism class of) the line bundle $\Omega_X^{dim\, X}$.\\
\begin{theorem}[Serre duality]
If $D$ is any divisor on a smooth variety $X$ of dimension $n$, then for any $k \in \{0, \ldots, n\}$ we have $$h^k ({\cal O}_X (D)) = h^{n - k} ({\cal O}_X (K_X - D)).$$
\end{theorem}
We will denote by $K_X \in \Pic (X)$ the (isomorphism class of) the line bundle $\Omega_X^{dim\, X}$.\\
If $X$ is a surface, intersection theory provides a bilinear symmetric pairing $CH^1 (X) \times CH^1(X) \to \integer$, which automatically induces a bilinear pairing $\Pic (X) \times \Pic(X) \to \integer$. For any pair of different irreducible curves $C_1, C_2$, the product $C_1 C_2$ is a non-negative integer number, as it counts the cardinality of $C_1 \cap C_2$ with the appropriate multiplicities. Conversely, the self-intersection $C^2$ of a curve $C$ can be any integer number.\\
The construction of $CH^1 (X)$ is functorial: if $f : X \to Y$ is a regular morphism between smooth irreducible surface, then it induces a pull-back $$f^* : CH^1 (Y) \to CH^1 (X)$$ and a push-forward $$f_* : CH^1 (X) \to CH^1 (Y).$$ The pull-back $f^*$ of a curve $C \in CH^1 (Y)$ is defined passing through the pull-back of the corresponding line bundle ${\cal O}_Y (C)$, while the push-forward $f_*$ of an effective irreducible curve $C \subset X$ is given as $f_* (C) = m C'$ if $C' \subset Y$ is the image of $C$, and ${\rm deg\,} (f: C \to C') = m$, with the convention that $m = 0$ if $f(C)$ is a point. This constructions respect the relation of linear equivalence, and satisfy $$f^* (C_1) C_2 = C_1 f_* (C_2)$$ for all $C_1 \in CH^1 (Y), C_2 \in CH^1 (X)$. Moreover, if $f$ is generically finite of degree $d$, then $$f^* (C_1) f^* (C_2) = d C_1 C_2$$ for any $C_1, C_2 \in CH^1 (Y)$.

\subsection{Projective models via line bundles}
For any divisor $D$, the set of effective divisors rationally equivalent to $X$ will be denoted by $|D|$. If $V \subset H^0 ({\cal O}_X (D))$ is a linear subspace of sections, it induces a rational map $f_{|V|}: X \dasharrow \mathbb{P} (V^*) \simeq \mathbb{P}^{dim\, V - 1}$ in the following way. The projective dual $\mathbb{P}(V^*)$ parametrizes hyperplanes of $V$. For any $x \in X$ define $$f_{|V|} (x) := \{\sigma \in V {\rm\, s.\,t.\,} \sigma (x) = 0\} \subset V.$$ This map is undefined at the base locus of $V$, which is made of common zeroes for all sections in $V$. Thus $f_{|V|}$ is a morphism if and only if $V$ is basepoint-free. If $\sigma_0, \ldots, \sigma_n$ is a base for $V$, the map $f_{|V|} : X \dasharrow \mathbb{P}^n$ is given by $$f_{|V|} (x) = [\sigma_0 (x), \ldots, \sigma_n (x)],$$ where the evaluations $\sigma_i (x)$ are computed with respect to a common local trivialization of ${\cal O}_X (D)$ around $x$. Two trivializations differ by the multiplication for a nowhere zero factor around $x$, hence the projective point $f_{|V|} (x)$ is well defined, and it depends only on the choice of the basis $\sigma_0, \ldots, \sigma_n$. This determines a family of different maps $f_{|V|}$, but all of them differ by the post-composition with an element in $\mathbb{P}GL(n + 1)$, which acts transitively on the set of basis of $V$. This writing also makes clear that the indeterminacy locus for $f_{|V|}$ is exactly the set of common zeroes of $V$, given by $\sigma_0 (x) = \cdots = \sigma_n (x) = 0$.\\
We denote by $f_{|D|}$ the map induced choosing $V$ as the complete linear system $V = H^0 ({\cal O}_X (D))$.\\
\begin{definition}
Let $X$ be any variety.\\
A line bundle ${\cal O}_X (D)$ is very ample if $f_{|D|}$ is a regular embedding.\\
A line bundle ${\cal O}_X (D)$ is ample if there exists a natural $n \ge 1$ such that $n D$ is very ample.\\
A line bundle ${\cal O}_X (D)$ is big if there exists a natural $n \ge 1$ such that the rational morphism $f_{|n D|}$ is birational on its image.\\
A divisor $D$ is nef (numerically effective) if $D C \ge 0$ for any curve $C \subset X$.
A divisor $D$ is numerically equivalent to $0$ if it is orthogonal to all $\Pic (X)$. In this case we write $D \simeq_{num} 0$.\\
We denote by $\Num (X)$ the quotient $\Num (X) := \Pic (X) / \simeq_{num}$.
\end{definition}
Note that there exist a natural surjective homomorphism $\Pic (X) \to \Num (X)$, which is an isomorphism if and only if there are no nontrivial divisors $D \simeq_{num} 0$. Moreover, $\Num (X)$ is always torsion-free, since a divisor $D$ is orthogonal to all $\Pic (X)$ if and only if all its multiples $m D$ are.
We will use repeatedly the following facts:
\begin{theorem}[Nakai-Moishezon Criterion]
On a surface $X$, a line bundle ${\cal O}_X (D)$ is ample if and only if $D^2 > 0$ and $D C > 0$ for any curve $C \subset X$.
\end{theorem}
\begin{theorem}[Hodge Index Theorem]\cite{TieLuo1990}\label{HIT}
If ${\cal O}_X (D)$ is a big line bundle on a smooth surface $X$, and $C \subset X$ is any divisor such that $D C = 0$, then $C^2 \le 0$ and equality holds if and only if $C \simeq_{num} 0$.
\end{theorem}
\begin{theorem}[Riemann-Roch Theorem on surfaces]
If $C$ is any divisor on a smooth surface $X$, then $$h^0 ({\cal O}_X (D)) - h^1({\cal O}_X (D)) + h^2 ({\cal O}_X (D)) = h^0 ({\cal O}_X) - h^1({\cal O}_X) + h^2 ({\cal O}_X) + \frac{C(C - K_X)}{2}$$
\end{theorem}
\begin{theorem}[Reider's Theorem]\cite{Reider1988}
Let $X$ be a smooth surface, and let $H \in \Pic (X)$ be a nef divisor.\\
{\rm Part\,} I): If $H^2 > 4$, and the adjoint linear system $H + K_X$ has a base point $p \in X$, then there exists an effective divisor $D$ containing $p$, such that one of the following statements is true:\\
i) $H D = 0$ and $D^2 = -1$.\\
ii) $H D = 1$ and $D^2 = 0$.\\
{\rm Part\,} II) If $H^2 > 8$ and the adjoint linear system $H + K_X$ does not separate two points $p, q \in X$, then there exists an effective divisor $D$ containing both $p, q$, such that one of the following statements is true:\\
i) $H D = 0$ and $D^2 \in \{-1, -2\}$.\\
ii) $H D = 1$ and $D^2 \in \{0, -1\}$.\\
iii) $H D = 2$ and $D^2 = 0$.\\
iv) $H = 3 D$ and $D^2 = 1$.
\end{theorem}
\begin{proposition}
If $D \in \Pic (X)$ is an effective divisor, and $V \subset H^0 ({\cal O}_X (D))$ is a subsystem with no base locus, then the map $f := f_{|V|} : X \to \mathbb{P}^{{\rm dim\,} V - 1}$ is a regular morphism. In this case we have $$D^2 = ({\rm deg\,} f|_X : X \to f(X))({\rm deg\,} f(X))$$ and for any irreducible, reduced, effective divisor $R$, we will use $$D R = ({\rm deg\,} f|_R : R \to f(R))({\rm deg\,} f(R)).$$ These two formulas remain true even when ${\rm dim\,} f(R) < R$, adopting the convention that ${\rm deg\,} f(R) = 0$ in this case.
\end{proposition}
\begin{definition}\label{nodality}
For any $k > 0$, a $(- k)$-curve on $X$ will be an irreducible smooth curve $C \subset X$ such that $C^2 = - k$.\\
We say that a smooth surface $X$ is nodal if it contains a $(-2)$-curve.
\end{definition}
The following fact is due to Castelnuovo, and we will use it repeatedly:
\begin{theorem}
A $(-1)$-curve on a smooth surface $X$ can be contracted. In other words, there exists a smooth surface $X'$, a regular birational morphism $\pi: X \to X'$, and a point $x' \in X'$ such that $\pi (C) = x'$ and $\pi : X - C \to X' - \{x'\}$ is an isomorphism.
\end{theorem}
\subsection{Rational surfaces}
A key role will be played by rational surfaces. 
\begin{definition}
A surface $X$ is rational if there exists an open subset $U \subset X$ and an open subset $V \subset \mathbb{P}^2$ such that $U \simeq V$.
\end{definition}
This geometric property has a purely cohomological description, thanks to the following Criterion:
\begin{theorem}[Castelnuovo Criterion]
A surface $X$ is rational if and only if $H^1 ({\cal O}_X) = H^0 ({\cal O}_X (K_X)) = H^0 ({\cal O}_X (2 K_X)) = 0$. 
\end{theorem}
There are two types of rational surfaces: those which do not contain $(-1)$-curves, and those which do. A rational surface without $(-1)$-curves is a minimal rational surface.\\
The simplest example is the projective plane $\mathbb{P}^2$. The Picard group is $\Pic (X) = \integer L$, where $L$ is the class of any line. The intersection product is defined by $L^2 = 1$, and the canonical class divisor is $K_{\mathbb{P}^2} = - 3 L$. For any $d > 0$, the class of a curve of degree $d$ is $d L$.\\
Another example of minimal rational surface is the product $\mathbb{P}^1 \times \mathbb{P}^1$. The fibers $F_1, F_2$ of the two fibrations generate the Picard group $\Pic (\mathbb{P}^1 \times \mathbb{P}^1) = \integer F_1 \oplus \integer F_2$, and the intersection product is given by $F_1^2 = F_2^2 = 0$, $F_1 F_2 = 1$. The canonical class divisor is $K_{\mathbb{P}^1 \times \mathbb{P}^1} = -2 F_1 - 2 F_2$. For any pair of naturals $a, b \ge 0$, a curve of type $(a, b)$ will be a divisor in the linear system ${\cal O}_{\mathbb{P}^1 \times \mathbb{P}^1} (a, b)$, defined by a bi-homogeneous equation of bi-degree $(a, b)$.
\subsubsection{Hirzebruch surfaces}
The remaining minimal rational surfaces are the Hirzebruch surfaces, $\mathbb{F}_n$. They are defined for any $n \ge 0$ as the projectified rank-$2$ vector bundle ${\cal O}_{\mathbb{P}^1} \oplus {\cal O}_{\mathbb{P}^1} (- n)$ over $\mathbb{P}^1$. We will denote by $F$ the class of a fiber, and by $C_{- n}$ the unique negative section, given by the projectified sub-bundle $\mathbb{P} ({\cal O}_{\mathbb{P}^1}) \subset \mathbb{F}_n$. The classes $C_{- n}, F$ generate the Picard group of the surface, and the intersection product is determined by $C_{- n}^ 2 = -n$, $F^2 = 0$ and $C_{- n} F = 1$. The canonical class is given by $K_{\mathbb{F}_n} = - 2 C_{- n} - (n + 2) F$.\\ For $n = 0$, $\mathbb{F}_0$ is just $\mathbb{P}^1 \times \mathbb{P}^1$, while $\mathbb{F}_1$ is the unique non-minimal Hirzebruch surface, because of the presence of the curve $C_{-1}$.
\subsubsection{Blow ups of $\mathbb{P}^2$}
All the others rational surfaces are blow-ups of $\mathbb{P}^2$ or $\mathbb{F}_n$, with $n \ne 1$. We will deal a lot with blow-ups of $\mathbb{P}^2$ at $N$ points $p_1, \ldots, p_N$. If $X$ is such a blow-up, then $\Pic (X)$ is a free abelian group, of rank $N + 1$, generated by the class $L$ of a line not containing any of the $p_i$'s, and the exceptional curves $E_1, \ldots, E_N$ over the base points. The intersection rules are $L^2 = 1, L E_i = 0$ and $E_i E_j = - \delta_{i, j}$. Let $\pi : X \to \mathbb{P}^2$ be the blow-down morphism. The class of any irreducible (smooth or not) curve in $C \subset X$, different from the $E_i$'s, can be decomposed in this basis as $C = d L - m_1 E_1 \cdots - m_N E_N$, where $d > 0$ is the degree of the plane curve $\pi (C) \subset \mathbb{P}^2$, and $m_i \ge 0$ is the multiplicity of $\pi(C)$ at $p_i$.\\
Conversely, any linear system of the form $|d L - m_1 E_1 \cdots - m_N E_N|$ contains the strict transforms of curves of degree $d$, and multiplicity at least $m_i$ at $p_i$.\\
We will specify in any situation wether we are allowing infinitely near blow-ups, since they produce curves of self intersection lesser or equal than $(-2)$.
\subsection{Del Pezzo surfaces}
\begin{definition}
A smooth surface $X$ is a Del Pezzo surface if is rational and the anti - canonical linear system $- K_X$ is ample.\\
The self intersection $(- K_X)^2$ is called the degree of the Del Pezzo surface $X$.
\end{definition}
\begin{proposition}
Assume $X$ is a blow up of $\mathbb{P}^2$ at $N$ points. Then $X$ is a Del Pezzo surface if and only if $N \le 8$, and in this case its degree is $9 - N$.
\end{proposition}
The projective plane $\mathbb{P}^2$ and the product $\mathbb{P}^1 \times \mathbb{P}^1$ are the only cases of Del Pezzo surfaces with divisible anti-canonical bundle, that is, the anti canonical is a positive multiple of another very ample line bundle.\\
For $\mathbb{P}^2$ we have $$- K_{\mathbb{P}^2} = 3 L$$ which defines the Veronese embedding of $\mathbb{P}^2$ as a surface of degree $9$ in $\mathbb{P}^9$, but of course we can just pick $L$ itself as a very ample line bundle.\\
For $\mathbb{P}^1 \times \mathbb{P}^1$ we have $$- K_{\mathbb{P}^1 \times \mathbb{P}^1} = 2 (F_1 + F_2)$$ which realizes $\mathbb{P}^1 \times \mathbb{P}^1$ as a surface of degree $8$ inside $\mathbb{P}^8$. Again, we can just pick $F_1 + F_2$ as a very ample line bundle, which identifies $\mathbb{P}^1 \times \mathbb{P}^1$ with a smooth quadric $Q \subset \mathbb{P}^3$, and the embedding defined by $- K_{\mathbb{P}^1 \times \mathbb{P}^1}$ is just the composition of this Segre embedding with the Veronese embedding $\mathbb{P}^3 \to \mathbb{P}^9$ defined by the complete linear system of quadrics. Since $Q$ itself is one of the quadrics of the Veronese embedding, the image of the composition $$\mathbb{P}^1 \times \mathbb{P}^1 \to \mathbb{P}^3 \to \mathbb{P}^9$$ is actually contained in a hyperplane $\mathbb{P}^8 \subset \mathbb{P}^9$.\\
\subsubsection{Del Pezzo surfaces of degree in $\{3, \ldots, 8\}$}
Each of these surfaces is obtained as a blow up of $\mathbb{P}^2$ in a set of $N$ points, with $N \in \{1, \ldots, 6\}$. If $X$ is such a surface, with the notation we adopted we have $$\Pic (X) = \integer L \oplus \integer E_1 \cdots \oplus \integer E_N$$ and $$- K_X = 3 L - E_1 \cdots - E_N$$ which corresponds to the strict transforms of plane cubics passing through the $N$ base points. If the base points are in general position, the anti canonical linear system induces an embedding $$X \subset \mathbb{P}^d$$ as a surface of degree $d$, with $$d = 9 - N = (- K_X)^2$$ An important role in this representation is played by $(-1)$-curves, thanks to the following fact:
\begin{proposition}
The $(-1)$-curves in $X$ are in $1 : 1$ correspondence with the straight lines of $\mathbb{P}^d$ contained in $X$.\\
Moreover, each of this surfaces contains a finite number of lines.
\end{proposition}
It is also easy to list the $(-1)$-curves in each of these surfaces, still under the assumption that the $N$ base points are in general position:\\
{\bf N = 1:} There is only one line, the exceptional divisor of the base point.\\
{\bf N = 2:} There are $3$ lines, which live in the classes $E_1, E_2, L - E_1 - E_2$. These curves are displaced in a chain of length $3$.\\
{\bf N = 3:} There are $6$ lines, which live in the classes $E_1, E_2, E_3, L - E_1 - E_2, L - E_1 - E_3, L - E_2 - E_3$. They are displaced in an hexagon.\\
{\bf N = 4:} There are $10$ lines, which live in the classes $E_1, \cdots E_4$ and $L - E_i - E_j$ for all pairs $i < j$.\\
{\bf N = 5:} There are $16$ lines, which live in the classes $E_1, \cdots E_5$, plus all the $L - E_i - E_j$ for all pairs $i < j$, and the class $2 L - E_1 \cdots - E_5$.\\
{\bf N = 6:} There are $27$ lines, which live in the classes $E_1, \cdots E_6$, plus all the $L - E_i - E_j$ for all pairs $i < j$, and the classes $2 L - E_1 \cdots - \hat E_i \cdots - E_5$.
\subsubsection{Del Pezzo surfaces of degree $2$, and the Geiser involution}
A Del Pezzo surface of degree $2$ is a blow - up of $\mathbb{P}^2$ at $N = 7$ points $p_1, \ldots, p_7$. Let $X$ be such a surface. With the notation we adopted, the linear system $- K_X$ is given by $- K_X = 3 L - E_1 \cdots - E_7$, which consists of strict transforms of cubics through the base points $p_1, \ldots, p_7$. We have $$\Dim\, |- K_X| = 2$$ thus we have the induced morphism $f_{|- K_X|} : X \to \mathbb{P}^2$. The degree of $f$ is simply given by $${\rm deg\,} f = (- K_X)^2 = 2$$ This is the first case when $|- K_X|$ is not very ample, but just ample. The branch locus of $f$ is a smooth curve $B \subset \mathbb{P}^2$ of degree $4$, which is a non - hyperelliptic curve of genus $3$. The $(-1)$-curves in $X$ are still a finite number, namely $56$, and they are in $2 : 1$ correspondence with the set of bitangent lines of $B$, which are $28$.\\
Conversely, any double cover of $\mathbb{P}^2$ branched along such a curve $B$ is a Del Pezzo surface of degree $2$. The deck involution $i$ with respect to this double cover is called Geiser involution. For any point $x \in X$, the point $i (x)$ is the ninth base point for the pencil of cubics defined by $p_1, \ldots, p_7, x$.\\
The linear system $|- 2 K_X|$ is very ample, since it induces an embedding $|- 2 K_X| : X \to \mathbb{P}^6$, which identifies $X$ as a surface of degree $8$.
\subsubsection{Del Pezzo surfaces of degree $1$, and the Bertini involution}
A Del Pezzo surface $X$ of degree $1$ is the blow up of $\mathbb{P}^2$ at $8$ points $p_1, \ldots, p_8$. The anti - canonical divisor $- K_X$ has the form $$- K_X = 3 L - E_1 \cdots - E_8$$ so it consists of strict transforms of cubics through $p_1, \ldots, p_8$. Such cubic curves form a pencil, with a ninth base point $p_9$, which we do not blow up. The system $- 2 K_X = 6 L - 2 E_1 - \cdots - 2 E_8$ is made up of sextics with $8$ nodes, and $$h^0 ({\cal O}_X (- 2 K_X)) = 4$$ Let $F, G$ be generators for $H^0 ({\cal O}_X (- K_X))$, then the three forms $F^2, F G, G^2$ generate $Sym^2 H^0 ({\cal O}_X (- K_X)) \subset H^0 ({\cal O}_X (- 2 K_X))$, so we can pick a fourth generator $H$ to have $$H^0 ({\cal O}_X (- 2 K_X)) = Span (F^2, F G, G^2, H)$$ and the corresponding map $$f := f_{|- 2 K_X|} : X \to \mathbb{P}^3$$ takes the form $$[F^2, F G, G^2, H] : X \to \mathbb{P}^3$$ If we put coordinates $[X_0, X_1, X_2, X_3]$ on $\mathbb{P}^3$ we see then that the image of $X$ is the cone $$S := \{X_1^2 = X_0 X_2\}$$ which has the vertex $v$ in the point $v = [0, 0, 0, 1]$. The point $p_9$ is defined by equations $$p_9 = \{F = G = 0\}$$ on $X$, hence it is sent by $f$ on the vertex $v$: $$f(p_9) = v$$ Moreover $${\rm deg\,} f = \frac{(- 2 K_X)^2}{({\rm deg\,} f(X))} = 2$$ so that $f$ is a double cover of the cone. The restriction of $f$ to any elliptic curve ${\cal E}$ in $|- K_X|$ is a double cover of a generatrix line of $S$, so ${\cal E}$ contains $4$ ramification points of $f$. One of them is of course $p_9$, while the other $3$ move with ${\cal E}$, hence they describe a curve $R$ in $X$. The union $R \cup p_9$ is the ramificarion locus of $f$, and it corresponds to the fixed locus of the deck involution, called the Bertini involution on $X$.\\
The $(-1)$-curves of $X$ are $240$, and they are in $2:1$ correspondence with the set of plane sections of the cone which are totally tangent to the branch locus $B := f(R)$.\\
Conversely, any double cover of $S$ branched on the vertex and a non - hyperelliptic curve of genus $4$ is a Del Pezzo surface of degree $1$.
\subsubsection{The De-Jonquieres involution}
We recall here the construction of the De-Jonquieres involution of degree $d$.\\
Let $C_d$ be a curve of degree $d \ge 3$ in $\mathbb{P}^2$, with a unique singular point $p_0 \in Sing\, C_d$, of multiplicity $d - 2$. Consider the plane birational involution $i$ given as follows: for any line $L$ through $p_0$, write the divisor ${C_d}|_L$ as $${C_d}|_L = (d - 2) p_0 + a + b$$ Let $i|_L$ be the involution on $L$ with $a, b$ as fixed points, and let $i$ be the glueing of all the $i|_L$'s. Of course $i$ is undefined at $p_0$. Moreover, there are finitely many lines $L_i$ such that the restricted divisor ${C_d}|_{L_i}$ takes the form $${C_d}|_{L_i} = (d - 2) p_0 + 2 p_i$$ The points $p_i$ correspond to the ramification points of the double cover induced by the projection from $p_0$ $\pi_{p_0} : C_d \dasharrow \mathbb{P}^1$. Since $C_d$ has geometric genus $d - 2$, the number of $p_i$'s is $2 d - 2$. Clearly $i$ is undefined along $L_i$ too.\\
Let write it in coordinates: $C_d$ has a unhomogeneous local equation $$F_{d - 2} (x, y) + F_{d - 1} (x, y) + F_d (x, y) = 0$$ with $F_k$ an homogeneous polynomial of degree $k$. The discriminant equation $$F_{d - 1}^2 - 4 F_{d - 2} F_d = 0$$ describes the union of the lines $L_1, \ldots, L_{2 d - 2}$. The involution $i$ has the form $$i(x, y) = (-x \cdot \frac{F_{d - 1} + 2 F_{d - 2}}{2 F_d + F_{d - 1}}, -y \cdot \frac{F_{d - 1} + 2 F_{d - 2}}{2 F_d + F_{d - 1}}).$$ The three polynomials $x (F_{d - 1} + 2 F_{d - 2}), y (F_{d - 1} + 2 F_{d - 2}), 2 F_d + F_{d - 1}$ generate the net of curves of degree $d$, passing through $p_1, \cdots, p_{2 d - 2}$ and having multiplicity $d - 1$ at $p_0$.\\
Let $X$ be the blow up of $\mathbb{P}^2$ at $p_0, \ldots, p_{2 d - 2}$. The surface $X$ has a fibration $|F|$ in rational curves, namely the strict transforms of lines through $p_0$. On $X$ the involution $i$ becomes biregular, and it preserves each fiber of $|F|$. There are exactly $2 d - 2$ singular fibers, given by the union of the exceptional divisor associated to $p_i$, and the strict transform of the line $L_i$. The involution $i$ switches the component of each singular fiber, fixing the unique singular point. The fixed locus is given by construction by the strict transform of $C_d$, which is a bisection of $|F|$.
\subsection{Some Lattice theory}
We will keep this subsection to give some preliminaries about lattice theory.
\begin{definition}
A lattice $\Lambda$ is a free abelian group of finite rank, endowed with a bilinear symmetric product $< \cdot, \cdot > : \Lambda \times \Lambda \to \integer$.\\
For $a, b \in \Lambda$, we will denote by $a b$ the product $<a, b>$.\\
If $\Lambda_1, \Lambda_2$ are lattices, we denote by $\Lambda_1 \oplus \Lambda_2$ the lattice equipped with the product $(a_1, b_1) (a_2, b_2) := a_1 a_2 + b_1 b_2$.\\
An isometry of lattices $f : \Lambda_1 \to \Lambda_2$ is a $\integer$-isomorphism $f$ such that $f(a) f(b) = a b$ for all $a, b \in \Lambda_1$.
\end{definition}
We will denote by $\integer^{1, 10}$ the lattice of rank $11$ with signature $(1, 10)$, given as follows: it is generated by a list of elements $e_0, \ldots, e_{10}$, with the products $e_i e_j = 0$ for $i \ne j$, $e_0^2 = 1$, $e_1^2 = \cdots = e_{10}^2 = -1$. Let $k := - 3 e_0 + e_1 + \cdots + e_{10} \in \integer^{1, 10}$, and let its orthogonal $$k^\perp:= \{v \in \integer^{1, 10} {\rm\, s.\, t.\,} v \cdot k = 0\}.$$ Since $k^2 = -1$, it is possible to perform the Gram-Schmidt algorithm in $\integer^{1, 10}$ to write any vector $v \in \integer^{1, 10}$ as $v = (v + (v \cdot k) k) - (v \cdot k) k$ to show that $$\integer^{1, 10} = k^\perp \oplus \integer k$$ 
\begin{definition}
We denote by $\mathbb{E}_{10}$ the sublattice $\mathbb{E}_{10}:= k^\perp \subset \integer^{1, 10}$.
\end{definition}
A useful basis to deal with $\mathbb{E}_{10}$ is defined by the following elements: $\alpha_0 := e_0 - e_1 - e_2 - e_3, \alpha_1 := e_1 - e_2, \cdots, \alpha_9 := e_9 - e_{10}$.
The intersection product of this base are given by:\\
i) $\alpha_i^2 = -2$ for all $i = 0, \cdots, 9$.\\
ii) $\alpha_i \alpha_{i + 1} = 1$ for all $i = 1, \cdots, 8$.\\
iii) $\alpha_0 \alpha_3 = 1$.\\
iv) $\alpha_i \alpha_j = 0$ for all other cases.
\begin{definition}
For any $\alpha \in \mathbb{E}_{10}$ with $\alpha^2 = - 2$, consider the map $\rho_{\alpha} : \mathbb{E}_{10} \to \mathbb{E}_{10}$, $$\rho_{\alpha} (x) := x + (x \cdot \alpha) \alpha$$
\end{definition}
The map $\rho_{\alpha}$ is an isometric involution of $\mathbb{E}_{10}$, meaning that $\rho_{\alpha}^2 = \mathbb{1}$.
\begin{definition}
Let $O(\mathbb{E}_{10})$ be the orthogonal group of isometries of $\mathbb{E}_{10}$ in itself. The Weyl group $W(\mathbb{E}_{10}) \subset O (\mathbb{E}_{10})$ is the subgroup generated by all the $\rho_{\alpha}$'s, with $\alpha^2 = - 2$.
\end{definition}
Note that $W(E_{10})$ is a normal subgroup, since for any isometry $g \in O ({\mathbb{E}}_{10})$ we have $g \rho_{\alpha} g^{-1} = \rho_{g (\alpha)}$.
In \cite{ABriefIntro2016}, Dolgachev showed that $$O (\mathbb{E}_{10}) = W(\mathbb{E}_{10}) \times {\pm \mathbb{1}}$$
\subsection{Enriques surfaces}
An important example of non-rational surfaces we will deal with are Enriques surfaces.
\begin{definition}
An Enriques surface is a smooth surface $X$ satisfying $$h^1 ({\cal O}_X) = h^0 ({\cal O}_X (K_X)) = 0$$ and $$h^0 ({\cal O}_X (2 K_X)) = 1.$$
\end{definition}
\begin{proposition}
On an Enriques surface $X$, the canonical divisor is a $2$-torsion element in $\Pic (X)$, that is, $$2 K_X = 0.$$
\end{proposition}
The classical example of an Enriques surface is the following:
\begin{proposition}\label{ClassicalEnriques}
Let $T \subset \mathbb{P}^3$ be the tetrahedron $$T := \bigcup_{0 \le i < j \le 3} \{X_i = X_j = 0\}$$ and let $\overline X \subset \mathbb{P}^3$ be a surface of degree $6$ with double points along all lines in $T$. If $\overline X$ is generic, then the normalization $\nu : X \to \overline X$ is a smooth Enriques surface.
\end{proposition}
Enriques surfaces are the best known example of surfaces with $\Pic (X) \ne \Num (X)$. Indeed, since $K_X$ is a torsion divisor, we have $K_X D = 0$ for any $D \in \Pic (X)$.
\begin{proposition}
For an Enriques surface $X$, the subgroup in $\Pic (X)$ of divisors numerically equivalent to $0$ coincides with the subgroup $\integer_2$ generated by $K_X$. As a consequence, $$\Pic (X) = \Num (X) \oplus \integer_2 K_X$$
\end{proposition}
\begin{theorem}\cite{Cossec1985}
For an Enriques surface $X$, the lattice $\Num (X)$ is isometric to $\mathbb{E}_{10}$.
\end{theorem}
One of the most known invariants to classify polarized Enriques $(X, H)$ surfaces is the following:
\begin{definition}
A polarized Enriques surface is a pair $(X, H)$, with $H \in \Pic (X)$ the class of a big and nef divisor.\\
The $\phi$-invariant of the polarized pair $(X, H)$ is by definition: $$\phi (H) := {\rm\, min\,} \{H {\cal E}, {\rm\, with\,} {\cal E} {\rm\, a\, rigid\, elliptic\, curve\, in\,} X\}$$
\end{definition}
The previous definition is always well defined, as any Enriques surface contains rigid elliptic curves ${\cal E}$. These are also known as half - fibers, as the double $|2 {\cal E}|$ can always move in a basepoint - free pencil of elliptic curves, which has exactly two double fibers, namely ${\cal E}$ and ${\cal E} + K_X$.\\
Moreover, it was shown in \cite{EnriquesI2025} that the equality $\phi (H) = 1$ corresponds to a polarization $H$ with base points, since otherwise we the polarization $H$ would induce an isomorphism between a smooth elliptic curve ${\cal E}$ and $\mathbb{P}^1$. Moreover, the authors also classify which classes $H \in \mathbb{E}_{10}$ actually correspond to big and nef polarizations which achieve $\phi (H) = 1$.\\
The next value $\phi (H) = 2$ is achieved for example by Enriques surfaces $(X, H)$ which arise from Proposition \ref{ClassicalEnriques}, with $H$ the natural polarization induced from ${\cal O}_{\mathbb{P}^3} (1)$. Indeed, the normalizations of all the $6$ edges of the tetrahedron $T$ are rigid elliptic curves, and they are double covers over the corresponding lines.\\
The next value $\phi (H) = 3$ plays a special role:
\begin{definition}\label{Reyepol}
We call $(X, H)$ a Fano polarization if $H^2 = 10$ and $\phi (H) = 3$.\\
One immediately sees that a Fano polarization provides a map $|H| : X \to \mathbb{P}^5$.\\
A Fano polarization $(X, H)$ is called a Fano - Reye polarization if the image of $X$ is contained in a smooth quadric hypersurface of $\mathbb{P}^5$.
\end{definition}
The reason of the previous definition is the following: we know that $\Num (X) = \mathbb{E}_{10}$ has rank $10$. Assume for a moment that we able to find a sequence ${\cal E}_1, \ldots, {\cal E}_{10}$ of half - fibers, with intersection products ${\cal E}_i {\cal E}_j = 1 - \delta_{i, j}$. Such a sequence is known in literature as a maximal isotropic sequence. By a purely lattice - theoretical argument, the $\rational$-divisor $H := \frac{1}{3} ({\cal E}_1 + \cdots + {\cal E}_{10}) \in \mathbb{E}_{10} \otimes \rational$ is integer, that is, it actually belongs to $\mathbb{E}_{10}$, and it satifies $H^2 = 10$ and $\phi (H) = 3$. The minimum product $H {\cal E} = 3$ is achieved for example on all the ${\cal E}_i$'s.\\
Of course this works provided a maximal isotropic sequence of half - fibers, which not always exists. Moreover, it is a very hard work to find the maximal lengths for such sequences, see for example \cite{Ciliberto2023}, \cite{Knutsen2020}, \cite{Schaffler2024}.\\
The definition of a Fano - Reye polarization in \ref{Reyepol} is due to the following construction:
\begin{definition}\label{ClassicalReye} \cite{Cossec1983} \cite{Reye1882}
Let $W \subset H^0 ({\cal O}_{\mathbb{P}^3} (2))$ be a vector space of dimension $4$ of quardrics of $\mathbb{P}^3$. The surface $$Reye(W) := \{l \in Gr (1, \mathbb{P}^3) {\rm\, s.\, t.\,} l \subset Bs (|P|) {\rm\, for\, a\, pencil\,} |P| \subset |W|\} \subset$$ $$ \subset Gr (1, \mathbb{P}^3) \subset \mathbb{P}^5$$ is called a classical Reye congruence.
\end{definition}
The following fact is very well - known:
\begin{theorem}\cite{VerraReye1993}
If $W \subset H^0 ({\cal O}_{\mathbb{P}^3} (2))$ is a general $4$-dimensional $\complex$-vector subspace, then $Reye(W)$ is a smooth nodal Enriques surface, of degree $10$ in $\mathbb{P}^5$. 
\end{theorem}
It is also a very recent result that also the converse is true:
\begin{theorem}\cite{Gebhard2024}
A smooth nodal Enriques surface always admits a Fano - Reye polarization, although it need not be a classical Reye congruence.
\end{theorem}
\newpage
\section{Coble surfaces}
Now we introduce the main object of this work:
\begin{definition}\label{Coble}
A Coble surface is a smooth rational projective surface $X$ whose canonical divisor $K_X$ satisfies two properties:\\
$1.\quad h^0 ({\cal O}_X (- K_X)) = 0$,\\
$2.\quad h^0 ({\cal O}_X (- 2 K_X)) = 1$.
\end{definition}
We refer to \cite{CobleRatSurf2001} for the above definition. The definition clearly points out some analogy between Coble surfaces and Enriques surfaces,
to be especially reconsidered in this thesis. Note also that $-2K_X$ is a non zero effective divisor. Otherwise $-K_X$ would be a non trivial $2$-torsion element of $\Pic (X)$: against the rationality of $X$. 
\begin{definition} The unique effective curve $C \in |-2 K_X|$ is often called Coble curve, or boundary curve of the surface $X$.\\
The irreducible decomposition of $C$ is the equality \begin{equation} C = C_1 +  \dots + C_n, \end{equation} whose summands are irreducible
curves with two by two distinct supports. 
\end{definition}
\subsection{First properties and examples}
Definition \ref{Coble} above differs from the one given by Dolgachev and Zhang in \cite{ZhangDolgachev2001}, where the condition $h^0 ({\cal O}_X (- 2 K_X)) = 1$ is weakened to $h^0 ({\cal O}_X (-2 K_X)) \ge 1$. However, the authors showed the following fact:
\begin{proposition}\label{snc}
Let $X$ be a smooth rational surface with $h^0 ({\cal O}_X (- K_X)) = 0$ and $h^0({\cal O}_X (- 2 K_X)) \ge 1$. Then every divisor in $|- 2 K_X|$ is simple normal crossing.
\end{proposition}
Proposition \ref{snc} has quite remarkable consequences.  Indeed, let $D$ be any divisor in $|- 2 K_X|$. Since $D$ is simple normal crossing, it does not admit multiple components. Moreover, for any singular point $p \in Sing (D)$ there are two cases:\\
$i)$ either $p$ belongs to a unique irreducible component of $D$, and it is a simple node for that component;\\
$ii)$ or $p$ belongs to exactly two irreducible components of $D$, it is a smooth point for both of them, and the intersection is transverse.\\
Consequently, let $p_1, \ldots, p_n$ be the collection of all singularities of $D$, and let $\tilde X := Bl_{p_1, \ldots, p_n} X$ be the blow up of $X$ at $p_1, \ldots, p_n$, with exceptional components $E_1, \ldots, E_n \subset \tilde X$. Let $\tilde D \subset \tilde X$ be the strict transform of $D$. We claim that $\tilde X$ is a Coble surface with respect to Definition \ref{Coble}. Indeed, let $p: \tilde X \to X$ be the blow down of the exceptional curves $E_1, \ldots, E_n$ onto $p_1, \ldots, p_n$. Then $$K_{\tilde X} = p^* K_X + E_1 +\cdots + E_n.$$ If $|- K_{\tilde X}|$ was effective, its members would be sent via $p$ on effective members of $|- K_X|$, but this is empty. Hence $|- K_{\tilde X}| = \emptyset$. Moreover, since the points $p_1, \ldots, p_n$ are ordinary double points for $D$, the divisor $\tilde D$ is smooth, and it belongs to the class: $$\tilde D = p^* D - 2 E_1 - \cdots - 2 E_n = - 2 (p^* K_X + E_1 + \cdots + E_n) = - 2 K_{\tilde X}$$ To prove the claim, we only need to show that $\tilde D$ cannot move in $\tilde X$. To do so, we will use the following Proposition, which was proved by Dolgachev and Zhang in \cite{ZhangDolgachev2001}.
\begin{proposition}\label{comp}
Let $\{C\} = |- 2 K_X|$ be the Coble curve in a Coble surface $X$, with irreducible decomposition $C = C_1 + \cdots + C_n$.\\
If the divisor $C$ is smooth, then:\\
i) the $C_i$'s are smooth rational curves,\\
ii) $C_i^2 = -4$,\\
iii) $K_X^2 = - n$.
\end{proposition}
Before proving Proposition \ref{comp}, we will use it to show that the divisor $\tilde D$ can not move, so that $h^0 ({\cal O}_{\tilde X} (- 2 K_{\tilde X})) = 1$. It suffices to consider the short exact sequence $$0 \to {\cal O}_{\tilde X} \to {\cal O}_{\tilde X} (\tilde D) \to {\cal O}_{\tilde D} (\tilde D) \to 0$$ By Proposition \ref{comp}, the divisor $\tilde D$ has the form $\tilde D = C_1 + \cdots + C_n$, with $C_i \simeq \mathbb{P}^1$ and $C_i C_j = - 4 \delta_{i, j}$, so the right - hand term equals $\bigoplus_{i = 1}^n {\cal O}_{\mathbb{P}^1} (- 4)$. As a consequence, the associated long exact sequence establishes an isomorphism $H^0 ({\cal O}_{\tilde X}) \simeq H^0 ({\cal O}_{\tilde X} (\tilde D))$. Thus we showed that any Coble surface in the weak sense can be blown up to a Coble surface with respect to Definition \ref{Coble}, and that the divisor $\{C\} = | - 2 K_X|$ can be taken smooth. From now on, we will always assume this to be true for any Coble surface we will consider.\\\\
{\it Proof of Proposition \ref{comp}:} i) We use the short exact sequence $$0 \to {\cal O}_X (- C_1 - \cdots - C_n) \to {\cal O}_X \to \bigoplus_{i = 1}^n {\cal O}_{C_i} \to 0$$ Since $X$ is rational, the long exact sequence gives: $$0 \to \bigoplus_{i = 1}^n H^1 ({\cal O}_{C_i}) \to H^2 ({\cal O}_X (- C_1 - \cdots - C_n)) \to \cdots$$ But by Serre's duality $$H^2 ({\cal O}_X (- C_1 - \cdots - C_n)) = H^2 ({\cal O}_X (2 K_X)) = H^0 ({\cal O}_X (- K_X)) = 0$$ which forces $$H^1 ({\cal O}_{C_i}) = 0.$$
ii) From one side, part $i)$ gives $$C_i^2 + C_i K_X = -2$$ But from the other side, $$C_i^2 + C_i K_X = C_i^2 - \frac{1}{2} C_i (C_1 + \cdots + C_n) = \frac{1}{2} C_i^2$$ hence we have the thesis.\\
iii) We start from the equality $$- 2 K_X = C_1 + \cdots + C_n$$ and taking the square of both sides, by part $ii)$, we have $$4 K_X^2 = -4 n.\quad \square$$ 
\\
The following Proposition states which are the negative curves inside a Coble surface. The reader is referred to \cite{EnriquesII2025} for a proof.
\begin{proposition}\label{negativecurves}
Let $X$ be a Coble surface, with Coble curve $C = \{C_1 + \cdots + C_n\}$. If $D \subset X$ is an irreducible curve with $D^2 < 0$, then $D$ is a smooth rational curve with $$D^2 \in \{-1,-2,-4\}$$ and $$D^2 = -4 {\rm\, iff\,} D \in \{C_1, \ldots, C_n\}$$
\end{proposition}
{\it Proof:}  Assume $D$ is not one of the $C_i$'s, and let $p_a (D), p_g (D)$ be its algebraic and geometric genus respectively. By the adjunction formula $D^2 + D K_X = 2 p_a (D) - 2$. Since $C \in |-2 K_X|$ is effective, we have then $$D K_X \le 0$$ so $$2 p_a (D) - 2 \le D^2 < 0$$ This forces $$p_a (D) = 0$$ and $$D^2 \in \{-1, -2\}$$ The equality $p_a (D) = 0$ also forces $p_g (D) = 0$, which means that $D$ is both rational and smooth. $\square$
\begin{proposition}\cite{EnriquesII2025} \cite{Nikulin1979}
Every Coble surface $X$ can be constructed as the blow-up $X \to \mathbb{P}^2$ in a finite set $\Sigma \subset \mathbb{P}^2$, which can possibly contain infinitely near points.\\
The image $\overline C \subset \mathbb{P}^2$ of the anti - bicanonical curve $C \in |- 2 K_X|$ has degree $6$, and all its components are rational curves.\\
There exists an upper bound for the number $n$ of components of $C = C_1 +\cdots + C_{n}$, and it is $n \le 10$.
\end{proposition}
We briefly remark that, when a smooth surface is blown up at a point, the self - intersection of the canonical divisor drops by $1$. As a consequence, for a Coble surface $X \simeq Bl_\Sigma \mathbb{P}^2$ we find $K_X^2 = K_{\mathbb{P}^2}^2 - |\Sigma| = 9 - |\Sigma|$. This fact implies that $|\Sigma| = n + 9$ and hence the following property:
\begin{proposition}
Let $X$ be a smooth Coble surface such that $\{C\} = |- 2 K_X|$ is smooth and irreducible, then $X$ is the blow up of $\mathbb{P}^2$ at $10$ distinct point.
\end{proposition}
Historically speaking, that is exactly how Coble surfaces came out for the first time. In 1917 Coble himself posed the following question: if $\Sigma$ is a finite set of points in $\mathbb{P}^2$, how can we describe the automorphism group $\Aut  (Bl_\Sigma \mathbb{P}^2)$ ? The answer depended on the cardinality $|\Sigma|$ of the chosen set. For $|\Sigma| \le 8$ it was already well - known that the study of $\Aut  (Bl_\Sigma \mathbb{P}^2)$ was linked to the study of the Del Pezzo surfaces, which are by definition surfaces whose anti - canonical divisor is ample.\\
On the other hand, for $|\Sigma| \ge 9$, the answer was given later by the following theorem:
\begin{theorem}[Hirschowitz]\cite{Hirschowitz1988}
A general set $\Sigma$ of points in $\mathbb{P}^2$ with $|\Sigma| \ge 9$ satisfies $$\Aut  (Bl_\Sigma \mathbb{P}^2) = \mathbb{1}$$
\end{theorem}
Coble was looking for non general configurations $\Sigma$ such that $|\Sigma| \ge 9$ and the group $\Aut  (Bl_\Sigma \mathbb{P}^2)$ is nontrivial. Finite subsets of $\mathbb{P}^2$ with this property are known as Cremona - special subsets, and are deeply described by Cantat - Dolgachev, see for example \cite{CantatDolg2012}.
\begin{lemma}\cite{CantatDolg2012}
If $\overline C \subset \mathbb{P}^2$ is an irreducible sextic curve with nodes at ten points $p_1, \ldots, p_{10}$, then $$X := Bl_{p_1, \ldots, p_{10}} \mathbb{P}^2$$ is a Coble surface, and $\Aut  (X)$ is an infinite discrete group.
\end{lemma}
This was actually the original definition for a Coble surface, given by Coble himself in \cite{OriginalCoble1919}, and it was later generalized to Definition \ref{Coble}, so to include also reducible sextics, with more than $10$ nodes.\\
Note that for a generic finite subset $\Sigma \subset \mathbb{P}^2, |\Sigma| = 10$, there is no sextic $\overline C$ nodal at $\Sigma$, because $10$ nodes correspond to $30$ generally indipendent linear conditions over the $\complex$-vector space $H^0 ({\cal O}_{\mathbb{P}^2} (6 L))$, which has dimension $28$. When $\overline C$ actually does exist, the surface $X := Bl_\Sigma \mathbb{P}^2$ satisfies the requirements of Definition \ref{Coble}. In this case, the Coble curve $C$ is the proper transform of $\overline C$ in $X$.\\
Vice versa, if $X$ is a Coble surface (with respect to Definition \ref{Coble}), and $\pi: X \to \mathbb{P}^2$ is a blow-down map, then $$K_X = -3 L + E$$ where $L$ is the hyperplane section of $\mathbb{P}^2$, and the support of $E$ is contracted by $\pi$. If we put ourselves in the simplest situation, where we have no infinitely near points, then $E$ is a smooth divisor, and the requirement $$h^0 ({\cal O}_X (- 2 K_X)) = h^0 ({\cal O}_X (6 L - 2 E)) =1$$ correspond exactly to the existence of a sextic curve with nodes at the points in $\pi(E)$.\\ Meanwhile, the condition $$h^0 ({\cal O}_X (- K_X)) = h^0 ({\cal O}_X (3 L - E)) = 0$$ requires that no plane cubic curve contains all the points in $\pi (E)$.
\begin{example}
We will show in this example how to build a Coble surface $X$ whose Coble curve $\{C\} =|- 2 K_X|$ has $n = 2$ irreducible components. We know that we need to start with a sextic curve $\overline C \subset \mathbb{P}^2$ whose irreducible components are all rational. Moreover, the number of these components must necessarily be equal to $2$. Thus there are only three possibilities:\\
i) $\overline C = \overline C_5 + L$, where $\overline C_5$ is a rational quintic curve and $L$ is a line.\\
ii) $\overline C = \overline C_4 + C_2$, where $\overline C_4$ is a rational quartic, and $C_2$ is a smooth conic.\\
iii) $\overline C = \overline C_3^{(1)} + \overline C_3^{(2)}$, where $C_3^{(i)}$ are two rational cubics.\\
If we assume that the singularities of $\overline C_5, \overline C_4, \overline C_3^{(i)}$ are just simple nodes, their rationality forces $6$ nodes for $\overline C_5$, $3$ nodes for $\overline C_4$ and $1$ node for each $\overline C_3^{(i)}$. In total, each of these models has $11$ nodes. The blow up $X = Bl_{11} \mathbb{P}^2$ gives a Coble surface, where the strict transforms of the two components of $\overline C$ become the disjoint components of the Coble curve $C \in |- 2 K_X|$. The condition $|- K_X| = \emptyset$ corresponds to asking that no cubic curve passes through all the $11$ base points, and this is assured by counting the intersection multiplicities with the components of $\overline C$.
\end{example}
\begin{remark}\label{occhio}
The three cases in this example do not exhaust all the possible constructions for a Coble surface with two boundary components: there is another remaining case. Let $p_1, \ldots, p_8 \in \mathbb{P}^2$ be eight points in general position, and let $\overline X$ be their blow up $$\overline X := Bl_{p_1, \ldots, p_8} \mathbb{P}^2.$$ With the usual notation for the lattice $\Pic (X)$, a dimensional count shows that $$h^0 ({\cal O}_{\overline X} (6 L - 2 E_1 - \cdots - 2 E_7 - 3 E_8)) = 1$$ This divisor class contains the strict transform of the unique plane sextic curve with nodes at $p_1, \ldots, p_7$ and a triple point at $p_8$. Let $D \subset \overline X$ be this divisor. Then the sum $D + E_8$ equals $$D + E_8 = 6 L - 2 E_1 \cdots - 2 E_8 = -2 K_{\overline X}$$ and the two components meet at $D E_8 = 3$ points. Note that $D$ is a $(-1)$-curve too. Let $X$ be the blow up of the three intersection points of $D, E_8$, and let $\tilde D, \tilde E_8 \subset X$ be their strict transforms. Then $$\tilde D + \tilde E_8 = - 2 K_X$$ and the two components are disjoint $(- 4)$-curves, so that $X$ is again a Coble surface. Note that the total number of points we blew up is still $11$.
\end{remark}
\begin{example}
We now describe the costruction of a Coble surface $X$ with $n = 6$ irreducible components in the anti - bicanonical divisor $|- 2 K_X|$. We already know we need to start with a plane sextic curve $\overline C \subset \mathbb{P}^2$, but now we choose it in a very peculiar way, that is $\overline C = L_1 \cup \cdots \cup L_6$ is the union of six plane lines in general position. If the lines $L_i$ are general enough, the reducible curve $\overline C$ will have ${6 \choose 2} = 15$ singular points, namely the intersections of different components. Then let $X := Bl_{15} \mathbb{P}^2$ be the blow up of the plane at these points. The strict transform $C_i \subset X$ of each $L_i$ is a curve of self - intersection $-4$, since we blew $5$ points on $L_i$, and of course $C_i \cap C_j = \emptyset$.\\
It is also easy to show that Coble surfaces of this type are degenerations of Coble surfaces with irreducible anti - bicanonical divisor. Indeed, let $V \subset \mathbb{P}^6$ be the Del Pezzo surface of degree $6$, and let $V^* \subset {\mathbb{P}^6}^*$ the dual variety, that is $$V^* := \{H \subset \mathbb{P}^6 {\rm\, such\, that\,} H {\rm\, is\, tangent\, to\,} V\}$$ Pick two disjoint linear subspaces $\Lambda, \Gamma \subset \mathbb{P}^6$ with ${\rm dim\,} \Lambda = 2, {\rm dim\, } \Gamma = 3$, and $\Gamma \cap V = \emptyset$ and consider the projection from $\Gamma$, $$\pi_{\Gamma} : \mathbb{P}^6 \setminus \Gamma \to \Lambda \simeq \mathbb{P}^2$$ The embedding $V \subset \mathbb{P}^6$ is determined by the anti canonical divisor of $V$, hence the generic hyperplane sections of $V$ is a smooth elliptic curve. If $H \in V^*$, then the intersection $V \cap H$ is a curve of degree $6$ and geometric genus $0$. Then $\pi_{\Gamma} (V \cap H)$ has the same degree and geometric genus, hence $\pi_{\Gamma} (V \cap H)$ is an irreducible rational sextic plane curve, for a generic $H \in V^*$. But now we recall that $V$ contains a very peculiar section $H_0$, consisting of an hexagon of lines $$V \cap H_0 =  L_1' + \cdots + L_6'$$ with $$L_i \cap L_j = 1 {\rm\, if\,} |i - j| = 1 {\rm\, mod\,} 6$$ and $$L_i' \cap L_j' = 0 {\rm\, if\,} |i - j| \ge 2 {\rm\, mod\,} 6$$ The projection $\pi_{\Gamma}$ is linear, hence $\pi_{\Gamma} (L_1' + \cdots + L_6')$ consists of $6$ plane lines. Thus, we can fix a quasi-projective curve $R \subset V^*$ passing through the point corresponding to $H_0$, and we get an induced family of plane sextics, with a central element which is totally reducible.
\end{example}
\begin{example}
It is easy to show an example of a Coble surface with a maximal number $n = 10$ of irreducible components in the anti - bicanonical divisor. Let $S$ be the Del Pezzo surface of degree $5$ obtained blowing up $4$ points $p_1, p_2, p_3, p_4$ of $\mathbb{P}^2$, with no three of them collinear. It is well known that the linear system $-K_S$  defines a regular map $j : S \to \mathbb{P}^5$, which identifies $S$ as a quintic surface inside $\mathbb{P}^5$. We denote by $E_1, E_2, E_3, E_4 \subset S$ the exceptional curves associated to the $4$ base points; for each couple $(i, j)$ with $1 \le i < j \le 4$ let $L_{i, j} \subset S$ be the strict transform of the line through $p_i, p_j$. The curves $E_i, L_{i, j}$ are smooth rational $(-1)$-curves in $S$, and hence the linear system $|- K_S|$ embeds them as lines in $\mathbb{P}^5$. In total we have $10$ lines, and each of them touches $3$ of the others, leading to $15$ total intersection points. Let $$\pi : S' \to S$$ be the blow up of these $15$ points. Then $S'$ is a Coble surface. Indeed, let us denote by $F_1, \ldots, F_{15} \subset S'$ the new exceptional curves. We have $$- K_{S'} = \pi^* (- K_S) - F_1 - \cdots - F_{15} = \pi^*  {\cal O}_S (1) - F_1 - \cdots - F_{15}$$ Since no hyperplane of $\mathbb{P}^5$ contains all the $15$ intersection points, we have $|- K_{S'}| = \emptyset$. On the other hand, the relation $$\sum_i E_i + \sum_{i, j} L_{i, j} = - 2 K_S$$ holds in $\Pic(S)$, and the divisor on the left hand side is nodal at the $15$ points. If $\overline E_i, \overline L_{i, j} \subset S'$ are the strict transforms of the $10$ lines, then $$\sum_i \overline E_i + \sum_{i < j} \overline L_{i, j} = \pi^*( \sum_i E_i + \sum_{i, j} L_{i, j}) - 2 F_1 - \cdots - 2 F_{15} = - 2 K_{S'}$$ so that $$|- 2 K_{S'}| \ne \emptyset$$ The $K3$ double cover $V \to S'$ branched over the smooth divisor $\sum_i \overline E_i + \sum_{i < j} \overline L_{i, j}$ is called the Vinberg ``most algebraic'' $K3$ surface, see for example \cite{EnriquesII2025},\cite{Vinberg1983}, and it is a ``rigid'' surface. Indeed, its moduli are determined only by the choice of the first $4$ points $p_1, p_2, p_3, p_4 \in \mathbb{P}^2$, but you can always move a $4$-ple of $\mathbb{P}^2$ to another one by an element of $PGL (3)$.
\end{example}

\subsection{An extension result}
The following is a nice property of $(-1)$-curves on a Coble surfaces.
\begin{proposition}\label{extmoc}
Let $X$ be an unnodal Coble surface with irreducible Coble curve $\{C\} = |- 2 K_X|$, and let $E_1, \ldots, E_s \subset X$ be a set of disjoint $(-1)$-curves, with $s \le 8$. Then it can be completed to a $10$-ple $E_1, \ldots, E_{10}$ of pairwise disjoint $(-1)$-curves, such that the blowing down of $E_1, \ldots, E_{10}$ is $\mathbb{P}^2$.
\end{proposition}
{\it Proof:} Let's extend  $E_1, \ldots, E_s$ to a family $E_1, \ldots, E_r$ satisfying $E_i E_j = - \delta_{i, j}$ of maximal length $r$, with $r \ge s$. Let $$\pi: X \to \overline X$$ be the contraction of all the $E_i$'s. Since $X$ is unnodal, the smooth rational surface $\overline X$ is minimal, which forces either $$\overline X = \mathbb{P}^2$$ or $$\overline X = \mathbb{F}_m, m \ne 1$$ where $\mathbb{F}_m$ denotes the $m$-th Hirzebruch surface.\\
The first step is to exclude $\mathbb{F}_m$ for $m \ge 2$. This $\mathbb{F}_m$ contains a $(-m)$-curve, which would force $X$ to contain a curve $D$ with $D^2 \le -2$. Since $X$ is unnodal, $D$ should be the anti-bicanonical curve. But then the curve $\pi(D)$ is nodal, a contradiction since the $(-m)$-curve in $\mathbb{F}_m$ is smooth.\\
Then we are left with only two cases: $$\overline X = \mathbb{P}^2$$ or $$\overline X = \mathbb{F}_0 = \mathbb{P}^1 \times \mathbb{P}^1$$
Of course $$K_X = \pi^* K_{\overline X} + E_1 + \cdots + E_r$$ and taking squares one finds: $$-1 = K_X^2 = K_{\overline X}^2 - r$$ which gives $$\overline X = \mathbb{P}^2 {\rm\, iff\,} r = 10$$ and $$\overline X = \mathbb{P}^1 \times \mathbb{P}^1 {\rm\, iff\,} r = 9$$ In the first case we are done.\\
In the second case we extended $E_1,\ldots, E_s$ to a collection $E_1, \ldots, E_9 \subset X$ of disjoint $(-1)$-curves, and the blowing down of $E_1, \ldots, E_9$ is $\mathbb{P}^1 \times \mathbb{P}^1$. Let $$\pi : X \to \mathbb{P}^1 \times \mathbb{P}^1$$ be the blow down map, with $$p_1, \ldots, p_9 \in \mathbb{P}^1 \times \mathbb{P}^1$$ the points $$p_i := \pi (E_i)$$
Note that a fiber $F$ of one of the two rulings of $\mathbb{P}^1 \times \mathbb{P}^1$ cannot contain two distinct $p_i, p_j$, otherwise the strict transform of $F$ in $X$ is a $(-2)$-curve, which is forbidden by hypothesis.\\
Using this, and the fact that $s \le 8$, we can blow up $p_9$ on $\mathbb{P}^1 \times \mathbb{P}^1$ and blow down the fibers $F_1, F_2$ of the two rulings through $p_9$. This substitutes $\mathbb{P}^1 \times \mathbb{P}^1$ with the required $\mathbb{P}^2$. The final family of $(-1)$-curves will be $E_1, \ldots, E_8, F_1, F_2.$ $\square$\\
\\\\
Note that $s < 9$ is a sharp condition. Indeed, given a curve $\overline C \subset \mathbb{P}^1 \times \mathbb{P}^1$ of type $(4, 4)$ with nodes in $9$ points $p_1, \ldots, p_9 \in \mathbb{P}^1 \times \mathbb{P}^1$, we can consider the blow up $X := Bl_{p_1, \ldots, p_9} \mathbb{P}^1 \times \mathbb{P}^1$. Then $X$ is a Coble surface, but the family $E_1, \ldots, E_9$ of exceptional curves associated to $p_1, \ldots, p_9$ cannot be extended to a $10$-ple with the desired properties. $\square$
\\Note that it is hard to hope that the extension of $E_1, \ldots, E_s$ to $E_1, \ldots, E_{10}$ is unique. For example, if $s = 7$, a family $$E_1, \ldots, E_7$$ can be extended to $$E_1, \ldots, E_{10}$$ But we can also consider the lines joining two of the last three nodes $$\hat{E_8} \in |L - E_9 - E_{10}|, \hat{E_9} \in |L - E_8 - E_{10}|, {\hat E}_{10} \in |L - E_8 - E_9|$$ and the family $E_1, \ldots, E_7, \hat{E_8}, \hat{E_9}, {\hat E}_{10}$ still works as well. For $s \ge 2$, counting how many the completions of $E_1, \ldots, E_s$ are is the same as counting how many sets of $10 - s$ disjoint $(-1)$-curves are contained in a Del Pezzo surface of degree $s - 1$.\\
Here we generalize the result of Proposition \ref{extmoc} to possibly nodal Coble surface with irreducible boundary $\{C\} = |- 2 K_X|$.
\begin{definition}
Let $X$ be any smooth surface. A connected $(-1)$-chain in $X$ is an effective divisor $E = F_1 + \cdots + F_r$, with $$F_i \simeq \mathbb{P}^1$$ $$F_1^2  = \cdots = F_{r - 1}^2 = - 2$$ $$F_r^2 = -1$$ $$F_i F_{i + 1} = 1, F_i F_j = 0 \,{\rm\, for\,}\, |i - j| \ge 2$$
If $E$ is irreducible, we simply require that $E$ is a smooth rational curve of self intersection $-1$.
Given a chain $E = F_1 + \cdots + F_m$, the length $l(E)$ is the number $m$ of its irreducible components.\\
If $E_1, \ldots, E_n$ is a set of disjoint $(-1)$ chains, we set the length $l(E_1 + \cdots + E_n) := l(E_1) + \cdots + l(E_n)$ as the number of all irreducible components.\\
An extension of $E_1, \ldots, E_n$ is a set of disjoint $(-1)$-chains $E_1', \ldots, E_m'$, such that $m \ge n$ and $E_i' \ge E_i$ for $i = 1, \ldots, n$.
\end{definition}

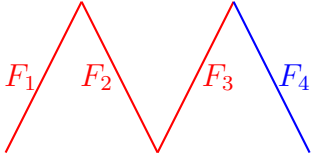
\begin{figure}[h!]
\begin{tikzpicture}
\draw [red,thick,mark={}, smooth] (-2,-1) -- (-1,1);
\draw [red,thick,mark={}, smooth] (-1,1) -- (0, -1);
\draw [red,thick,mark={}, smooth] (0, -1) -- (1,1);
\draw [blue,thick,mark={}, smooth] (1, 1) -- (2, -1);
\node [red] at (-1.8,0) {$F_1$};
\node [red] at (-0.8,0) {$F_2$};
\node [red] at (0.8,0) {$F_3$};
\node [blue] at (1.8,0) {$F_4$};
\end{tikzpicture}
\caption{A $(-1)$-chain with length $4$. The blue component $F_4$ is a $(-1)$-curve, the red components are $(-2)$-curves.}
\end{figure}

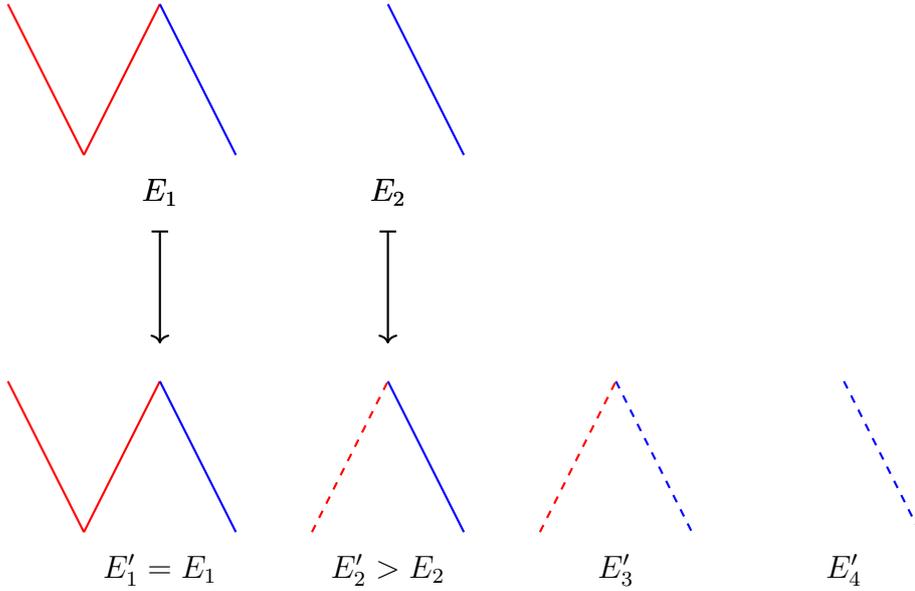
\begin{figure}[h!]
\begin{tikzpicture}
\draw [red,thick,mark={}, smooth] (-7, 4) -- (-6,2);
\draw [red,thick,mark={}, smooth] (-6, 2) -- (-5,4);
\draw [blue,thick,mark={}, smooth] (-5, 4) -- (-4, 2);
\draw [blue,thick,mark={}, smooth] (-2, 4) -- (-1, 2);
\node at (-5,1.5) {$E_1$};
\node at (-2,1.5) {$E_2$};
\draw [|->,thick] (-5,1) --(-5, -0.5);
\draw [|->,thick] (-2,1) --(-2, -0.5);
\draw [red,thick,mark={}, smooth] (-7, -1) -- (-6,-3);
\draw [red,thick,mark={}, smooth] (-6, -3) -- (-5,-1);
\draw [blue,thick,mark={}, smooth] (-5, -1) -- (-4, -3);
\draw [red,thick,mark={},dashed, smooth] (-3, -3) -- (-2,-1);
\draw [blue,thick,mark={}, smooth] (-2, -1) -- (-1, -3);
\draw [red,thick,mark={},dashed, smooth] (0, -3) -- (1,-1);
\draw [blue,thick,mark={},dashed, smooth] (1, -1) -- (2, -3);
\draw [blue,thick,mark={},dashed, smooth] (4, -1) -- (5, -3);
\node at (-5,1.5) {$E_1$};
\node at (-2,1.5) {$E_2$};
\node at (-5,-3.5) {$E'_1 = E_1$};
\node at (-2,-3.5) {$E'_2 > E_2$};
\node at (1,-3.5) {$E'_3$};
\node at (4, -3.5) {$E'_4$};
\end{tikzpicture}
\caption{With the same colors as above, the set $E_1', E_2', E_3', E_4'$ is an extension of $E_1, E_2$. The extra components are dashed.}
\end{figure}

By the Castelnuovo Criterion, any set of disjoint $(-1)$-chains $E_1, \ldots, E_n$ on a surface $X$ can be contracted to a set of points $p_1, \ldots, p_n$ on a smooth surface $Y$, with \begin{eqnarray}\label{length}K_Y^2 - K_X^2 = l(E_1) + \cdots + l (E_n),\end{eqnarray} see Figure \ref{figura}.
Also note that, given a smooth curve $D \subset Y$, with strict transform $\tilde D \subset X$, for any point $p_i$ which lies inside $D$, we can take the corresponding chain $E_i = F_1 + \cdots + F_{r_i}$. There exists a unique $F_{s_i}$ touching $\tilde D$, for some $1 \le s_i \le r_i$ and these numbers satisfy \begin{eqnarray}\label{diff}\sum_{p_i \in D} s_i = D^2 - \tilde D^2,\end{eqnarray} see Figure \ref{figura}.
\\
\\
\\
\\
\begin{figure}[h!]
\begin{tikzpicture}
\node [green] at (-4, 3) {$\tilde D$};
\draw [green, thick, smooth] (-3.5,3) to [out = 270, in = 60] (-4,2) to [out = 240, in = 120] (-4,0) to [out = 300, in = 60] (-4, -2) to [out= 240, in = 90] (-4.5, -3);
\node [red] at (-4.7,0) {. . .};
\node [red] at (-3.3, 0) {. . .}; 
\draw [red, thick, smooth] (-4.5, -1) -- (-3.5, 1);
\node [red] at (-3.5, 0.5) {$F_s$};
\draw [red, thick, smooth] (-7, 1) -- (-6, -1);
\node [red] at (-7, 0) {$F_1$};
\draw [red, thick, smooth] (-6, -1) -- (-5, 1);
\draw [red, thick, smooth] (-3, -1) -- (-2, 1);
\draw [blue, thick, smooth] (-2, 1) -- (-1, -1);
\node [blue] at (-1, 0) {$F_r$};
\draw [|->, thick] (0, 0) -- (3, 0);
\node [green] at (4, 3) {$D$};
\draw [green, thick, smooth] (4.5,3) to [out = 270, in = 60] (4,2) to [out = 240, in = 120, mark=*, ultra thick] (4,0) to [out = 300, in = 60] (4, -2) to [out= 240, in = 90] (3.5, -3);
\fill [green] (4,0) circle (2pt);
\node [green] at (4.5, 0) {$p$}; 
\end{tikzpicture}
\caption{We can imagine to contract the $(-1)$-chain $F_1 + \cdots + F_r$ on the point $p$ starting from the $(-1)$-curve $F_r$. This turns $F_{r - 1}$ into a $(-1)$-curve, so it can be contracted as well, and we proceed backwards to $F_1$. Each step makes the self - intersection $K_X^2$ jump by $1$, hence we get formula \ref{length}. Meanwhile, the self - intersection $\tilde D^2$ is affected only by $F_s, F_{s - 1}, \ldots, F_1$, and each of these produces a jump by $1$. The sum of the contributions of all $(-1)$-chains involved gives formula \ref{diff}.}\label{figura}
\end{figure}
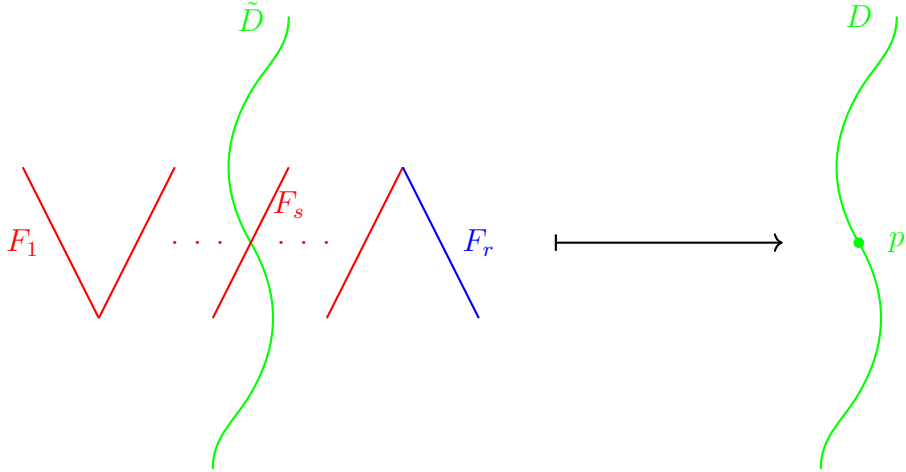

\begin{proposition}
Let $X$ be a smooth Coble surface with irreducible Coble curve $C$. Then any set $E_1, \ldots, E_n$ of disjoint $(-1)$-chains with length $l(E_1 + \cdots + E_n) \le 8$ can be extended to a set $E_1', \cdots E_m'$ with length $10$, such that the contraction of $E_1', \ldots, E_m'$ is $\mathbb{P}^2$.
\end{proposition}
{\it Proof:} The proof will be very similar to the proof of Proposition \ref{extmoc}.\\
Let $E_1', \cdots E_m' \subset X$ be an extension of $E_1, \ldots, E_n$ of maximal length, and let $Y$ be the surface obtained via the contraction of $E_1', \ldots, E_m'$, with base points $p_1, \ldots, p_m \in Y$. We claim that $Y$ does not contain any $(-1)$ curve.\\ On the countrary, assume $E \subset Y$ to be a $(-1)$ curve, and let $\tilde E \subset X$ be its strict transform. If $E$ does not contain any of the $p_i$, then $\tilde E$ is still a $(-1)$ curve, so we can build a greater set $E_1', \ldots, E_m', E$, which contradicts the maximality of the length. Then $E$ must contain some of the $p_i$'s, so $$\tilde E^2 \le -2$$ By Proposition \ref{negativecurves}, the equality $$\tilde E^2 = -2$$ must hold. The relation \eqref{diff} applied to $E, \tilde E$ gives $$\sum_{p_i \in E} s_i = 1$$ This means that $E$ contains exactly one base point, say $p_1 \in E$ while $p_2, \ldots, p_m \not\in E$, and the chain $E_1'$ has the shape $E_1' = F_1 + \cdots + F_{r_1}$, with $\tilde E$ touching $F_1$ at exactly one point, and disjoint from the other components. Hence we can build another extension, namely $\tilde E + E_1', E_2', \ldots, E_m'$, which again is impossible by maximality.\\
Thus $Y$ does not contain any $(-1)$-curve, so it is a smooth minimal rational surface, that is $$Y = \mathbb{P}^2$$ or $$Y = \mathbb{F}_m$$ for some $m \ne 1$. We want to examine all the possibilities.\\
First of all, remember that $\mathbb{F}_m$ contains a curve of self intersection $-m$. Hence, if $m \ge 3$, then $X$ contains a curve of self intersection smaller or equal than $-3$, which is forbidden by Proposition \ref{negativecurves}.\\
So we are left with $3$ cases: $Y = \mathbb{P}^2, \mathbb{P}^1 \times \mathbb{P}^1$ or $\mathbb{F}_2$. By \eqref{length}, we can now compute the length $$l(E_1' + \cdots + E_m') = K_Y^2 - K_X^2 = K_Y^2 + 1$$ which gives $$l(E_1' + \cdots + E_m') = 10 \ {\rm\, if\,} \ Y = \mathbb{P}^2$$ and $$l(E_1' + \cdots + E_m') = 9 \ {\rm\, if\,} \ Y = \mathbb{P}^1 \times \mathbb{P}^1 \ {\rm\, or\,} \ \mathbb{F}_2$$
If $Y = \mathbb{P}^2$ the proof is complete.\\
If $Y = \mathbb{P}^1 \times \mathbb{P}^1$ we act as in the proof of Proposition \ref{extmoc}, blowing up a base point, and blowing down the strict transforms of the fibers passing through it, and we get a family of $(-1)$-chains of strictly higher length, which is absurd.\\
Finally, if $Y = \mathbb{F}_2$, note that none of the points $p_1, \ldots, p_m$ can lie on the $(-2)$-section $C_{-2} \subset \mathbb{F}_2$, otherwise the strict transform $\tilde C_{-2} \subset X$ would have $$\tilde C_{-2} \le -3$$ which is impossible.\\
By definition, the original set $E_1, \ldots, E_n$ satisfies $n \le m$ and $E_i \le E_i'$ for $i = 1, \ldots, n$. By hypothesis, $l(E_1) + \cdots + l(E_n) \le 8$, while $l(E_1') + \cdots + l(E_m') = 9$. So we can write the last extended $(-1)$-chain $E_m'$ as $$E_m' = F_1 + \cdots + F_{r'}$$ with $F_1$ not appearing in the original $E_1, \dots, E_n$. Consider the corresponding base point $p_m \in Y$, and let $F$ be the fiber in $\mathbb{F}_2$ passing through $p_m$. Let  $\tilde F \subset X$ be its strict transform. Proposition \ref{negativecurves} gives $$\tilde F^2 \in \{-1, -2\}$$ If $\tilde F^2 = -1$, relation \ref{diff} states that $p_m$ is the only base point lying in $F$, and $\tilde F$ touches the first component $F_1$. So we can substitute the $(-1)$-chain $E_m'$ of length $r'$ with two disjoint $(-1)$-chains, namely $(E_m' - F_1)$ and $\tilde C_{-2} + \tilde F$, with length $r' + 1$. This is impossible by maximality.\\
Similarly, if $\tilde F^2 = -2$ we have $$\sum_{p_i \in F} s_i = 2$$ which means that $F$ can contain at most two base points. If $F$ contains only $p_m$ as base point, then $\tilde F$ touches $F_2$ at one point, so we can throw away $F_1$ to build a longer chain $\tilde C_{-2} + \tilde F + (E_{m}' - F_1)$, which is a connected $(-1)$-chain of length $r + 1$. Again, this is forbidden by maximality.\\
If another base point, say $p_{m - 1}$, lies inside $F$ together with $p_m$, we need to write down its corresponding $(-1)$-chain. So $$E_{m - 1}' = F_1^{(m -1)} + \cdots + F_s^{(m - 1)}$$ where $F_s^{(m - 1)}$ is a $(-1)$-curve, while the other components have self-intersection $-2$. Moreover, $\tilde F$ touches at one point both $F_1$ and $F_1^{(m - 1)}$, which gives the possibility to build two disjoint chains, that are $\tilde C_{-2} + \tilde F + E_{m - 1}'$ and $E_m' - F_1$. Since their complexive length is $r + s + 1$, this contradicts the maximality of $E_1', \ldots, E_m'$.\\
This exhausts all the possibilities, so the proof is complete. $\square$
\\\\Exceptional curves inside Coble surfaces are closely related to elliptic curves. Indeed we first claim that a Coble surface $X$ with irreducible Coble curve $C$ satisfies: \begin{eqnarray}\label{nohuno} h^1 ({\cal O}_X (- K_X)) = 0. \end{eqnarray} This comes from the application of the Riemann - Roch Theorem to the divisor $- K_X$, so that: $$h^0 ({\cal O}_X (- K_X)) - h^1 ({\cal O}_X (- K_X)) + h^2 ({\cal O}_X (- K_X)) = 1 + \frac{(- K_X) (- 2 K_X)}{2} = 1 + K_X^2 = 0.$$ But $h^0 ({\cal O}_X (- K_X)) = 0$ by Definition \ref{Coble}, and $h^2 ({\cal O}_X (- K_X)) = h^0 ({\cal O}_X (2 K_X)) = 0$, because of Serre's duality and the rationality of $X$. This proves equality \eqref{nohuno}.\\
Now for any $(-1)$-curve $E \subset X$ inside a Coble surface $X$ with irreducible Coble curve $C$, we have a short exact sequence: $$0 \to {\cal O}_X (- K_X) \to {\cal O}_X (E - K_X) \to {\cal O}_E (E - K_X) \to 0$$ Together with the relations \eqref{nohuno} and $${\cal O}_E (E - K_X) = {\cal O}_{\mathbb{P}^1} (E(E - K_X)) = {\cal O}_{\mathbb{P}^1},$$ this induces an isomorphism $$H^0 ({\cal O}_X (E - K_X)) \simeq \complex.$$ Hence the divisor ${\cal E} := E - K_X$ is effective and rigid, that is, it cannot move. The adjunction formula immediately proves that ${\cal E}$ is an elliptic curve. The following proposition shows also that also the converse is true:
\begin{proposition}\label{ellpenc}
Let ${\cal E}$ be an elliptic curve on a Coble surface $X$, with $C \in |- 2 K_X|$ irreducible anti - bicanonical divisor. If ${\cal E}$ is rigid, then the divisor $2 {\cal E}$ moves in a basepoint - free pencil of elliptic curves.\\
Conversely, every base point-free pencil of elliptic curves ${\cal E'}$ has the form ${\cal E'} = 2 {\cal E}$, with ${\cal E}$ an isolated elliptic curve.
\end{proposition}
{\it Proof:} 
We start from the short exact sequence \begin{eqnarray}\label{iec} 0 \to {\cal O}_X \to {\cal O}_X ({\cal E}) \to {\cal O}_{\cal E} ({\cal E}) \to 0\end{eqnarray} which remains exact on global sections: $$0 \to H^0 ({\cal O}_X) \to H^0 ({\cal O}_X ({\cal E})) \to H^0 ({\cal O}_{\cal E} ({\cal E})) \to 0.$$ Since ${\cal E}$ is rigid, we deduce $$H^0 ({\cal O}_{\cal E} ({\cal E})) = 0.$$ The adjunction formula gives $${\cal O}_{\cal E} ({\cal E} + K_X) = {\cal O}_{\cal E}$$ thus $${\cal O}_{\cal E} ({\cal E}) = {\cal O}_{\cal E} (- K_X).$$ In particular, $$\Deg\, {\cal O}_{\cal E} ({\cal E}) = \frac{1}{2} {\cal E} C \ge 0.$$ If the inequality was sharp, then Riemann-Roch Theorem would imply $h^0 ({\cal O}_{\cal E} ({\cal E})) > 0$ which is false, hence $$\Deg\, {\cal O}_{\cal E} ({\cal E}) = 0$$ Consequently, $${\cal O}_{\cal E} (2 {\cal E}) = {\cal O}_{\cal E} (- 2 K_X) = {\cal O}_{\cal E} (C)$$ is an effective line bundle of degree $0$, thus ${\cal O}_{\cal E} (2 {\cal E}) = {\cal O}_{\cal E}$. In other words, ${\cal O}_{\cal E} ({\cal E})$ is a $2$-torsion divisor, and by Serre duality: $$h^1 ({\cal O}_{\cal E} ({\cal E})) = h^0 ({\cal O}_{\cal E} ({\cal E})) = 0$$ The long exact sequence induced by \eqref{iec} now implies that $$h^1 ({\cal O}_X ({\cal E})) = 0.$$ Finally we look at the sequence $$ 0 \to {\cal O}_X ({\cal E}) \to {\cal O}_X (2 {\cal E}) \to {\cal O}_{\cal E} (2 {\cal E}) = {\cal O}_{\cal E} \to 0,$$ which gives $$h^0 ({\cal O}_X (2 {\cal E})) = 2.$$
It is easy to show that the pencil $|2 {\cal E}|$ has no base points or base components: indeed, a base component of $2 {\cal E}$ can only be ${\cal E}$ itself. But if ${\cal E}$ was a base component, then we would have $h^0 ({\cal O}_X (2 {\cal E} - {\cal E})) = h^0 ({\cal O}_X (2 {\cal E})) = 2$, which is a contradiction with the rigidity of ${\cal E}$. Finally, if the pencil $|2 {\cal E}|$ had a base point, this should lie on the reduced curve ${\cal E}$. Since $|2 {\cal E}|$ has no base components, the intersection product ${\cal E} (2 {\cal E}) = 2 \, \Deg\, {\cal O}_{\cal E} ({\cal E}) = 0$ shows that the generic element of $|2 {\cal E}|$ does not touch ${\cal E}$, and hence this pencil has no base points.\\
Conversely, assume ${\cal E'}$ is a base point free pencil of elliptic curves: necessarily we have $${\cal O}_{\cal E'} ({\cal E'}) = {\cal O}_{\cal E'}$$ and by adjunction formula also $${\cal O}_{\cal E'} (K_X) = {\cal O}_{\cal E'}$$ Consider the divisor $$D:= {\cal E'} + K_X.$$ We first show that it is effective. With this purpose, we consider the sequence $$0 \to {\cal O}_X (K_X) \to {\cal O}_X ({\cal E'} + K_X) \to {\cal O}_{\cal E'} ({\cal E'} + K_X) \to 0$$ which induces an isomorphism $$H^0 ({\cal O}_X ({\cal E'} + K_X)) \simeq H^0 ({\cal O}_{\cal E'} ({\cal E'} + K_X)) = H^0 ({\cal O}_{\cal E'}) \simeq \complex.$$ Hence $D$ is effective, and the relation $${\cal E'} D = {\cal E'} ({\cal E'} + K_X) = 0$$ yields that every component of $D$ is contained in a fiber of ${\cal E'}$. Nonetheless, we also have $$H^0 ({\cal O}_X ({\cal E'} - D)) = H^0 ({\cal O}_X (- K_X)) = 0$$ so one fiber of $|{\cal E'}|$ is not sufficient to contain all of $D$. On the other side, \begin{eqnarray}\label{parteuno} H^0 ({\cal O}_X (2 {\cal E'} - D)) = H^0 ({\cal O}_X ({\cal E'} - K_X)) \end{eqnarray} This quantity can be computed via the short exact sequence $$0 \to {\cal O}_X (- K_X) \to {\cal O}_X ({\cal E'} - K_X) \to {\cal O}_{\cal E'} ({\cal E'} - K_X) \to 0$$ which induces an isomorphism \begin{eqnarray}\label{partedue}H^0 ({\cal O}_X ({\cal E'} - K_X)) \simeq H^0 ({\cal O}_{\cal E'} ({\cal E'} - K_X)) = H^0 ({\cal O}_{\cal E'}) = \complex \end{eqnarray} Of course, the vanishing \eqref{nohuno} was used to derive this isomorphism. Putting together the relations \eqref{parteuno} and \eqref{partedue} we find that two fibers of ${\cal E'}$ are sufficient to contain all of $D$, so we have a disjoint decomposition $$D = E + {\cal E}$$ Moreover, the product $$D C = ({\cal E'} + K_X) (- 2K_X) = 2$$ forces, up to the order, $$E C = 2 \ {\rm\, and\,} \ {\cal E} C = 0$$ Let $p_a (E)$ denote the algebraic genus of $E$. Since $C$ and $E$ touch each other, they belong to the same fiber of ${\cal E'}$, so that $C + E \le {\cal E'}$. This forces $p_a (C + E) \le 1$, but we also know that $p_a (C + E) = p_a (C) + p_a (E) + 1 = p_a (E) + 1$. The only possibility is that $$p_a (E) = 0$$ and by adjunction formula $$E^2 = 2 p_a (E) - 2 - E K_X = - 2 + \frac{E C}{2} = -1,$$ so that $E$ is a smooth $(-1)$-curve. Now we compute the intersection product between $E$ and the effective divisor ${\cal E'} - C - E$, and we find $$({\cal E'} - C - E) E = -1$$ Hence $E$ is necessarily a component of ${\cal E'} - C - E$. So we showed that $$C + 2 E \le {\cal E'}$$ and since both these two divisors generate base point-free pencils, we have $$C + 2 E = {\cal E'}$$ and thus $${\cal E'} - 2 {\cal E} = {\cal E'} - 2 ({\cal E'} + K_X - E) = - {\cal E'} + C + 2 E = 0$$ that is $${\cal E'} = 2 {\cal E} \ \ \square$$\\
The previous Proposition is a glimpse on a very large theory, about Halphen pencils. By definition, a Halphen pencil of index $m$ is a rational surface $Y$ such that the divisor $|- m K_Y|$ defines a relatively minimal and basepoint - free pencil of curves, with exactly one multiple fiber, of multiplicity $m$. The blow - down of the curve $E$ in Proposition \ref{ellpenc} defines a Halphen pencil of index $2$, with $2 {\cal E}$ as the unique double fiber. For a wider overview on Halphen pencils, see \cite{MirandaZanardini2022}, \cite{Zanardini2022}.\\
Note also that Proposition \eqref{ellpenc} also showed that ${\cal E} + K_X$ is the class of a $(-1)$-curve in $X$, for any isolated elliptic curve ${\cal E}$. We also remark that if we start with two isolated elliptic curves ${\cal E}_1 = E_1 - K_X, {\cal E}_2 = E_2 - K_X$, then ${\cal E}_1 {\cal E}_2 = E_1 E_2 + 1$. Combining this fact with the results from Propositions \ref{extmoc} and \ref{ellpenc}, we finally get to the main result of this subsection:
\begin{theorem}\label{ndindex}
Let $X$ be an unnodal Coble surface, with irreducible boundary curve $C$. Then any sequence ${\cal E}_1, \ldots, {\cal E}_r$ of isolated elliptic curves satisfying ${\cal E}_i {\cal E}_j = 1 - \delta_{i, j}$ and $r \le 8$ can be extended to a sequence ${\cal E}_1, \ldots, {\cal E}_{10}$ with the same property.
\end{theorem}
In the classical literature of Enriques surfaces, a sequence ${\cal E}_1, \ldots, {\cal E}_r$ of isolated elliptic curves with intersection products ${\cal E}_i {\cal E}_j = 1 - \delta_{i, j}$ is known as an isotropic sequence. Cossec proved in \cite{Cossec1985} that Theorem \ref{ndindex} is true also on unnodal Enriques surfaces. For the study of this problem on Enriques surfaces, including nodal ones, we refer to \cite{Ciliberto2023}, \cite{Knutsen2020}, \cite{Schaffler2024}.

\newpage
\subsection{Moduli space of Coble surfaces}  
Nowadays, Coble surfaces are studied also for other reasons: the two conditions of Definition \ref{Coble} imply that for every Coble surface $X$ there exists a map $\pi: \tilde X \to X$, where $\tilde X$ is a smooth $K3$ surface and $\pi$ is a double cover, ramified over the Coble curve. As a consequence, Coble surfaces are closely related to Enriques surfaces, which are quotients of $K3$ surfaces by a fixed point - free involution. More precisely, we refer to \cite{Moduli2013}, where the authors proved the following result:
\begin{theorem}\cite{Moduli2013}
The coarse moduli space of nodal Enriques surfaces ${\cal M}_{En, nod}$ and of Coble surfaces ${\cal M}_{Co}$ are both rational varieties of dimension $9$.
\end{theorem}
Our goal is to give an idea about why it is reasonable that ${\cal M}_{Co}$ has dimension $9$: a rational sextic curve $\overline C \subset \mathbb{P}^2$ is the image of a regular map $\gamma$: $$\gamma: \mathbb{P}^1 \to \mathbb{P}^2$$ \begin{eqnarray}\label{parsex}\gamma (u, v) = [F_0 (u, v), F_1 (u, v), F_2(u, v)]\end{eqnarray} where the $F_i$'s are three linearly independent homogeneous polynomials of degree $6$ over $\mathbb{P}^1$, without common zeroes. These three forms span a subspace $V \subset H^0 ({\cal O}_{\mathbb{P}^1} (6))$. Clearly, another base of $V$ made by $G_0, G_1, G_2$ defines a different map $$\gamma' : \mathbb{P}^1 \to \mathbb{P}^2$$ $$\gamma' (u, v) = [G_0 (u, v), G_1 (u,  v), G_2 (u, v)]$$ with image a curve $C'$. Of course there is an element $M \in \mathbb{P}GL (2)$ which moves $C$ onto $C'$, which lifts to an isomorphism of the corresponding Coble surfaces. Thus the isomorphism class of a Coble surface $X$ depends only on the choice of an element $V \in \Gr (3, H^0 ({\cal O}_{\mathbb{P}^1} (6)))$.\\
Viceversa, ``almost every'' subspace $V \subset H^0 ({\cal O}_{\mathbb{P}^1} (6))$ of dimension $3$ is good to build a Coble surface.\\
Indeed, if $V$ is basepoint-free one can define the regular map $$\gamma_{|V|} : \mathbb{P}^1 \to \mathbb{P} (V^*) \simeq \mathbb{P}^2$$ $$\gamma(p) := \{\sigma \in V {\rm\, such\, that\,} \sigma (p) = 0\}$$ When we look at the projective coordinates, the map $\gamma_{|V|}$ has exactly the shape \ref{parsex}. We can give a geometric interpretation of this map as the composite \begin{eqnarray}\label{projcurve}\mathbb{P}^1 \to \mathbb{P} (H^0 ({\cal O}_{\mathbb{P}^1} (6))^*) \dasharrow \mathbb{P} (V^*)\end{eqnarray} The first map in \eqref{projcurve} is the standard Veronese embedding $$\mathbb{P}^1 \to \mathbb{P}(H^0 ({\cal O}_{\mathbb{P}^1} (6))^*) \simeq \mathbb{P}^6$$
$$p \to \{\sigma \in H^0 ({\cal O}_{\mathbb{P}^1} (6)) {\rm\, such\, that\,} \sigma (p) = 0\}$$
The rational function on the right side of \eqref{projcurve} is induced by the linear map $H^0 ({\cal O}_{\mathbb{P}^1} (6))^* \to V^*$, dual to the inclusion $V \subset H^0 ({\cal O}_{\mathbb{P}^1} (6))$, and consists precisely to the projection with center the projectified annihilator $\mathbb{P}( \Ann\, (V)) \subset \mathbb{P} (H^0 ({\cal O}_{\mathbb{P}^1} (6))^*)$. Thus the image of $$C_V := \gamma_{|V|} (\mathbb{P}^1)$$ will be the projection of a rational normal curve from some subspace.\\
Since $$\gamma_{|V|}^* {\cal O}_{\mathbb{P}^2} (1) = {\cal O}_{\mathbb{P}^1} (6)$$ one finds $$(\Deg\, \gamma_{|V|}) (\Deg\, C_{|V|}) = 6$$ Then the space of semi - stable nets, in the sense of \cite{Mumford1994} $$U:= \{V \in \Gr (3, H^0 ({\cal O}_{\mathbb{P}^1} (6))) {\rm\, such\, that\,} Bs |V| = \emptyset $$ $${\rm\, and\,} \Im (\gamma_{|V|}) {\rm\, is\, a\,} 10-{\rm nodal\, sextic}\}$$ is an open subset of the Grassmannian $\Gr (3, H^0 ({\cal O}_{\mathbb{P}^1} (6))) \simeq \Gr (3, 7)$. In the complementary closed subset, there are subspaces $V$ such that $\gamma_{|V|}$ is a double cover of a nodal cubic curve, or a triple cover of a smooth conic, or the image has worse singularities.\\
Moreover, we have to take in account the reparametrizations of $\mathbb{P}^1$, that are elements of $\mathbb{P} GL (2, \complex)$. Thus, the GIT-quotient $${\cal M}_{Co} := U // PGL (2, \complex)$$ constructed again as in \cite{Mumford1994}, is a suitable moduli space for rational plane sextics. Note that $$\Dim\, U = \Dim\, \Gr (3, H^0 ({\cal O}_{\mathbb{P}^1} (6))) = 12,$$ so that $$\Dim\, {\cal M}_{Co} = 9.$$

\newpage
\section{Projective models of Coble surfaces}
\subsection{The Bordiga model}
There is another model for Coble surfaces, which also underlines their parenthood with Enriques surfaces: the Bordiga model. This subsection is inspired by the construction of a Reye - Enriques model in \cite{Verra2023}.\\
One looks at the family of quadrics in $\mathbb{P}^3$, that is, the family $$\mathbb{P}^9 \simeq \mathbb{P}(H^0 (\mathbb{P}^3, {\cal O}(2))) $$ We identify each quadric polynomial $Q \in H^0 (\mathbb{P}^3, {\cal O} (2))$ with a symmetric matrix $M$. Thus $\mathbb{P}^9$ can be stratified by the rank of matrices: for $k = 1, \ldots, 4$, let us denote by $${\cal Q}_k := \{M {\rm\, such\, that\,} rk\, M \le k\}$$ the family of quadrics of rank less or equal than $k$.\\
We have a chain of closed inclusions $${\cal Q}_1 \subset {\cal Q}_2 \subset {\cal Q}_3 \subset {\cal Q}_4 = \mathbb{P}^9$$
The generic element of every ${\cal Q}_k$, more precisely the ones in ${\cal Q}_k \setminus {\cal Q}_{k - 1}$, has rank exactly $k$. In particular, the space ${\cal Q}_3$ is defined by the annihilation of the determinant, which is a section of $H^0 (\mathbb{P}^9, {\cal O} (4))$. Hence ${\cal Q}_3$ is an $8$-dimensional hypersurface of degree $4$ in $\mathbb{P}^9$.\\
The inner space ${\cal Q}_2$ consists of quadrics which are made up of two hyperplanes, counted with multiplicities. Hence we have a parametrization $$\pi: (\mathbb{P}^3)^* \times (\mathbb{P}^3)^* \to {\cal Q}_2 \subset \mathbb{P}^9$$
$$\pi ([F], [G]) := [FG]$$ The map $\pi$ is a double cover over its image, which is exactly ${\cal Q}_2$. Thus ${\cal Q}_2$ is a variety of dimension $6$, of degree equal to $$\Deg\, {\cal Q}_2 = \frac{1}{2} \pi^* H^6 = 10$$ where $H$ is an hyperplane section of $\mathbb{P}^9$. The space ${\cal Q}_1$ is the branch locus of $\pi$, while the diagonal in $(\mathbb{P}^3)^* \times (\mathbb{P}^3)^*$ is the ramification locus.\\
This locus is a threefold of degree $8$ in $\mathbb{P}^9$, and it corresponds to the Veronese embedding of $(\mathbb{P}^3)^*$ through quadric forms, and it is the singular locus of the six-fold ${\cal Q}_2$.\\
To get a Coble surface, we choose a 5-dimensional linear subspace $\Lambda \subset \mathbb{P}^9$, $$\Lambda = V(L_1, L_2, L_3, L_4) \simeq \mathbb{P}^5$$ where the $L_i$'s are linear forms on $\mathbb{P}^9$. We consider the intersection surface $$S := \Lambda \cap {\cal Q}_2$$ and its pre-image $$\tilde S := \pi^{-1} (S) = \pi^{-1} (\Lambda).$$ 
The surface $\tilde S \subset (\mathbb{P}^3)^* \times (\mathbb{P}^3)^*$ is defined by the system of bilinear symmetric equations \begin{eqnarray}\label{sysS}\tilde S = \{ ([F],[G]) {\rm \ such\ that\ } L_i (FG) =0 \, {\rm\, for\,} \, i = 1, 2, 3, 4\} \end{eqnarray}
If $\Lambda$ is general enough, we have $\Lambda \cap {\cal Q}_1 = \emptyset$. In this case, the system \eqref{sysS} has no solutions along the diagonal $\Delta$ of $(\mathbb{P}^3)^* \times (\mathbb{P}^3)^*$, both $\tilde S$ and $S$ are smooth, and the 2:1 cover $\pi : \tilde S \to S$ is unramified. The automorphism group of this cover is $\mathbb{Z}_2$, generated by the involution $i$ of $(\mathbb{P}^3)^* \times (\mathbb{P}^3)^*$ which switches the factors. In $(\mathbb{P}^3)^* \times (\mathbb{P}^3)^*$, the surface $\tilde S$ is defined by the four bilinear equations of \eqref{sysS}, which are four sections of ${\cal O} (1,1)$. Hence, the canonical bundle of $\tilde S$ is $$K_{\tilde S} = (4\, {\cal O} (1,1) + K_{(\mathbb{P}^3)^* \times (\mathbb{P}^3)^*})|_{\tilde S} = {\cal O}_{\tilde S}$$ Together with the fact that $\tilde S$ is simply connected, this says that $\tilde S$ is a $K3$ surface, and $S = \pi (\tilde S) = \tilde S / i$ is an Enriques surface.\\
Now, suppose we are in the pathological case when $\Lambda \cap {\cal Q}_1 \ne \emptyset$. In this case, the surfaces $S$ inherits the singularities coming from ${\cal Q}_1$, and so $S$ becomes a singular surface, with a quartic point $p \in S \cap {\cal Q}_1$. Now let us choose a hyperplane $$H \subset \Lambda$$ $$H \simeq \mathbb{P}^4$$ such that $p \notin H$, and consider the projection with center $p$, $$\pi_p : S \setminus \{p\} \subset \Lambda \setminus \{p\} \to H$$ We denote by $$\tilde \pi_p := {\rm Bl}_p S \to H$$ the resolution of the indeterminacy of $\pi$ at $p$.\\
\begin{definition}
The image surface $X := {\rm\, Im\, } \pi_p \subset \mathbb{P}^4$ is called Bordiga surface.
\end{definition}
\begin{proposition}
The linear system $|H + K_X|$ has projective dimension $2$, and it contracts $10$ smooth rational $(-1)$-curves. Thus, the Bordiga surface $X \subset \mathbb{P}^4$ is isomorphic to the blow-up of $\mathbb{P}^2$ in 10 base-points.\\
If $E \subset Bl_p S$ is the exceptional curve over the singular point $p$, then its image $$C := \tilde \pi (E) \subset X$$ is the proper transform of a sextic curve $C \subset \mathbb{P}^2$, having multiplicity $2$ at all the $10$ base points. As a consequence, the class $C \in \Pic (X)$ is given by $$[\tilde C] = 6 L - 2 E_1 - \cdots - 2 E_{10} = -2 K_X$$
\end{proposition}
Conversely, if we start with a Coble surface $X' = Bl_{10} \mathbb{P}^2$, then the linear system $$|H| := |4 L  - E_1 - \cdots - E_{10}|$$ of quartics passing through all the base points determines an embedding $$X' \to \mathbb{P}^4$$ whose image is a Bordiga surface.\\

\subsection{Quintic Coble surfaces}\label{qcs}
Another way to see a Coble surface is using the linear system $$H := |6 L - 2 E_1 - \cdots - 2 E_7 - E_8 - E_9 - E_{10}|$$ We have $$h^0 (X, H) = 4$$ and $$H^2 = 5$$ so that the pair $(X, H)$ determines a birational map $\phi: X \to \mathbb{P}^3$ onto a quintic hypersurface $\overline X$. For each permutation $(i, j, k)$ of the indices $(8, 9, 10)$, the exceptional divisor $E_i$ is sent biregularly onto a line $l_i$, because $$H E_i =1.$$ Moreover the cubic curve through $p_1, \ldots, p_7, p_i, p_j$ lives in the linear system $- K_X + E_k$, and covers with degree $2$ a line $\hat l_k \subset X'$, because $$H (- K_X + E_k) = 2$$ The line $\hat l_k$ meets both $l_i, l_j$ at the same point. The three lines $\hat l_8, \hat l_9, \hat l_{10}$ are double lines for $\overline X$, and they are concurrent at a point. Indeed, the decomposition $$H = - 2 K_X + E_8 + E_9 + E_{10} = (- K_X + E_i) + (- K_K + E_j ) + E_k$$ implies that $\hat l_i, \hat l_j, l_k$ are coplanar for every $i, j$. Hence $\overline X$ contains a tetrahedron $T \subset \mathbb{P}^3$ consisting of $6$ lines, and passes doubly through three of them. The three double lines share a common vertex of $T$, which is a triple point for $\overline X$. In a suitable choice of coordinates, the equation of such a quintic $\overline X$ is: \begin{eqnarray}\label{quinticoble}\alpha X_0 X_2^2 X_3^2 + \beta X_0 X_1^2 X_3^2 + \gamma X_0 X_1^2 X_2^2 + X_1 X_2 X_3 q = 0\end{eqnarray} with $\alpha, \beta, \gamma \in \complex$, and $q = q (X_0, X_1, X_2, X_3)$ is a quadratic form on $\mathbb{P}^3$.\\
We write $L_{i, j}$ to denote the line $X_i = X_j = 0$. The three lines $l_{0, 1}, l_{0, 2}, l_{0, 3}$ are double lines, and they meet at the triple point $[1, 0, 0, 0]$. The plane $X_0 = 0$ cuts the three simple edges $l_{1, 2}, l_{2, 3}, l_{1, 3}$, plus the plane quadric defined by $q = 0, X_0 = 0$, which is the image of $C$.
\\
Now we want to prove the opposite, that is, the normalization of a generic quintic surface $\overline X$ defined as in \eqref{quinticoble} is actually a Coble surface.\\
First note that we can assume $$\alpha, \beta, \gamma \ne 0$$ and hence we can perform the diagonal change of variables $$X_1 = \sqrt{\alpha} X_1'$$ $$X_2 = \sqrt{\beta} X_2'$$ $$X_3 = \sqrt{\gamma} X_3'$$ which turns Equation \eqref{quinticoble} in \begin{eqnarray}\label{newquinticoble} X_0 X_2^2 X_3^2 + X_0 X_1^2 X_3^2 + X_0 X_1^2 X_2^2 + X_1 X_2 X_3 q = 0 \end{eqnarray}
Now let $\nu : X \to \overline X$ be its normalization, and assume $X$ is smooth.\\
\begin{definition}
Let $\overline X$ be a quintic surface inside $\mathbb{P}^3$ defined as in \eqref{newquinticoble}. We denote by $L_{i, j}$ the line $$L_{i, j} := \{X_i = X_j = 0\}$$ and by $E_{i, j}$ the corresponding divisorial pre-image in $X$.\\
Moreover, let $\overline C \subset \overline X$ the plane conic defined by $$\overline C := \{X_0 = 0, q (0, X_1, X_2, X_3) = 0\}$$ and let $C$ be its pre-image in $X$.
\end{definition}
Note that, if $X$ is generic, then the unique singular points of $\overline X$ in $L_{0, 1}, L_{0, 2}$, $L_{0, 3}$ are the three vertices $[0,1,0,0], [0,0,1,0], [0,0,0,1]$. Consequently, also $E_{0,1}, E_{0, 2}, E_{0, 3}$ are still smooth rational curves. For the same reason, also $C$ is a smooth rational curve in $X$.\\
On the converse, a local computation, in complete analogy with the case of Enriques surfaces, shows that the three lines $L_{1, 2}, L_{1, 3}, L_{2, 3}$ consist of double points of $\overline X$, and each of them contains $4$ pinch points. Hence, the three curves $E_{1, 2}, E_{1, 3}, E_{2, 3}$ are smooth genus $1$ curves.
\begin{proposition}
The divisor $4 L_{1, 2} + 4 L_{1, 3} + 4 L_{2, 3} - \overline C$ lies in the linear system ${\cal O}_{\overline X} (2)$, so it is a Cartier divisor. Its pull-back on $X$ equals $\nu^* (4 L_{1, 2} + 4 L_{1, 3} + 4 L_{2, 3} - \overline C) = 2 E_{1,2} + 2 E_{1, 3} + 2 E_{2, 3} - C$
\end{proposition}\label{cartier}
{\it Proof:} Let $$H_i \subset \mathbb{P}^3$$ be the coordinate hyperplanes defined by $$H_i = \{X_i = 0\}$$ for $i = 0, 1, 2, 3$. The $H_i$'s are Cartier divisor on $\mathbb{P}^3$, and so are their restrictions on $\overline X$. But clearly $$(H_0)|_{\overline X} = L_{0, 1} + L_{0, 2} + L_{0, 3} + \overline C$$ $$(H_1)|_{\overline X} = L_{0, 1} + 2 L_{1, 2} + 2 L_{1, 3} $$ $$(H_2)|_{\overline X} = L_{0, 2} + 2 L_{1, 2} + 2 L_{2, 3}$$ $$(H_3)|_{\overline X} = L_{0, 3} + 2 L_{1, 3} + 2 L_{2, 3}$$ Let $D$ be the following divisor on ${\overline X}$: $$D := 4 L_{1, 2} + 4 L_{1, 3} + 4 L_{2, 3} - \overline C,$$ so that we have $$D = (H_1 + H_2 + H_3 - H_0)|_{\overline X}$$ and the right-hand term lives in $|{\cal O}_{\overline X} (2)|$. The pull-back $\nu^*D$ clearly takes the form $\nu^*D = m (E_{1, 2} + E_{1, 3} + E_{2, 3}) - C$ for some $m > 0$. The equality $m = 2$ follows since $\nu_* E_{i, j} = 2 L_{i, j}$. $\square$
\begin{corollary}\label{quasiraz}
The surface $X$ satisfies $h^0 ({\cal O}_X (K_X)) = h^0 ({\cal O}_X (2 K_X)) = h^0 ({\cal O}_X (- K_X)) = 0$, and $C \in |- 2 K_X|$. 
\end{corollary}
{\it Proof:} We apply the formula $$K_X = \nu^* K_{\overline X} - (E_{1, 2} + E_{1, 3} + E_{2, 3})$$ By the adjunction formula $$K_{\overline X} = {\cal O}_{\overline X} (1),$$ so by Proposition \ref{cartier} $${\cal O}_X (2 K_X) = \nu^* {\cal O}_{\overline X} (2) - 2 (E_{1, 2} + E_{1, 3} + E_{2, 3}) = {\cal O}_X (- C).$$ This immediately proves that $2 K_X$ (and so $K_X$) is noneffective, while $- 2K_X$ is.\\
It only remains to show that $h^0 (- K_X) = 0$, but by Serre duality we know that $$h^0 ({\cal O}_X (- K_X)) = h^2 ({\cal O}_X (2 K_X)) = h^2({\cal O}_X (- C)).$$ The short exact sequence $$0 \to {\cal O}_X (- C) \to {\cal O}_X \to {\cal O}_C \to 0$$ gives a long exact sequence $$0 \to H^2 ({\cal O}_X (- C)) \to H^2 ({\cal O}_X).$$ Using again Serre duality and the first part of this statement, $$h^2 ({\cal O}_X) = h^0 ({\cal O}_X (K_X)) = 0.$$ $\square$\\
The next step is to compute some intersection products in $X$.
\begin{proposition}\label{somep}
The intersection products are $C^2 = - 4, E_{0, i} E_{0, j} = -\delta_{i, j}, C E_{0, i} = 2$. In particular, the Picard rank $rk \Pic (X)$ and the second Betti number $b_2 (X)$ are greater than or equal to $4$.
\end{proposition}
{\it Proof:} The curves $\overline C, L_{0, i}$ are smooth rational curves in $\overline X$, and so are their pull-backs $C, E_{0, i}$ in $X$. In particular, $$C^2 + C K_X = -2$$ and by Corollary \ref{quasiraz} we find $$C^2 + C K_X = \frac{1}{2} C^2$$ which gives $$C^2 = - 4.$$ The line $L_{0, i}$ meets the conic $\overline C$ at two smooth points of $\overline X$, so $C E_{0, i} = 2$. Again, the adjunction formula and Corollary \ref{quasiraz} give $$E_{0, i}^2 = -1.$$ Finally, if $\overline X$ is generic enough, none of the intersection points $L_{0, i} \cap L_{0, j}$ is a pinch point. Indeed, look for example at $$L_{0, 1} \cap L_{0, 2} = [0, 0, 0, 1],$$ which lies on the double line $L_{1, 2} = \{X_1 = X_2 = 0\}$. We consider the equation \ref{newquinticoble}, putting in evidence the terms of minimal degree in the variables $X_1, X_2$. Such terms are: \begin{eqnarray}\label{pinch} (X_0 X_3^2) X_1^2 + (X_0 X_3^2) X_2^2 + (X_3 \hat q(X_0, X_3)) X_1 X_2,\end{eqnarray} where $$\hat q(X_0, X_3) := q (X_0, 0, 0, X_3)$$ is the part of $q$ which is independent from $X_1, X_2$. The discriminant $\Delta$ of Equation \ref{pinch} is $$\Delta (X_0, X_3) = (X_3^2) (\hat q^2 - 4 X_0^2 X_3^2).$$ The external factor $X_3^2$ corresponds to the intersection of $L_{1, 2}$ with the plane $H_3$, which is the triple point $[1, 0, 0, 0]$. The four solutions of $$\hat q^2 - 4 X_0^2 X_3^2 = 0$$ are the four pinch points along $L_{1, 2}$. In particular, the point $[0, 0, 0, 1]$ is a pinch point if and only if $\hat q (0, 1) = q (0, 0, 0, 1) = 0$. For a generic equation in the form \ref{newquinticoble}, this equality does not hold, so the double point $[0, 0, 0, 1]$ has two distinct pre-images in $E_{0, 1}, E_{0, 2}$, hence these are two disjoint curves. An analogous argument holds of course for $[0, 0, 1, 0], [0, 1, 0, 0]$.\\
Since the matrix of these intersection products in non-degenerate, this states the linear independence of $C, E_{0,1}, E_{0, 2}, E_{0, 3}$ in $\Pic(X)$ and $H^2 (X, \complex)$. $\square$ 
\begin{lemma}
The surface $X$ is a Coble surface.
\end{lemma}
{\it Proof:} We already know from Proposition \ref{quasiraz} that $h^0 ({\cal O}_X (K_X)) = h^0 ({\cal O}_X (2 K_X)) = 0$ and $h^0({\cal O}_X (- 2 K_X)) = 1$.\\
Let $q = h^1 ({\cal O}_X)$ be the irregularity of $X$. We apply the Noether formula, see for example \cite{Huybrechts2005}: $$\chi ({\cal O}_X) = \frac{\chi_{top} (X) + K_X^2}{12}$$ Applying together Poincare duality and the Corollary \ref{quasiraz}, the left-hand term equals $$\chi ({\cal O}_X) = 1 - q + h^0 ({\cal O}_X (K_X)) = 1 - q$$ Meanwhile, the right-hand term is $$\frac{\chi_{top} (X) + K_X^2}{12} = \frac{b_0 - b_1 + b_2 - b_3 + b_4 - 1}{12}$$ where the $b_i$'s are the Betti numbers of $X$. We used the equality $K_X^2 = -1$ which comes from Proposition \ref{somep}. By Serre duality $b_4 = b_0 = 1$ and $b_1 = b_3$. The Hodge decomposition $$H^1 (X, \complex) = H^{1, 0} (X) \oplus H^{0, 1} (X)$$  gives $b_1 = 2 q$. Putting all together, we find $$1 - q = \frac{1 - 4 q + b_2}{12}$$ that is $$11 - 8 q = b_2$$ But Proposition \ref{somep} states $$b_2 \ge 4$$ which is equivalent to saying $$q = 0$$ and $$b_2 = 11.$$ $\square$\\
\\
Note that the subspace $V \subset H^0 ({\cal O}_{\mathbb{P}^3} (5))$ of forms as in \eqref{quinticoble} has dimension $13$. Meanwhile, the subgroup $G \subset GL (4)$ preserving $V$ consists of invertible matrices, which fix the triple point $[1, 0, 0, 0]$ and permute the vertices $[0, 1, 0, 0], [0, 0, 1, 0], [0, 0, 0, 1]$. The dimension of $G$ is $4$, since it fits in a short exact sequence of groups $$\mathbb{1} \to D \to G \to {\cal S}_3 \to \mathbb{1}$$ Here $D$ consists of all $4 \times 4$ invertible diagonal matrices, and the discrete symmetric group ${\cal S}_3$ acts by permutation on $[0, 1, 0, 0], [0, 0, 1, 0], [0, 0, 0, 1]$.  Thus the quotient $V/G$ has still dimension $13 - 4 = 9$, the same of the moduli space of Coble. 



Suppose we start with a quintic surface $X \subset \mathbb{P}^3$ given by the equation: $$\alpha X_0 X_2^2 X_3^2 + \beta X_0 X_1^2 X_3^2 + \gamma X_0 X_1^2 X_2^2 + X_1 X_2 X_3 q(X_0, X_1, X_2, X_3) = 0$$ with $$q = \sum_{0 \le i \le j \le 3} \lambda_{i, j} X_i X_j$$ Now consider the birational involution of $\mathbb{P}^3$ given by: $$i [X_0, X_1, X_2, X_3] = [\frac{1}{X_0}, \frac{1}{X_1}, \frac{1}{X_2}, \frac{1}{X_3}] = [X_1 X_2 X_3, X_0 X_2 X_3, X_0 X_1 X_3, X_0 X_1 X_2]$$ Under this transformation the equation of $X$ becomes: $$(X_0^3 X_1^2 X_2^2 X_3^2)* (\lambda_{0, 0} (X_1 X_2 X_3)^2 + \lambda_{1, 1} (X_0 X_2 X_3)^2 + \lambda_{2, 2} (X_0 X_1 X_3)^2 + $$ $$+ \lambda_{3, 3} (X_0 X_1 X_2)^2 + X_0 X_1 X_2 X_3 (\alpha X_1^2 + \beta X_2^2 + \gamma X_3^2 + \sum_{0 \le i < j \le 3} \hat \lambda_{i, j} X_i X_j)) = 0$$ where the coefficients $\hat \lambda_{i, j}$ are defined as $$\hat \lambda{i, j} = \lambda_{h, k} {\rm\, if\,} \{i, j\} \cup \{h, k\} = \{0, 1, 2, 3\}$$ If we cut out the initial factor $X_0^3 X_1^2 X_2^2 X_3^2$ we finally arrive at an expression of the form $$\lambda_{0, 0} (X_1 X_2 X_3)^2 + \lambda_{1, 1} (X_0 X_2 X_3)^2 + \lambda_{2, 2} (X_0 X_1 X_3)^2 + \lambda_{3, 3} (X_0 X_1 X_2)^2 +$$ $$+ X_0 X_1 X_2 X_3 \hat q (X_0, X_1, X_2, X_3) = 0$$ which is the expression of an Enriques sextic, with the additive $1$-codimensional condition that $$\hat q(1, 0, 0, 0) = 0.$$
\begin{proposition}
The linear system on the Coble surface $X$ which realizes the birational map $X \to \mathbb{P}^3$ onto an Enriques sextic with a quartic point is $|H'| = |9 L - 3 E_1 - \cdots - 3 E_7 - 2 E_8 - 2 E_9 - 2 E_{10}|$.
\end{proposition}
{\it Proof:} By definition, we have $$|H| = |6 L - 2 E_1 - \cdots - 2 E_7 - E_8 - E_9 - E_{10}|$$ The cubo - cubic involution of $\mathbb{P}^3$ is defined by the system of cubics containing the tetrahedron $T$, hence $$H' = 3 H - D$$ where $D$ is the preimage in $X$ of $T$. But we saw that $D$ is given by the union of $E_8, E_9, E_{10}$ and the three plane cubics $- K_X + E_8, - K_X + E_9, - K_X + E_{10}$, thus $$D = 9 L - 3 E_1 - \cdots - 3 E_7 - E_8 - E_9 - E_{10}$$ Finally, we find $$H' = 9 L - 3 E_1 - \cdots - 3 E_7 - 2 E_8 - 2 E_9 - 2 E_{10}$$ $\square$\\
Note that $$H = H' + K_X$$ that is, $H$ is the adjoint linear system of $H'$.

\subsection{Nodal Coble cubic surfaces}
Let's consider three points $p_1, p_2, p_3 \in \mathbb{P}^2$, and three lines $l_1, l_2, l_3 \subset \mathbb{P}^2$ such that $$p_i \in l_i, \quad p_i \notin l_j \quad {\rm for\,} \quad i \ne j \in \{1,2,3\}$$ If we look at the cubics in $\mathbb{P}^2$ passing through $p_i$ and tangent to $l_i$, we get a subspace  $W \subset H^0 ({\cal O}_{\mathbb{P}^2} (3))$ with $\Dim\, |W| = 3$. The corresponding rational map $|W| : \mathbb{P}^2 \dasharrow \mathbb{P}^3$ has three base points at the $p_i$'s. The blow up $\tilde{\mathbb{P}}^2$ of $\mathbb{P}^2$ at these points does not resolve the indeterminacy of the map, since we have still three base points $p'_i$ on the exceptional divisors $E_i$, for $i = 1, 2, 3$. If one performs a second blow - up on the $p'_i$'s, the result is a surface $\tilde {\tilde {\mathbb{P}}}^2$, with a regular map $\tau: \tilde{ \tilde{ \mathbb{P}}}^2 \to \mathbb{P}^3$.\\
If we denote by $\overline E_i$ the proper transform of $E_i$ in $\tilde{ \tilde{ \mathbb{P}}}^2$, and by $F_i$ the new exceptional divisors of the second blow - up, then $$\tau^* {\cal O}_{\mathbb{P}^3} (1) = - K_{\tilde{\tilde{\mathbb{P}}}^2} = 3 L - \sum_{i = 1}^3 \overline E_i - 2 \sum_{i = 1}^3 F_i$$ We have the relations $$L^2 = 1, L \overline E_i = L F_i = 0, \overline E_i^2  = - 2, E_i F_i = 1, F_i^2 =  - 1$$ so that $$(\tau^* {\cal O}_{\mathbb{P}^3} (1))^2 = 3$$ and $$\overline E_i \tau^* {\cal O}_{\mathbb{P}^3} (1) = 0$$ this means that the image of $\tau$ is a cubic surface $S \subset \mathbb{P}^3$, and the $\overline E_i$'s are three smooth rational $(-2)$-curves, which are contracted by $\tau$ to three nodes of $S$.\\
If $q = [q_0, q_1, q_2, q_3] \in S$ is a smooth point, and $H \subset \mathbb{P}^3$ is a plane not containing $q$, then we have a projection $\pi_q : S \setminus q \to H$, which is a map of degree two. The ramification locus is the curve $R$ made up of points $p$ such that the line $\overline {q, p}$ is tangent to $S$ at $p$. If $F \in H^0 ({\cal O}_{\mathbb{P}^3} (3))$ is the equation of $S$, then $R$ is given geometrically by $$F = 0,$$ $$\sum_{i = 0}^3 q_i \frac{\partial F}{\partial z_i} = 0$$ This is a curve of degree $6$, described as the complete intersection of a cubic and a quadric surface in $\mathbb{P}^3$. By construction $R$ has a node at $q$, so its projection $\pi_q (R)$ on $H$ is a curve of degree $6 - 2 = 4$. Moreover, $R$ of course contains the three nodes of $S$, where the partial derivatives simultaneously vanish. So $\pi_q (R)$ is a three - nodal quartic curve, and hence it is rational.

\subsection{A quartic Coble in $\mathbb{P}^3$}
Suppose $X$ is a Coble surface, realized as a blow up of ten points $p_1, \ldots, p_{10} \in \mathbb{P}^2$, with exceptional curves $E_1, \ldots, E_{10} \subset X$. Let $|H|$ be the following linear system: $$|H| := |4 L - E_1 - \cdots - E_8 - 2 E_{10}|$$ the linear system of quartics passing through $p_1, \ldots, p_8$ and nodal at $p_{10}$. We have $$\Dim\, |H| = 3$$ and $$H^2 = 4$$ so that $|H|$ defines a map $f : X \to \mathbb{P}^3$ over a quartic surface $Y \subset \mathbb{P}^3$. The cubic curve $\hat F_9$ passing through $p_1, \ldots, p_8, p_{10}$ lives in the class $F_9 \in |3 L - E_1 - \cdots - E_8 - E_{10}|$ so it satisfies $$H \hat F_9 = 2.$$ Moreover, we have $$h^0 (X, H - \hat F_9) = h^0 (X, L - E_{10}) = 2,$$ so that there are two independent planes in $\mathbb{P}^3$ containing the curve $$\hat l := f(\hat F_9).$$ So $\hat l$ is a line, and the restricted map $f : \hat F_9 \to \hat l$ has degree $2$. This means that $\hat l$ is a double line for the surface $Y$. Hence, given a general plane $\Gamma \subset \mathbb{P}^3$, the intersection $X \cap \Gamma$ is a curve of degree $4$ with a double point in $\hat l \cap H$. Thus $X \cap \Gamma$ has algebraic genus $2$, as well as the general member of $|H|$.

\newpage
\section{Coble conjecture}
Coming back to Coble's original construction, from now on we will assume that the number $n$ of Proposition \ref{comp} equals $1$, so that the Coble curve $C \subset X$ is a copy of $\mathbb{P}^1$. Note that every automorphism $\phi \in \Aut (X)$ must satisfy $$\phi_* (K_X) = K_X \in \Pic (X)$$ and consequently, $$\phi (C) = C$$ Thus, there exists a well - defined restriction map $$\rho : \Aut  (X) \to \Aut  (C) \simeq \mathbb{P} GL (2, \complex)$$ 
\begin{conjecture}[Coble, \cite{OriginalCoble1919}] Is $\Ker\, \rho$ trivial for a general Coble surface $X$ ? What is the image $\Im\, \rho \subset \mathbb{P} GL (2)$?
\end{conjecture}

\subsection{Pompilj's method}
We refer to \cite{Coble1939}, \cite{Dolg2019}, \cite{Pompilj1940} for the content of this Section.\\
Let $X$ be a Coble surface with one irreducible boundary component $C \in |-2 K_X|$. In this chapter, we will reconstruct the path made by Pompilj in an article of 1938 in an attempt to provide a non - trivial element in the kernel of the restriction map $\rho: \Aut  (X) \to \Aut  (C) \simeq PGL (2, \complex)$. The idea was the following: we represent $X$ as a blow-up of ten points $p_1, \ldots, p_{10} \in \mathbb{P}^2$. Let $E_i \subset X$ be the exceptional divisor associated to the point $p_i$. For every $i=1, \ldots, 10$, we consider the linear system $|6 L - 2 E_1 - \cdots - 2 E_{i - 1} - 2 E_{i + 1} - \cdots- 2 E_{10}|$. Its sections consist of the strict transforms of sextics with nodes at all $p_j$'s but $p_i$. One can prove that $$h^0 ({\cal O}_X (6 L - 2 E_1 - \cdots - 2 E_{i - 1} - 2 E_{i + 1} - \cdots- 2 E_{10})) = 2$$ and it is easy to show two generators. One of them has the form $2 C_i$, where $C_i$ is the unique cubic curve through $p_1, \ldots, p_{i-1}, p_{i + 1},\ldots, p_{10}$, while the other one has the form $C + 2 E_i$. It is also immediate to show that the self intersection of a curve in $|6 L - 2 E_1 - \cdots - 2 E_{i - 1} - 2 E_{i + 1} - \cdots- 2 E_{10}|$ is $0$, and the arithmetic genus is $1$. Hence, we have just defined $10$ elliptic fibrations $\pi_i : X \to \mathbb{P}^1$, and each of them admits a unique double fiber $2 C_i$.\\
We need to specialize our attention to three of these points, so we set $$A:= p_8, B:= p_9, C:=p_{10}$$ and let $$E_A, E_B, E_C \subset X$$ be the corresponding exceptional divisors, and consequently $$\pi_A:= \pi_8: X \to \mathbb{P}^1, \pi_B := \pi_9 : X \to \mathbb{P}^1, \pi_C := \pi_{10} : X \to \mathbb{P}^1$$ the associated fibrations.\\

\begin{definition}
Let $F_A$ be a smooth fiber of the fibration $\pi_A$, which is an elliptic curve. On $F_A$ the divisor $(E_B - E_C)|_{F_A}$ has degree zero, so for every point $p \in F_A$ there is a unique point $T_A (p) \in F_A$ satisfying $$T_A (p) - p = (E_B - E_C)|_{F_A} {\rm \ in\ } \Pic (F_A)$$ Letting $F_A$ vary among the smooth fibers of $\pi_A$, this defines a birational morphism $T_A : X \dasharrow X$.
In a similar way, we define two morphisms $T_B, T_C : X \dasharrow X$ as $$T_B (p) - p = (E_C - E_A)|_{F_B} {\rm \ in\ } \Pic (F_B)$$ and $$T_C (p) - p = (E_A - E_B)|_{F_C} {\rm \ in\ } \Pic (F_C)$$ where $F_B, F_C$ are, respectively, smooth fibers of $\pi_B, \pi_C : X \to \mathbb{P}^1$, and (respectively) $p \in F_B, F_C$.
\end{definition}
Up to now, these are just birational transformations, which fail to be well - defined on the singular, or nonreduced, fibers of $\pi_A, \pi_B, \pi_C$. Nonetheless, this can be adjusted:
\begin{proposition}\cite{Coble1939}
The birational morphisms $T_A, T_B, T_C$ extend to biregular automorphisms of $X$.
\end{proposition}
Now we need a very simple observation:
\begin{proposition}
If $p_1,\ldots, p_{10}$ are points in $\mathbb{P}^2$ such that there exists a reduced sextic curve $\overline C$ nodal at the $p_i$'s, then the Bertini involution $\sigma$ associated to any eight of these points fixes the remaining two.\\
In particular, $\sigma$ lifts to a biregular involution on the Coble surface $X := Bl_{p_1, \ldots, p_{10}} \mathbb{P}^2$.
\end{proposition}\label{bertsempre}
{\it Proof:} Let $Y$ be the blow up of $\mathbb{P}^2$ at $p_1, \ldots, p_8$. The Bertini involution $\sigma: Y \to Y$ preserves the curves in the linear systems $|- K_Y|, |- 2 K_Y|$. In particular, let $D_1 \in |- K_Y|$ be the proper transform of the unique plane cubic passing also through $p_9$, and let $D_2 \in |- 2 K_Y|$ be the strict transform of the sextic $\overline C$ having the additional node at $p_9$. The divisor $D_1$ is smooth at $p_9$, while $D_2$ has a node, and these divisors have no common components, because otherwise $D_2 = 2 D_1$, which cannot happen because $\overline C$ is reduced. As a consequence $$D_1 D_2 = (- K_Y) (- 2 K_Y) = 2$$ which implies that $D_1, D_2$ intersect at $p_9$ with multiplicity $2$, and they have no further points in common. Since both these curves are preserved by $\sigma$, it follows that $\sigma (p_9) = p_9$, and similarly $\sigma (p_{10}) = p_{10}$.\\
As a consequence, $\sigma$ lifts to a biregular involution of the Coble surface $X = Bl_{p_9, p_{10}} Y = Bl_{p_1, \ldots, p_{10}} \mathbb{P}^2$, which switches the direction of $D_2$ at both $p_9, p_{10}$. $\square$ 
\begin{proposition}\label{PompBert}\cite{Coble1939} \cite{Dolg2019}
Let $i_A, i_B, i_C$ be the Bertini involutions associated to the points $p_1, \ldots, p_7$, and choosing $A, B, C$ as the eighth point. Then the three automorphisms $T_A, T_B, T_C$ satisfy: $$T_A = i_B \circ i_C, T_B = i_C \circ i_A, T_C = i_A \circ i_B$$
\end{proposition}
Pompilj claimed the following fact: \begin{eqnarray}\label{fake} (T_C \circ T_B \circ T_A)|_C = \mathbb{1}_C \end{eqnarray}
This equality can be restated in terms of Proposition \ref{PompBert} claiming that $$(i_A \circ i_B \circ i_C)|_C^2 = \mathbb{1}$$
However, Coble proved the following result:
\begin{theorem}\label{Coblecorrez} \cite{Dolg2019}
For the generic Coble surface $X \in {\cal M}_{Co}$ equality \eqref{fake} is false. The family of Coble $X$ such that \eqref{fake} actually holds is a divisor in ${\cal M}_{Co}$.
\end{theorem}
Now we will prove the previous Theorem, but it needs a number of technical results.
\begin{proposition}\label{jacq}
Let $F, G \in H^0 ({\cal O}_{\mathbb{P}^1} (2))$ two linearly independent quadric forms, and let $\sigma$ be the involution associated to the $g_2^1$ generated by $F, G$. Then the solution of the quadric form $J (F, G) := \frac{\partial F}{\partial u} \frac{\partial G}{\partial v} - \frac{\partial F}{\partial v} \frac{\partial G}{\partial u}$ are the two fixed points of $\sigma$.
\end{proposition}
{\it Proof:} The polynomials $F, G$ cannot share a common root $a \in \mathbb{P}^1$, otherwise by hypothesis they would also share its image $\sigma (a)$, contradicting the linear independence of these forms. As a consequence, the map $$h: \mathbb{P}^1 \to \mathbb{P}^1, h(u, v):= ( F(u, v), G (u , v))$$ is a well defined regular double cover. The corresponding deck involution is $\sigma$ itself by construction. The fixed locus of $\sigma$ is the ramification locus, where $h$ fails to be a local isomorphism, and hence is given by the annihilation of the Jacobian determinant \begin{eqnarray*}J (F, G) = \Det\, \left(\begin{array}{lr}
\displaystyle \frac{\partial F}{\partial u} & \displaystyle \frac{\partial F}{\partial v}\\\\
\displaystyle \frac{\partial G}{\partial u} & \displaystyle \frac{\partial G}{\partial v}
\end{array}
\right) = 0\,.\qquad \begin{array}{lr}
\\\\
\\\\
\square
\end{array}
\end{eqnarray*} 
\\
\begin{definition}\label{polar}
Let $X \subset \mathbb{P}^m$ be a hypersuface, defined by a homogeneous equation $F (Z_0, \ldots, Z_m) = 0$ in the projective coordinates $[Z_0, \ldots, Z_m]$. If $p = [p_0, \ldots, p_m]$ is any point in $\mathbb{P}^m$, the hypersurface given by $$P:= \{[z] \in \mathbb{P}^m {\rm\, such\, that\, } \sum_{i = 0}^m p_i \frac{\partial F}{\partial Z_i} ([z]) = 0\}$$ is called polar hypersurface to $X$, with centre $p$.
\end{definition}
\begin{remark}
Note that, despite the name, the polar hypersurface $P$ with centre $p$ does not necessarily contain $p$. Indeed, by Euler's formula, $p \in P$ if and only if $p \in X$.
\end{remark}
In particular, if $Q \subset \mathbb{P}^2$ is a smooth quadric, its polar hypersurfaces are lines. If the centre $p$ lies in $Q$, then the polar line is exactly the tangent at $Q$ in $p$. If $p \notin Q$, then there are two points $q_1, q_2$ such that the lines $\overline {q_i, p}$ are tangent to $Q$ at $q_i$. In this case, the polar line is exactly $\overline {q_1, q_2}$.\\
The following theorem is a standard result in plane geometry, due to Pascal.
\begin{theorem} [Pascal] \label{Pascal}\cite{Shafarevich2013}
Let consider six points $p_1, \ldots, p_6$ lying on a smooth conic $Q \subset \mathbb{P}^2$. Let consider an hexagon inscribed in $Q$ with vertices the $p_i$'s, with edges $l_1, l_2, l_3, m_1, m_2, m_3$, labeled so that every vertex $p_i$ belongs to exactly one of the $l_i$'s, and one of the $m_i$'s. Then the three points $l_1 \cap m_1, l_2 \cap m_2, l_3 \cap m_3$ are collinear.
\end{theorem}
{\it Proof:} Let identify the conic $Q$ with $\mathbb{P}^1$. Consider the restriction map $$H^0 ({\cal O}_{\mathbb{P}^2} (3 L)) \to H^0 ({\cal O}_{\mathbb{P}^1} (6))$$ and let $V \subset H^0 ({\cal O}_{\mathbb{P}^2} (3 L))$ be the pencil generated by the cubic forms $l_1 l_2 l_3$ and $m_1 m_2 m_3$. It has $9$ base points, namely the $6$ vertices, and the three points $l_i \cap m_i$. The image of $V$ in $H^0 ({\cal O}_{\mathbb{P}^1} (6))$ has rank $1$, because both the two generators of $V$ restrict to the divisor $p_1 + \cdots + p_6$. Hence there exist $[\lambda, \mu] \in \mathbb{P}^1$ such that the section $\lambda l_1 l_2 l_3 + \mu m_1 m_2 m_3$ restricts to the zero section in $H^0 ({\cal O}_{\mathbb{P}^1} (6))$, so the corresponding cubic splits as $Q + L$, where $L$ is a line which must necessarily contain the points $l_i \cap m_i$. $\square$
\begin{proposition}\label{lindipq}\cite{Dolg2019}
Let $\sigma_1, \sigma_2, \sigma_3$ three involutions of $\mathbb{P}^1$, and let $$G_1, G_2, G_3 \in H^0 ({\cal O}_{\mathbb{P}^1} (2))$$ the three quadric forms vanishing respectively on the pairs of fixed points of $\sigma_1, \sigma_2, \sigma_3$. If the composition $\sigma_1 \circ \sigma_2 \circ \sigma_3$ is still an involution, then $G_1, G_2, G_3$ are linearly dependent.
\end{proposition}
{\it Proof:} Let consider the Veronese embedding $v: \mathbb{P}^1 \to \mathbb{P}^2$, which identifies $\mathbb{P}^1$ as a smooth conic $Q \subset \mathbb{P}^2$. Then there are three points $p_1, p_2, p_3 \notin Q$, and a line $l \subset \mathbb{P}^2$ not containing any of the $p_i$'s, such that the induced involutions through $v$ on $Q$ correspond to the deck involutions of the double covers $\pi_{p_i} : Q \to l$, where $\pi_{p_i}$ is the projection with center $p_i$. Let $P_i$ be the polar lines to $Q$ with center $p_i$. The crucial point is that the $p_i$'s are collinear. This follows from Pascal's Theorem \ref{Pascal}, choosing a random point $x \in Q$ and considering the hexagon with edges 
\begin{eqnarray*}l_1& =& \overline {x, \sigma_1 (x)}\,,\  m_2 = \overline {\sigma_1 (x), \sigma_2 \sigma_1 (x)}\,,\ l_3 = \overline {\sigma_2 \sigma_1 (x), \sigma_3 \sigma_2 \sigma_1 (x)}\nonumber\\
m_1 &=& \overline {\sigma_3 \sigma_2 \sigma_1 (x), \sigma_1 \sigma_3 \sigma_2 \sigma_1 (x)}\,,\  l_2 = \overline {\sigma_1 \sigma_3 \sigma_2 \sigma_1 (x), \sigma_2  \sigma_1 \sigma_3 \sigma_2 \sigma_1 (x)}\,, \nonumber\\
m_3& =& \overline {\sigma_2  \sigma_1 \sigma_3 \sigma_2 \sigma_1 (x), x}\,.\end{eqnarray*}
Indeed, by construction, the point $p_i$ is exactly the intersection $$p_i = l_i \cap m_i$$ By Definition \ref{polar}, the $P_i$'s are linear combinations of the partials $\frac{\partial Q}{\partial Z_j}$, with coefficients the projective coordinates of $p_i$'s. Since the centers lie on the same line, the $P_i$ are linearly dependent, and so $G_i = v^* P_i$ are linearly dependent too. $\square$\\
\\
Now we have all the ingredients we need to prove Theorem \ref{Coblecorrez}:\\
{\it Proof of Theorem \ref{Coblecorrez}:} A $10$-nodal sextic plane curve $\overline C \subset \mathbb{P}^2$ is rational, hence there exists a regular birational parametrization $$\gamma: \mathbb{P}^1 \to \overline C \subset \mathbb{P}^2$$ We write $$\gamma (u, v) = [F_0 (u, v), F_1 (u, v), F_2 (u, v)]$$ where $F_0, F_1, F_2$ are linearly independent forms of degree $6$. Up to an automorphism of $\mathbb{P}^2$, we can assume that three nodes of $\overline C$ are placed at the points $[1, 0, 0], [0, 1, 0], [0, 0, 1]$. The curve $\overline C$ passes twice over $[0, 0, 1]$, so $F_0, F_1$ must share two common roots. The same holds for the pairs $(F_0, F_2)$ and $(F_1, F_2)$. Thus, it is not a restriction to assume that there exist $$A, B, C, G_0, G_1, G_2 \in H^0 ({\cal O}_{\mathbb{P}^1} (2))$$ such that the $F_i$'s factor as: $$F_0 = B C G_0, F_1 = A C G_1, F_2 = A B G_2$$
The polynomials $B, C$ are invariant under the Bertini involution $(i_A)|_C$, and the same is true for the pairs $(A, C), (A, B)$ with respect to $(i_B)|_C, (i_C)|_C$. Proposition \ref{jacq} implies that the three quadric Jacobian forms $$J (A, B) = \frac{\partial A}{\partial u} \frac{\partial B}{\partial v} - \frac{\partial A}{\partial v} \frac{\partial B}{\partial u}$$  $$J (A, C) = \frac{\partial A}{\partial u} \frac{\partial C}{\partial v} - \frac{\partial A}{\partial v} \frac{\partial C}{\partial u}$$ $$J (B, C) = \frac{\partial B}{\partial u} \frac{\partial B}{\partial v} - \frac{\partial C}{\partial v} \frac{\partial B}{\partial u}$$ vanish on the pairs of fixed points of $(i_C)|_C, (i_B)|_C, (i_A)|_C$ respectively, and Proposition \ref{lindipq} forces these Jacobians to be linearly dependent. But this defines a determinantal equation over the coefficients of $A, B, C$. $\square$\\
\\
If we write down explicitly $A, B, C \in H^0 ({\cal O}_{\mathbb{P}^1} (2))$, say $$A = a_0 u^2 + a_1 u v + a_2 v^2$$ $$B = b_0 u^2 + b_1 u v + b_2 v^2$$ $$C = c_0 u^2 + c_1 u v + c_2 v^2,$$ then  $$J (B, C) = (b_0 c_1 - b_1 c_0) u^2 + 2 (b_0 c_2 - b_2 c_0) u v + (b_1 c_2 - b_2 c_1) v^2$$ $$J (A, C) = (a_0 c_1 - a_1 c_0) u^2 + 2 (a_0 c_2 - a_2 c_0) u v + (a_1 c_2 - a_2 c_1) v^2$$ $$J (A, B) = (a_0 b_1 - a_1 b_0) u^2 + 2 (a_0 b_2 - a_2 b_0) u v + (a_1 b_2 - a_2 b_1) v^2$$ 
Let $M, N$ be the matrices $$M := \left(\begin{array}{lcr}
\displaystyle a_0 & \displaystyle a_1 & \displaystyle a_2\\
\displaystyle b_0 & \displaystyle b_1 & \displaystyle b_2\\
\displaystyle c_0 & \displaystyle c_1 & \displaystyle c_2
\end{array}
\right)$$

$$N := \left(\begin{array}{lcr}
\displaystyle b_0 c_1 - b_1 c_0 & \displaystyle b_0 c_2 - b_2 c_0 & \displaystyle b_1 c_2 - b_2 c_1\\
\displaystyle a_0 c_1 - a_1 c_0 & \displaystyle a_0 c_2 - a_2 c_0 & \displaystyle a_1 c_2 - a_2 c_1\\
\displaystyle a_0 b_1 - a_1 b_0 & \displaystyle a_0 b_2 - a_2 b_0 & \displaystyle a_1 b_2 - a_2 b_1
\end{array}
\right)$$
Up to a permutation of columns, and a factor $2$, we have $$N = (\Det\, M) M^{-1}$$ and hence $$\Det\, N = (\Det\, M)^2$$ Thus the linear dependence of the three Jacobian forms $J (A, B), J (A, C), J (B, C)$ is equivalent to the one of $A, B, C$. 
We now consider the product space $\mathbb{P}^2 \times \mathbb{P}^2$, where we put projective coordinates $[X_0, X_1, X_2], [Y_0, Y_1, Y_2]$. We denote $\pi_1$ the projection on the $X$-coordinates, and $\pi_2$ the projection on the $Y$-coordinates. Let $\lambda$ be the morphism $\lambda: \mathbb{P}^1 \to \mathbb{P}^2 \times \mathbb{P}^2$, $$\gamma(u ,v) := ([B C (u, v), A C (u, v), A B (u, v)], [G_0 (u, v), G_1 (u, v), G_2 (u, v)]).$$ For a generic choice of $A, B, C$, the morphism $\pi_1 \circ \lambda = [B C, A C, A B]$ is the normalization of a $3$-nodal quartic. The special case when the image of $\pi_1 \circ \lambda$ is a conic corresponds to a polynomial $H(X_0, X_1, X_2) \in H^0 ({\cal O}_{\mathbb{P}^2} (2))$ such that $H \circ \pi_1 \circ \lambda = 0$. Since the points $[1, 0, 0], [0, 1, 0], [0, 0, 1]$ must belong to $V(H)$, $H$ must have the form $H = \alpha X_1 X_2 + \beta X_0 X_2 + \gamma X_0 X_1,$ so the relation $$H \circ \pi_1 \circ \lambda = 0$$ becomes $$A B C (\alpha A + \beta B + \gamma C) = 0.$$ Thus the image of $\pi_1 \circ \lambda$ is a quartic if and only if $A, B, C$ are linearly independent. Meanwhile, the composition $\pi_2 \circ \lambda$ is an isomorphism on a smooth plane conic.\\
Remember that there exists a natural Segre embedding $\mathbb{P}^2 \times \mathbb{P}^2 \subset \mathbb{P}^8$, defined by the complete linear system $|{\cal O}_{\mathbb{P}^2 \times \mathbb{P}^2} (1, 1)|$. With respect to this embedding, the image $\lambda (\mathbb{P}^1)$ has degree $6$, since the two factors $\pi_1 \circ \lambda$ and $\pi_2 \circ \lambda$ have degree $4$ and $2$ respectively. Thus there exist a linear subspace $\mathbb{P}^6 \subset \mathbb{P}^8$ containing $\lambda(\mathbb{P}^1)$. The intersection $$S := \mathbb{P}^6 \cap (\mathbb{P}^2 \times \mathbb{P}^2)$$ is a surface. It projects birationally on both the factors, hence $S$ is a rational surface. 
Let $\rho: \mathbb{P}^2 \times \mathbb{P}^2 \dasharrow \mathbb{P}^2$ be the map $$\rho ([X_0, X_1, X_2], [Y_0, Y_1, Y_2]) := [X_0 Y_0, X_1 Y_1, X_2 Y_2]$$ Then the original morphism $\mathbb{P}^1 \to \mathbb{P}^2$ given by $[B C G_0, A C G_1, A B G_2]$ factors as the composition $$\gamma = \rho \circ \lambda$$
If $A, B, C$ are linearly dependent, then we know that the morphism $$\pi_1 \circ \lambda, \pi_2 \circ \lambda: \mathbb{P}^1 \to \mathbb{P}^2$$ have degree $2, 1$ respectively on smooth plane conics. Thus we can think of $\lambda(\mathbb{P}^1)$ as a curve of type $(2, 1)$ inside $\mathbb{P}^1 \times \mathbb{P}^1$. In this case, the inclusion map $\mathbb{P}^1 \times \mathbb{P}^1 \to \mathbb{P}^2 \times \mathbb{P}^2$ is given by the cartesian product of two Veronese embeddings of degree $2$, hence the composite morphism $$\mathbb{P}^1 \times \mathbb{P}^1 \to \mathbb{P}^2 \times \mathbb{P}^2 \xrightarrow{\rho} \mathbb{P}^2$$ is induced by a net in the linear system $|{\cal O}_{\mathbb{P}^1 \times \mathbb{P}^1} (2, 2)|$.

\newpage
\section{Involutions on Coble surfaces}
We focus our attention to the case of the biregular involutions of $X$, so let $i \ne id_X$ satisfy $$i^2 = \mathbb{1}_X$$ The main results of this Section are two: if $X$ is a Coble surface with irreducible boundary curve $\{C\} = |- 2 K_X|$, then there are no involutions which pointwise fix $C$.\\
Moreover, if $X$ is also unnodal, then any involution is the lift of a Bertini involution.\\
We will need the following lemma:
\begin{lemma}[Dolgachev, Zhang]\cite{ZhangDolgachev2001}\label{Smoothfixedlocus}
Suppose $X$ is a smooth rational algebraic surface, and $\psi : X \to X$ a non trivial automorphism of finite prime order $p$. Then the fixed locus $\Fix(\psi)$ is a disjoint union of smooth curves and isolated fixed points, that is, it is smooth.
\end{lemma}

\subsection{Classifying involutions}
We follow the definition given by Bayle-Beauville in \cite{Bayle2000} and Dolgachev - Zhang in \cite{ZhangDolgachev2001} :
\begin{definition}\label{pairs}
Consider the set of all possible pairs $(X, i)$, where $X$ is a smooth rational projective surface, and $i$ is a non-trivial involution.\\
An equivariant morphism $\phi: (X, i) \to (X', i')$ is a (regular) birational morphism $\phi : X \to X'$ such that $i' \circ \phi = \phi \circ i$.\\
We say that a pair $(X, i)$ is minimal if every equivariant morphism $(X, i) \to (X', i')$ to any other pair $(X', i')$ is an isomorphism.
\end{definition}

\begin{proposition}\label{minc}\cite{Bayle2000}
A pair $(X, i)$ is minimal if and only if for every smooth rational $(-1)$-curve $E \subset X$, both relations $$i (E) \ne E$$ and $$i(E) \cap E \ne \emptyset$$ are true.
\end{proposition}

\begin{lemma}\label{minimalpairs}\cite{Bayle2000}
Assume $(X, i)$ is a minimal pair. Then one of the following cases is necessarily true:\\
i) There exists a smooth $\mathbb{P}^1$-fibration $f : X \to \mathbb{P}^1$, stable under $i$. In other words, there exists a non-trivial involution $\tau$ on $\mathbb{P}^1$, such that: $$\tau \circ f = f \circ i.$$\\
ii) There exist a conic fibration $f : X \to \mathbb{P}^1$ such that $$f \circ i = f$$ In other words, $i$ preserves the fibers of $f$. The smooth fibers are rational curves. Each singular fiber is a linear chain of two $(-1)$-curves attached at one point, and exchanged by $i$. The fixed locus $\Fix (i)$ is smooth and it has pure dimension $1$, and it is a bisection of $f$, ramified exactly at the singular points of the singular fibers. Moreover, $\Fix (i)$ splits as a union of two disjoint sections of $f$ if and only if $f$ has no singular fibers.\\
iii) $X = \mathbb{P}^2$, and $i$ is a projective linear involution.\\
iv) $X = \mathbb{P}^1 \times \mathbb{P}^1$, and $i (x, y) = (y, x)$ is the involution switching the two factors.\\
v) $X$ is a Del Pezzo surface of degree $2$, and $i$ is the Geiser involution.\\
vi) $X$ is a Del Pezzo surface of degree $1$, and $i$ is the Bertini involution.\\
Viceversa, a pair $(X, i)$ built as in case $i), \ldots, vi)$ is actually minimal, except the following ones:\\
Case i), when $X = \mathbb{F}_1$ is the first Hirzebruch surface, or\\
Case ii), when $X = \mathbb{F}_1$, or $X = Bl_{p_1, p_2, p_3} \mathbb{P}^2$, with $p_i$'s non collinear points, and $i$ is the De - Jonquieres involution of degree $2$ (the quadro-quadric involution centered at $p_1, p_2, p_3$).
\end{lemma}
Lemma \ref{minimalpairs} allows us to exclude any involution $i \in \Aut  (X)$ to satisfy $i|_C = C$, if $C$ is irreducible:
\begin{proposition}\label{nifc}
On a Coble surface $X$ with irreducible curve $C \in |- 2 K_X|$ there is no involution $i$ such that $i|_C = \mathbb{1}_C$. 
\end{proposition}
{\it Proof:} Assume such an involution $i$ actually exists: we first show that the pair $(X, i)$ must necessarily be minimal, using Proposition \ref{minc} and Lemma \ref{Smoothfixedlocus}.\\
Indeed, if $E \subset X$ is any $(-1)$-curve such that $i(E) = E$, then we can execute the blow down $\pi: X \to X'$, which comes together with an involution $i' : X' \to X'$. But the relation $$E K_X = -1$$ implies $$E C = 2,$$ so the curve $\pi (C)$ is singular and lies inside the fixed locus of $i'$, a contradiction. This forces $i(E) \ne E$, and the relation $i(E) \cap E \ne \emptyset$ is obvious, since $E, i(E)$ at least share the two points in $E \cap C$.\\
So the pair $(X, i)$ is minimal, and hence we lie inside one of the cases $i), \ldots, vi)$ of Lemma \ref{minimalpairs}. Of course $X$ is not isomorphic to $\mathbb{P}^2, \mathbb{P}^1 \times \mathbb{P}^1$, nor to a Del Pezzo or Hirzebruch surface, so the unique admissible case is ii). Thus we find a fibration $f : X \to \mathbb{P}^1$, whose fibers are rational curves preserved by $i$.\\ Let $F \in \Pic(X)$ be the class of a fiber: the rationality of $F$ gives $$F K_X = -2$$ and hence $$F C = 4,$$ so $C$ is a $4$-section of $f$. This is a contradiction, since Lemma \ref{minimalpairs} states that $\Fix(i)$ must be a bi-section. 
The contradiction follows from the initial assumption on the existence of such an involution $i$. $\square$\\
\\
We saw in Remark \ref{bertsempre} that on a Coble surface $X$ with irreducible boundary, it is always possible to construct involutions whose minimal model is the Bertini involution.
Actually, we can show more:
\begin{theorem}\label{solobert}
On an unnodal Coble surface $X$ with irreducible Coble curve $C$, any involution is the lift of a Bertini involution.
\end{theorem}
Before starting the proof, we point out that the vice versa is false: you can build Coble surfaces with reducible boundary, equipped with involutions which are lifts of a Bertini, as in the following example.
\begin{example}
Consider the construction of Remark \ref{occhio}, where we considered a rational sextic $\overline C \subset \mathbb{P}^2$ with $7$ nodes at points $p_1, \ldots, p_7$, and a triple point $p_8$. The surface $X' = Bl_8 \mathbb{P}^2$ of these eight points is a Del Pezzo surface, equipped with the Bertini involution $i' : X' \to X'$, which switches the strict transform $D \subset X'$ of $\overline C$ with the exceptional divisor $E_8$. When we blow - up the three intersection points in $D \cap E_8$, we get a Coble surface $X$ with $2$ boundary components, still equipped with an involution $i : X \to X$ induced by $i'$.
\end{example}
The following construction is remarkable, since it provides a Coble surface with $2$ boundary components, both preserved by a lift of the Bertini involution:
\begin{example}
Let $X' := Bl_{p_1, \ldots, p_8} \mathbb{P}^2$ be a Del Pezzo surface of degree $1$, with its Bertini involution $i' : X' \to X'$. Consider the anti canonical pencil $|- K_{X'}|$: we know that it consists of curves of genus $1$, each preserved by $i'$. Moreover, $|- K_{X'}|$ has a base point $p_9$, which is an isolated fixed point for $i'$. Let $\overline C_1, \overline C_2 \in |- K_{X'}|$ be two irreducible singular members, with nodes at two points $q_1, q_2$. Then both $q_1, q_2$ must be fixed by $i'$, so we can blow - up again $X := Bl_{p_9, q_1, q_2} X'$ to obtain a Coble surface $X$. This is a Coble surface, since its anti - bicanonical class contains the divisor $C_1 + C_2$, where $C_i \subset X$ is the strict transform of $\overline C_i$. Moreover, by construction, $X$ is equipped with a lifted Bertini involution which preserves both $C_1, C_2$.
\end{example}
We refer to \cite{Lesieutre2021} for the construction of a Coble surface with three boundary components, equipped with a Bertini involution.\\
We also note that, on a given unnodal Coble surface $X$ with irreducible Coble curve $C$, there is not only ``one'' lift of a Bertini involution: indeed, given any two disjoint $(-1)$-curves $E_1, E_2 \subset X$, we can think of them as the start of a sequence $E_1, \ldots, E_{10}$ of disjoint $(-1)$-curves, whose blow - down is $\mathbb{P}^2$, thanks to Proposition \ref{extmoc}. But then, the blow - down $p : X \to X'$ of just $E_1, E_2$ is a Del Pezzo surface $X'$ of degree $1$, and this is equipped with its Bertini involution $i' : X' \to X'$. Moreover, the base points $p (E_1), p(E_2) \in X'$ are nodes for the curve $p(C)$, because $C E_1 = C E_2 = 2$, and these two points are fixed by $i'$ by Remark \ref{bertsempre}. But then $i'$ lifts to a non - minimal involution $i : X \to X$, which preserves both $E_1, E_2$. Hence any pair of disjoint $(-1)$-curves on $X$ determines a different lift of a Bertini involution.\\
The proof of Theorem \ref{solobert} needs the following trick, which we will use repeatedly in the following:
\begin{remark}\label{nontre}
We showed in Proposition \ref{extmoc} that, if $X$ has irreducible anti - bicanonical curve $C \in |-2 K_X|$ and it does not contain $(-2)$-curves, then any set of three disjoint $(-1)$-curves $E_1, E_2, E_3 \subset X$ can be extended to a maximal sequence $E_1, \ldots, E_{10}$, such that their contraction transforms $X$ in $\mathbb{P}^2$. In this case the class of $C$ becomes $C = 6 L - 2 E_1 - \cdots - 2 E_{10}$. Now look at the system $$H := C + E_1 + E_2 + E_3 = 6 L - E_1 - E_2 - E_3 - 2 E_4 - \cdots - 2 E_{10}.$$
We saw in Subsection \ref{qcs} that $|H|$ induces a birational map on  a quintic surface in $\mathbb{P}^3$, and we described its singularities. Note then that an involution $i$ on $X$ which preserves each of $E_1, E_2, E_3$ must act as the identity on one of them. Indeed, $i$ preserves the linear system $H = C + E_1 + E_2 + E_3$, so it induces a commutative diagram:
\[\begin{tikzcd}
X  \arrow{d}{{|H|}} \arrow[swap]{r}{i} & X \arrow{d}{{|H|}}
 \\\mathbb{P}^3 \arrow{r}{I} & \mathbb{P}^3
\end{tikzcd}
\]
The plane $\Pi \subset \mathbb{P}^3$ corresponding to the section $C + E_1 + E_2 + E_3$ is preserved by $I$ by construction. Let $L_i$ be the image of $E_i$ inside $\Pi$. Since $i$ preserves all the $E_i$'s, the involution $I$ must preserve all the $L_i$. But then $I$ fixes the three intersection points $L_i \cap L_j$, and hence one of the $L_i$, say $L_1$, must be a fixed line.
\end{remark}
We point out that we used the hypothesis that $X$ is unnodal in Remark \ref{nontre}, and we will show later in a couter - example that it can not be dropped.\\
This Remark will be used to exclude case $v)$ of Lemma \ref{minimalpairs}.
\begin{corollary}\label{nongei}
Given a degree $2$ Del Pezzo surface $X'$ equipped with its Geiser involution $i' : X' \to X'$, and an unnodal Coble surface $X$ with irreducible boundary, it is impossible to build a commutative diagram:
\[\begin{tikzcd}
X  \arrow{r}{i} \arrow[swap]{d}& X \arrow{d}
\\ X' \arrow{r}{i'} & X' 
\end{tikzcd}
\]
\end{corollary}
{\it Proof:} We should blow up three points on $X'$ in a smart way to get a well defined involution on $X$. Remember that, for any $p \in X'$ such that $i' (p) \ne p$, there exist a pencil of sections of $|- K_{X'}|$ containing both $p, i' (p)$. We should pick $$p_1, p_2, p_3 \in X'$$ preserved by $i'$. Assume $$i' (p_1) = p_2, i' (p_2) = p_1, i' (p_3) = p_3.$$ Then there exists a divisor in $|- K_{X'}|$ containing all the three points, so we would have $|- K_X| \ne \emptyset$.\\ 
Hence we are forced to pick $$p_1, p_2, p_3 \in \Fix (i')$$ If we assume by contradiction that the surface $X := Bl_{p_1, p_2, p_3} X'$ is an unnodal Coble surface, with a irreducible divisor $C \in |- 2 K_X|$, then we end up exactly in the situation forbidden by the previous Remark \ref{nontre}. $\square$\\
\\
Note that the following example nonetheless provides a Coble surface with reducible anti - bicanonical divisor equipped with a Geiser involution.
\begin{example}\label{tottang}
Let $Q \subset \mathbb{P}^2$ be a smooth quartic, that is, a curve of genus $3$, canonically embedded in $\mathbb{P}^2$. It is well-known that there exist $63$ families of plane quadrics which are totally tangent to $Q$. Indeed, this corresponds to choosing $p_1, p_2, p_3, p_4 \in Q$ such that $2 p_1 + 2 p_2 + 2 p_3 + 2 p_4 \in |2 K_Q|$. Hence $$p_1 + p_2 + p_3 + p_4 \in |K_Q + \eta|$$ where $\eta$ is a nontrivial $2$-torsion divisor on $Q$. There are exactly $63$ such divisors, and each of them satisfies (by Riemann - Roch): $$h^0 ({\cal O}_Q (p_1 + p_2 + p_3 + p_4 + \eta)) = 2$$ Let then $C \subset \mathbb{P}^2$ be one of these smooth conics, and let $\pi: X' \to \mathbb{P}^2$ be the double cover branched over $Q$. Then $X'$ is a smooth Del Pezzo surface of degree $2$. The pre - image $\pi^{-1}(C)$ is a section of $|- 2 K_{X'}|$, of the form $$\pi^{-1} (C) = C_1 + C_2,$$ where $C_1, C_2$ are smooth rational curves, which meet at $4$ points $p_1', p_2', p_3', p_4'$ on the ramification locus of $\pi$. Then $X := Bl_{p_1', p_2', p_3', p_4'} X'$ is a Coble surface, with $2$ boundary components, equipped with a biregular involution, which is induced by the Geiser involution.
\end{example}
The following example is more tricky, since it provides the construction of a Coble surface $X$, equipped with an involution $i$ induced by a Geiser involution, which preserves all the components of the anti - bicanonical divisor $|- 2 K_X|$.
\begin{example} 
Let $Q \subset \mathbb{P}^2$ be an irreducible rational quartic curve with three nodes, and let $p: Y \to \mathbb{P}^2$ be the blow up of the plane at the singular points of $Q$. Consider the double cover $$q : X' \to Y$$ branched over the strict transform of $Q$, and let $i' : X' \to X'$ be the deck involution. As usual, let $L, E_1, E_2, E_3$ be the generators of $\Pic (Y)$. Then $$K_Y = - 3 L + E_1 + E_2 + E_3$$ and $$K_{X'} = - q^* L$$ which means that the composite map $pq : X' \to \mathbb{P}^2$ is induced by the anti canonical linear system. Moreover, $(- K_{X'})^2 = (pq)^* L^2 = 2$, which states that $X$ is a Del Pezzo surface of degree $2$. Each curve $E_1, E_2, E_3$ in $Y$ cuts in two points the branch locus, hence the preimages $q^* E_1, q^* E_2, q^* E_3$ are three $(-2)$-curves in $X'$. Let $B \subset Y$ be the branch locus of $q$, and let $R \subset X'$ be the ramification locus. The class of $R$ is $$R = \frac{1}{2} q^* (4 L - 2 E_1 - 2 E_2 - 2 E_3) = q^* (2 L - E_1 - E_2 - E_3),$$ so the divisor $R + q^* E_1 + q^* E_2 + q^* E_3$ belongs to the class $$R + q^* E_1 + q^* E_2 + q^* E_3 \in |q^* (2 L)| = |- 2 K_{X'}|.$$ All the four components are smooth rational curves, which meet at $R (q^* E_1 + q^* E_2 + q^* E_3) = 6$ points, all fixed by $i'$. Thus we can blow up these points on $X'$ to get a Coble surface $X$ with $4$ boundary components, together with an involution $i: X \to X$ which preserves all of them.
\end{example}
In Remark \ref{nontre} we used the linear system $H = C + E_1 + E_2 + E_3$ under the hypothesis that $X$ is unnodal, $C$ is irreducible and the $E_i$'s are disjoint. We now show that the unnodality of $X$ can not be dropped. We start showing that $H$ has still almost all its good properties.

\begin{proposition}
Let $X$ be a Coble surface, with irreducible anti - bicanonical curve $C \in |- 2 K_X|$, and let $E_1, E_2, E_3$ three disjoint $(-1)$-curves in $X$. Then $H := C + E_1 + E_2 + E_3$ has no base points or base components, and it has dimension $h^0 ({\cal O}_X (H)) = 4$. The images of the $E_i$ under the map $f_{|H|} : X \to \mathbb{P}^3$ are three pairwise distinct lines lying in the same plane of $\mathbb{P}^3$.
\end{proposition}\label{wd}
{\it Proof:} For any sequence of disjoint $(-1)$-curves $E_1, \ldots, E_r$ in $X$ one has $$h^0 ({\cal O}_X (E_1 + \cdots + E_r)) = 1$$ and $$h^1({\cal O}_X (E_1 + \cdots + E_r)) = 0$$ Indeed, the short exact sequence $$0 \to {\cal O}_X (E_1 + \cdots + E_{r - 1}) \to {\cal O}_X (E_1 + \cdots + E_r) \to {\cal O}_{E_r} (E_1 + \cdots E_r) \to 0$$ induces isomorphisms $$H^0 ({\cal O}_X (E_1 + \cdots + E_{r - 1})) \simeq H^0 ({\cal O}_X (E_1 + \cdots + E_r))$$ and $$H^1 ({\cal O}_X (E_1 + \cdots + E_{r - 1})) \simeq H^1 ({\cal O}_X (E_1 + \cdots + E_r))$$ so by induction we find \begin{eqnarray}\label{hzero} h^0 ({\cal O}_X (E_1 + \cdots + E_r)) = h^0 ({\cal O}_X) = 1\end{eqnarray} and \begin{eqnarray}\label{hu} h^1 ({\cal O}_X (E_1 + \cdots + E_r)) = h^1 ({\cal O}_X) = 0.\end{eqnarray} Now we look at the short exact sequence: \begin{eqnarray}\label{inizio} 0 \to {\cal O}_X (- C) \to {\cal O}_X \to {\cal O}_C \to 0.\end{eqnarray} Taking the tensor with ${\cal O}_X (H)$ we find: $$0 \to {\cal O}_X (E_1 + E_2 + E_3) \to {\cal O}_X (H) \to {\cal O}_C (H) \to 0$$ Now we note that it remains exact on global sections, due to equality \ref{hu}, so we find an exact sequence $$0 \to H^0 ({\cal O}_X (E_1 + E_2 + E_3)) \to H^0 ({\cal O}_X (H)) \to H^0 ({\cal O}_C (H)) \to 0.$$ Note that $${\cal O}_C (H) \simeq {\cal O}_{\mathbb{P}^1} (C H) = {\cal O}_{\mathbb{P}^1} (2).$$ Together with equality \ref{hzero}, this shows that $h^0 ({\cal O}_X (H)) = 4$ and that $C$ is not a base component of $H$, since $h^0 ({\cal O}_X (H - C)) = h^0 ({\cal O}_X (E_1 + E_2 + E_3)) = 1 < 4$. Moreover, the surjectivity of $H^0 ({\cal O}_X (H)) \to H^0 ({\cal O}_C (H))$ proves that $H$ has not base points on the curve $C$.\\
Now we tensor the short exact sequence in \ref{inizio} with ${\cal O}_X (C + E_1 + E_2)$, and applying equalities \ref{hzero} and \ref{hu} we find \begin{eqnarray}\label{newhzero} h^0 ({\cal O}_X (C + E_1 + E_2)) = 2\end{eqnarray} and \begin{eqnarray}\label{newhu} h^1 ({\cal O}_X (C + E_1 + E_2)) = 0.\end{eqnarray} Finally, we consider $$0 \to {\cal O}_X (C + E_1 + E_2) \to {\cal O}_X (H) \to {\cal O}_{E_3} (H) \to 0$$ Equalities \ref{newhzero} and \ref{newhu} state that the induced sequence on global sections $$0 \to H^0({\cal O}_X (C + E_1 + E_2)) \to H^0 ({\cal O}_X (H)) \to H^0 ({\cal O}_{E_3} (H)) \to 0$$ is still exact. The surjectivity of $H^0 ({\cal O}_X (H)) \to H^0 ({\cal O}_{E_3} (H))$ shows that $H$ has no base points on $E_3$, and the exactness proves that $$h^0 ({\cal O}_X (H - E_3)) = h^0 ({\cal O}_X (C + E_1 + E_2)) = 2 < 4$$ shows that $E_3$ is not a base component of $H$. By a symmetric argument, the same is true for $E_1, E_2$. Since all the base locus of $H$ must be contained in the curve $C + E_1 + E_2 + E_3$, this proves that $H$ has no base points or base components. The equality $H E_i = 1$ proves that the image of each $E_i$ under the map $f_{|H|} : X \to \mathbb{P}^3$ is a line.\\
 Finally, the long exact sequence induced by $$0 \to {\cal O}_X (E_3) \to {\cal O}_X (C + E_3) \to {\cal O}_C (C + E_3) \to 0$$ proves that $h^0 ({\cal O}_X (C + E_3)) = 1$. But this is the same to say $H^0 ({\cal O}_X (H - E_1 - E_2)) = 1$, so there exists a unique plane containing the images of both $E_1, E_2$, which means that these are different lines. $\square$\\\\
 The next step is to consider the following definition:
\begin{definition}
We denote by $\mathbb{P}^8 = |{\cal O}_{\mathbb{P}^1} (8)|$ the set of all effective divisors of degree $8$ on $\mathbb{P}^1$. Inside $\mathbb{P}^8$ let ${\cal V}^{\circ}$ the set of all divisors of the form $2 p_1 + 2 p_2 + 2 p_3 + p_4 + p_5$, with $p_i$ distinct points, and let ${\cal V}$ be its closure.\\
Similarly, let ${\cal W}^{\circ}$ be the locus of all divisors of the form $2 p_1 + 2 p_2 + 2 p_3 + 2 p_4$, with $p_i$ distinct points, and let ${\cal W}$ be its closure.
\end{definition}

\begin{proposition}\label{closure}
Both ${\cal V}, {\cal W} \subset \mathbb{P}^8$ are irreducible, and ${\cal W} \subset {\cal V}$. Moreover ${\cal V}^{\circ}$ is an open subset inside ${\cal V}$. 
\end{proposition}
{\it Proof:} Let denote by $\mathbb{P}^2 := |{\cal O}_{\mathbb{P}^1} (2)|, \mathbb{P}^3 := |{\cal O}_{\mathbb{P}^1} (3)|$, and consider the map $\tau : \mathbb{P}^2 \times \mathbb{P}^3 \to \mathbb{P}^8$, $$\tau (F, G) := F G^2$$ Let $\overline \Delta_2 \subset \mathbb{P}^2$ be the locus of polynomials which are squares of a linear polynomial, $\overline \Delta_3 \subset \mathbb{P}^3$ be the locus of polynomials with a double root, and $R \subset \mathbb{P}^2 \times \mathbb{P}^3$ the resultant locus of pairs of polynomials $(F, G)$ with at least a common root. Set $$\Delta_2 := \overline \Delta_2 \times \mathbb{P}^3 \subset \mathbb{P}^2 \times \mathbb{P}^3,$$  $$\Delta_3 := \mathbb{P}^2 \times \overline \Delta_3 \subset \mathbb{P}^2 \times \mathbb{P}^3,$$ so that $${\cal V}^{\circ} = \tau(\mathbb{P}^2 \times \mathbb{P}^3 \setminus \Delta_2 \setminus \Delta_3 \setminus R).$$ This proves that ${\cal V}^{\circ}$ (and thus ${\cal V}$) is irreducible, and also that $$\tau^{-1} ({\cal V}) = \mathbb{P}^2 \times \mathbb{P}^3,$$ that is $${\cal V} = \tau (\mathbb{P}^2 \times \mathbb{P}^3).$$ It follows then that $${\cal V}^{\circ} = {\cal V} \setminus \tau (\Delta_2) \setminus \tau (\Delta_3) \setminus \tau (R)$$ is open inside ${\cal V}$.\\
Finally, the equality $${\cal W}^{\circ} = \tau(\Delta_2 \setminus \Delta_3 \setminus R)$$ proves the irreducibility of ${\cal W}^{\circ}$ (and consequently of ${\cal W}$), and also that ${\cal W}^{\circ} \subset {\cal V}$. $\square$\\\\
The previous proposition has the following Corollary:
\begin{corollary}\label{gei}
There exist infinitely many pairs $(Q, C)$ where $Q \subset \mathbb{P}^2$ is a smooth quartic, and $C \subset \mathbb{P}^2$ is a smooth conic, whose intersection takes the form $Q \cap C = 2 q_1 + 2 q_2 + 2 q_3 + q_4 + q_5$, with $q_i$ pairwise distinct points.
\end{corollary}
{\it Proof:} Let $Q \subset \mathbb{P}^2$ be any ausiliary smooth quartic curve, and let $C \subset \mathbb{P}^2$ be any of the smooth conics given by Example \ref{tottang}. Look at the short exact sequence $$0 \to {\cal O}_{\mathbb{P}^2} (2) \to {\cal O}_{\mathbb{P}^2} (4) \to {\cal O}_C (4) \to 0$$ It induces an exact sequence $$0 \to H^0 ({\cal O}_{\mathbb{P}^2} (2)) \to H^0 ({\cal O}_{\mathbb{P}^2} (4)) \to H^0 ({\cal O}_{\mathbb{P}^1} (8)) \to 0.$$ Let $\rho : H^0 ({\cal O}_{\mathbb{P}^2} (4)) \to H^0 ({\cal O}_{\mathbb{P}^1} (8))$ be the restriction map, and look at its projectification $\rho : \mathbb{P}^{14} \setminus \mathbb{P}^5 \to \mathbb{P}^8,$ which is an $\mathbb{A}^6$-bundle. Let $U^{sm} \subset \mathbb{P}^{14}$ be the open subset corresponding to smooth plane quartics. Note that $U \subset \mathbb{P}^{14} \setminus \mathbb{P}^5$, since all quartics containing $C$ are clearly singular. Now let ${\cal W} \subset {\cal V} \subset \mathbb{P}^8$ be the loci defined previously. We denote by $[Q] \in \mathbb{P}^{14}$ the point corresponding to the quartic $Q$. Its restriction to $C$ has the form $2 (p_1 + p_2 + p_3 + p_4)$, hence it belongs to $\rho^{-1} ({\cal W})$, and consequently to $\rho^{-1} ({\cal V})$ by Proposition \ref{closure}. Since $Q$ is smooth, this proves that $U^{sm} \cap \rho^{-1} ({\cal V})$ is nonempty. Using again Proposition \ref{closure}, we have that ${\cal V}^{\circ}$ is open inside ${\cal V}$, and hence $\rho^{-1} ({\cal V}^{\circ})$ is open in $\rho^{-1} ({\cal V})$. Since $\rho$ is an $\mathbb{A}^6$-bundle, the space $\rho^{-1} ({\cal V})$ is irreducible, and hence the intersection $(U^{sm} \cap \rho^{-1} ({\cal V})) \cap \rho^{-1} ({\cal V}^{\circ})$ is nonempty. Any quartic in this intersection is smooth and it restricts to a divisor of the required form. $\square$\\\\
Now we are finally able to show that, without the assumption of unnodality for a Coble surface $X$, Remark \ref{nontre} and Corollary \ref{nongei} can be false.
\begin{example}
By Corollary \ref{gei}, we can pick a smooth quartic $Q \subset \mathbb{P}^2$ and a smooth conic $C \subset \mathbb{P}^2$ with $Q \cap C = 2 q_1 + 2 q_2 + 2 q_3 + q_4 + q_5$, where the $q_i$'s are pairwise distinct points. Let $$q : Y \to \mathbb{P}^2$$ be the double cover branched over $Q$. The surface $Y$ is a Del Pezzo surface of degree $2$, equipped with the deck involution $i_Y : Y \to Y$, which is a Geiser involution. Its canonical divisor is $$K_Y = q^* (- L),$$ where $L$ is any line in the plane. Let $\tilde C \subset Y$ be the preimage of $C$. Since $q_4 \ne q_5$, the curve $\tilde C$ is irreducible, and it belongs to the linear system $$\tilde C \in |q^* (2 L)| = |- 2 K_Y|.$$ Its algebraic genus equals $$p_a (\tilde C) = 1 + \frac{1}{2} (\tilde C^2 + \tilde C K_Y) = 1 + K_Y^2 = 3.$$ However, $\tilde C$ has three double points $p_1, p_2, p_3$, with $q (p_i) = q_i$, so $\tilde C$ has geometric genus $0$. Let $X \to Y$ be the blow up of these $3$ points, with exceptional divisors $E_1, E_2, E_3$. Let $C \subset X$ be the strict transform of $\tilde C$. Then $C$ is smooth, and its class in $\Pic (X) = \Pic (Y) \oplus \integer E_1 \oplus \integer E_2 \oplus \integer E_3$ equals $C = \tilde C - 2 E_1 - 2 E_2 - 2 E_3 = -2 (K_Y + E_1 + E_2 + E_3) = -2 K_X$. This shows that $X$ is a Coble surface, equipped with an involution $i_X : X \to X$, which is well - defined since the points $p_i$ are fixed by $i_Y$. The equivariant morphism $(X, i_X) \to (Y, i_Y)$ contradicts Corollary \ref{nongei}, and since $p_i$ are non - isolated fixed points for $i_Y$, the action of $i_X$ on $E_1, E_2, E_3$ is different from the identity, which contradicts Remark \ref{nontre}. Nonetheless, we know from Proposition \ref{wd} that the linear system $H = C + E_1 + E_2 + E_3$ induces a  well - defined morphism $f_{|H|} : X \to \mathbb{P}^3$. The problem is that we are not able to state that the projective involution $I : \mathbb{P}^3 \to \mathbb{P}^3$ induced by $i_X$ acts as the identity on one of the three image lines $L_i = f_{|H|} (E_i)$, since they are concurrent at a point. Indeed, let $\tilde R \subset Y$ be the ramification curve for $q$, and let $R \subset X$ be its strict transform. We claim that $\tilde C, \tilde R$ are equivalent in $\Pic (Y)$. Indeed, we already know that $$\tilde C = q^* (2 L),$$ and we have $$2 \tilde R = q^* (Q) = q^* (4 L).$$ The rationality of $Y$ allows us to simplify and get $$\tilde R \simeq q^* (2 L) \simeq \tilde C.$$ But then in $\Pic (X)$ we have $$H = C + E_1 + E_2 + E_3 =$$  $$= (\tilde C - 2 E_1 - 2 E_2 - 2 E_3) + E_1 + E_2 + E_3 =$$ $$= \tilde C - E_1 - E_2 - E_3 = \tilde R - E_1 - E_2 - E_3 = R.$$ This shows that the curve $R$ is a member of $|H|$, hence there exists exactly one hyperplane $\Pi_0 \subset \mathbb{P}^3$ containing the image $f_{|H|} (R)$. Since $i_X$ fixes $R$, the projective involution $I$ fixes $\Pi_0$. Then the fixed locus of $I$ has the shape $\Fix (I) = \Pi_0 \cup p_0$, where $p_0$ is an isolated fixed point. But each $E_1, E_2, E_3$ contains isolated fixed points $x_1, x_2, x_3$ with respect to $i_X$, as they correspond to the $(-1)$-eigenvectors under the action of $i_Y$ in the tangent planes of $Y$ at $p_1, p_2, p_3$ respectively. Since $I \circ f_{|H|} = f_{|H|} \circ i_X$, this implies that $f_{|H|} (x_1) = f_{|H|} (x_2) = f_{|H|} (x_3) = p_0$, and thus $L_1, L_2, L_3$ are all concurrent at $p_0$.\\
The reason for this failure is that $X$ contains $(-2)$-curves. Indeed, look at the tangent lines $\overline L_i \subset \mathbb{P}^2$ at $Q$ at $p_i$ for $i = 1, 2, 3$. The pull - back $q^* (\overline L_i)$ is a singular curve with algebraic genus $1$, and it has a double point at $p_i$ and self - intersection $q^* (\overline L_i)^2 = 2$. Hence the strict transforms are smooth rational curves in $X$, and their class is $q^* (\overline L_i) - 2 E_i$, so the self - intersection equals $(q^* (\overline L_i) - 2 E_i)^2 = 2 - 4 = - 2$.
\end{example}
In the following part, we will use repeatedly the following fact:
\begin{remark}\label{noc}
Assume we have a birational morphism $p : X \to Y$, where $X$ is an unnodal Coble surface with irreducible anti - bicanonical curve $C \in |- 2 K_X|$, and $Y$ is a smooth rational surface. Then the curve $p(C)$ can have only nodes or cusps as singularities. Indeed, let $H \in \Pic (Y)$ be the class of a big divisor. Since $p$ is birational, the pull - back $p^* H$ is still big on $X$, and it satisfies $$E (p^* H) = 0$$ for any irreducible curve $E$ contracted by $p$. By Hodge Index Theorem \ref{HIT} we find $$E^2 < 0,$$ so we can apply Proposition \ref{negativecurves} to deduce $$E^2 \in \{-1, -2, -4\}.$$ But $X$ has no $(-2)$-curves, and the unique $(-4)$-curve is $C$ itself, and hence $$E^2 = -1.$$ Thus all the curves contracted by $p$ are $(-1)$-curves, and each of these intersects $C$ with multiplicity $2$. If the two intersection points are distinct we find a node for $p(C)$, if they coincide we get an ordinary cusp.
\end{remark}
Now we eliminate case $iv)$ of Lemma \ref{minimalpairs}.
\begin{proposition}\label{nopunodue}
Assume $X$ is a Coble surface with irreducible anti - bicanonical divisor $C \in |- 2 K_X|$. Then it is impossible to build an involution $i :X \to X$ together with a regular birational morphism $p: X \to \mathbb{P}^1 \times \mathbb{P}^1$ which makes the following diagram commute:
\[\begin{tikzcd}
X  \arrow{r}{i} \arrow[swap]{d}{p}& X \arrow{d}{p}
\\ \mathbb{P}^1 \times \mathbb{P}^1 \arrow{r}{(p, q) \to (q, p)}& \mathbb{P}^1 \times \mathbb{P}^1 
\end{tikzcd}
\]
\end{proposition}
{\it Proof:} Since $p$ is birational, the image curve $C' := p (C)$ belongs to $$C' \in |- 2 K_{\mathbb{P}^1 \times \mathbb{P}^1}| = |{\cal O}_{\mathbb{P}^1 \times \mathbb{P}^1} (4, 4)|$$ Since $C$ is irreducible, also $C'$ is irreducible of arithmetic genus $9$, and it must be preserved by the involution which switches the two factors. Moreover, its singularities can be only nodes or cusps, by Remark \ref{noc}. Each node or cusp makes the geometric genus drop by $1$, hence their total number must be equal to $9$. Thus the singularitues of $C'$ are divided in $s$ pairs outside of the diagonal $\Delta$, plus $t$ singular points along $\Delta$, with $2 s + t = 9$. Let $q : \mathbb{P}^1 \times \mathbb{P}^1 \to \mathbb{P}^2$ be the quotient map, defined by the subsystem of symmetric sections of ${\cal O}_{\mathbb{P}^1 \times \mathbb{P}^1} (1, 1)$. The branch locus of $q$ is a smooth plane conic $B \subset \mathbb{P}^2$. By degree reasons, $$\Deg\, q(C') = 4$$  Each pair of singular points of $C'$ outside $\Delta$ corresponds to a singularity of $q (C')$, while a singularity in $\Delta$ corresponds to a tangency point between $B$ and $q(C')$. Since $q(C')$ has degree $4$, it can have at most $3$ ordinary double points and $4$ tangency points with the conic $B$, so $$s \le 3 \quad {\rm\, and\,} \quad t \le 4$$ This forces $$s = t = 3$$ The three singularities of $C'$ along $\Delta$ correspond to three disjoint exceptional divisors in $X$, each of them preserved by $i$, a contradiction with Remark \ref{nontre}. $\square$\\
\\
Again, the hypothesis of irreducibility of the Coble curve $C$ is crucial, as the following example shows:
\begin{example}
Let $F_1 + F_2 + F_3 + F_4 \in |{\cal O}_{\mathbb{P}^1 \times \mathbb{P}^1} (4, 0)|$ be the union of $4$ lines in one ruling, and let $G_1 + G_2 + G_3 + G_4$ in $|{\cal O}_{\mathbb{P}^1 \times \mathbb{P}^1} (0, 4)|$ be their images in the other ruling under the involution $(p, q) \to (q, p)$. The divisor $\sum_i F_i + \sum_i G_i$ lives in the anti - bicanonical class $|- 2 K_{\mathbb{P}^1 \times \mathbb{P}^1}| = |{\cal O}_{\mathbb{P}^1 \times \mathbb{P}^1} (4, 4)|$, and it is invariant under the involution. The blow up of its $16$ nodes produces a Coble surface with $8$ boundary components, equipped with a biregular involution with $\mathbb{P}^1 \times \mathbb{P}^1$ as minimal model.
\end{example}
Next step is to show that case $ii)$ of Lemma \ref{minimalpairs} is impossible, still preserving the assumption that the Coble curve is irreducible.
\begin{proposition}\label{dueno}
If $X$ is an unnodal Coble surface with irreducible Coble curve $C \in |- 2 K_X|$, and $i: X \to X$ is an involution, then the pair $(X, i)$ does not admit a minimal model as in case $ii)$ of Lemma \ref{minimalpairs}.
\end{proposition}
{\it Proof:} Assume the countrary: then we find an equivariant morphism $$p : X \to Y$$ with $Y$ a smooth rational surface equipped with an involution $$i_Y : Y \to Y$$ such that $$p \circ i = i_Y \circ p$$ Moreover, $Y$ admits a conic fibration $\pi_Y : Y \to \mathbb{P}^1$, whose fibers are all preserved by $i_Y$. Note that such an involution $i_Y$ arises as the deck involution of a double cover $q : Y \to \mathbb{F}_n$ on a Hirzebruch surface $\mathbb{F}_n$, branched over a smooth bisection $B \subset \mathbb{F}_n$. The singular fibers of $\pi_Y: Y \to \mathbb{P}^1$ correspond to the ramification locus of the restriction $B \to \mathbb{P}^1$. Putting all together, we have a diagram:
\[\begin{tikzcd}
X \arrow{r}{i} \arrow[swap]{d}{p} & X \arrow{d}{p} &
\\ Y  \arrow{r}{i_Y} & Y \arrow{r}{q} \arrow[swap]{dr}{\pi_Y} & \mathbb{F}_n \arrow{d}{}
\\  &  & \mathbb{P}^1
\end{tikzcd}
\]
We first want to exclude the case $n > 0$. Assume this actually happens, and consider the negative section $C_{- n} \subset \mathbb{F}_n$. If the pullback $q^* (C_{-n}) \subset Y$ is irreducible, then it is a curve of self intersection $$q^*(C_{-n})^2 = - 2 n \le -2$$ By Proposition \ref{negativecurves} and the unnodality assumption, the strict transform of $q^* (C_{-n})$ in $X$ must necessarily coincide with a component of the Coble curve $C$. Since $C$ is irreducible, we have $$p (C) = q^* (C_{- n})$$ But by definition, $Y$ is obtained by $X$ as a blow - down of $i$-equivariant $(-1)$-curves. Each $(-1)$-curve of $X$ intersects $C$ in two points, hence its contraction makes the self-intersection of $C$ jump by $4$, that is, $$p(C)^2 \ge C^2 + 4 = 0,$$ which is a contradiction.\\
Another contradiction arises if we assume that $q^* (C_{- n})$ actually splits as $$q^* (C_{- n}) = A + B$$ with $A, B$ curves switched by $i_Y$. In this case, we find the same self - intersection $$- 2 n =q^* (C_{-n})^2 = A^2 + B^2 + 2 A B = 2 A^2 + 2 A B$$ This shows that $A, B$ are two distinct curves with negative self - intersection, hence $$A^2 = B^2 = - 1$$ by Proposition \ref{negativecurves}. So the previous equality becomes $$- 2 n = 2 A B - 2$$ which is possible only if $A B = 0$ and $n = 1$. This again is a contradiction, since $i_Y$ should exchange two disjoint $(-1)$-curves, against the hypothesis of minimality.\\
The last case to exclude is that $C_{- n}$ might be a component of the branch locus $B \subset \mathbb{F}_n$ of $q$. Since $B$ is a bisection of $\mathbb{F}_n$, the other component $B - C_{- n}$ must be another section, disjoint from $C_{- n}$ (otherwise $Y$ would be a singular surface). The only curves with this property live in the linear system $C_{- n} + n F$, where $F$ is a fiber of the $\mathbb{P}^1$-bundle $|F| : \mathbb{F}_n \to \mathbb{P}^1$. Thus we find $$B = 2 C_{- n} + n F,$$ which forces $n$ to be even. Let $R \subset Y$ be the ramification component corresponding to $C_{- n}$, so that $$q^* (C_{- n}) = 2 R.$$ Taking squares, we have $$- 2 n = 4 R^2,$$ that is $$R^2 = - \frac{n}{2}$$ Now we can argue as before: if $n \ge 4$, then $R^2 \le -2$ and we still find $R = p (C)$, which is a contradiction, since $p(C)^2 \ge 0$. But if $n = 2$, we still have a contradiction, since $R^2 = - \frac{n}{2} = -1$ would give a $(-1)$-curve fixed by $i_Y$, against the assumption of minimality.\\
Finally, we have proved that $n = 0$, that is $$\mathbb{F}_n = \mathbb{F}_0 = \mathbb{P}^1 \times \mathbb{P}^1.$$ To avoid confusion, we denote by $F_1 \in \Pic (\mathbb{P}^1 \times \mathbb{P}^1)$ the ruling induced from $Y$, so that $q^* (F_1)$ is the conic bundle on $Y$ given by Lemma \ref{minimalpairs}. Let $F_2$ be the other ruling. In this notation, the branch curve $B$ of $q$ lives in the linear system \begin{eqnarray}\label{branch} B \in |a F_1 + 2 F_2|\end{eqnarray} for some even number $a \ge 0$.\\
If $a = 0$ then $B$ is the disjoint union of two horizontal lines in $F_2$. This means that $Y = \mathbb{P}^1 \times \mathbb{P}^1$ too, with the involution $i_Y : \mathbb{P}^1 \times \mathbb{P}^1 \to \mathbb{P}^1 \times \mathbb{P}^1$ acting as the identity on the first component. In a suitable choice of coordinates $([X_0, X_1], [Y_0, Y_1])$ on $Y$, we have $$i_Y ([X_0, X_1], [Y_0, Y_1]) = ([X_0, X_1], [Y_0, - Y_1])$$ and \begin{eqnarray}\label{quoz} q([X_0, X_1], [Y_0, Y_1]) = ([X_0, X_1], [Y_0^2, Y_1^2]) \end{eqnarray} Since $p : X \to Y$ is birational, the image $p(C) = p (- 2 K_X)$ must live in the linear system $$p (C) \in | - 2 K_{\mathbb{P}^1 \times \mathbb{P}^1}| = |{\cal O}_{\mathbb{P}^1 \times \mathbb{P}^1}(4, 4)|$$ An irreducible member of this system has algebraic genus $9$, as it is confirmed by the difference $$K_Y^2 - K_X^2 = 9$$ Hence $Y$ comes from the contraction via $p$ of $9$ $(-1)$-curves of $X$ over $9$ nodes of $p (C)$, and these points must be set - theoretically preserved by $i_Y$. Denote by $s$ the number of pairs of nodes which are switched by $i_Y$, and by $t$ the number of nodes fixed by $i_Y$, with $$2 s + t = 9$$ By Remark \ref{nontre}, we must have $t < 3$, so that $s = 4$ and $t = 1$ is the only possible solution. 
Now let: $$D := q(p (C)) \subset \mathbb{P}^1 \times \mathbb{P}^1$$ be the set - theoretic image of $p(C)$ in the final $\mathbb{P}^1 \times \mathbb{P}^1$, with the reduced scheme structure. Since $q : p(C) \to D$ has degree $2$, as divisors we have $$q_* (p_* (C)) = 2 D$$ But the expression \ref{quoz} shows that the induced push - forward $q_*$ acts as $F_1 \to 2 F_1, F_2 \to F_2$, so that $$2 D = q_* (p (C)) = q_* (4 F_1 + 4 F_2) = 8 F_1 + 4 F_2$$ and hence $$D \in |4 F_1 + 2 F_2|$$
The unique fixed node of $p (C)$ corresponds to a tangency point between $D$ and the branch locus $B$, while the $4$ pairs correspond to $4$ actual nodes of $D$. But an irreducible curve $D \in |4 F_1 + 2 F_2|$ has algebraic genus $3$, hence it cannot admit $4$ nodes.\\
We now show that the number $a$ in \ref{branch} can only assume the values $2$ or $4$. Indeed, consider the pencil $F_2$ and the pull-back $q^* (F_2)$ on $Y$. It is a pencil of hyperelliptic curves, which cover $2 : 1$ the lines in $|F_2|$ with $a$ branch points, so they have arithmetic genus $$p_a (q^*(F_2)) = \frac{a - 2}{2}$$ Note that $|- 2 K_Y|$ is effective, since it contains the curve $p (C)$, so the product $$q^* (F_2) (- 2 K_Y) \ge 0$$ must be non - negative. Putting this inside adjunction formula, we have $$2 p_a (q^*(F_2)) - 2 = q^* (F_2)^2 + q^* (F_2) K_Y \le 0$$ This leaves only $p_a (q^* (F_2)) = 0, 1$ as possibilities, and they correspond respectively to $a = 2, 4$.\\
If $a = 2$, then $q : Y \to \mathbb{P}^1 \times \mathbb{P}^1$ is branched over a curve of type $B \in |2 F_1 + 2 F_2|$, and we can compute $K_Y$ as $$K_Y = q^* (K_{\mathbb{P}^1 \times \mathbb{P}^1}) + \frac{1}{2} q^*(B) = q^*(- F_1 - F_2)$$
Note that $$K_Y^2 = 4$$ so $Y$ is obtained from $X$ via the contraction of $K_Y^2 - K_X^2 = 5$ equivariant $(-1)$-curves, which correspond to $5$ nodes of $p (C)$. As before, assume that $s$ pairs are switched by $i_Y$, while $t$ nodes are fixed, with $2 s + t = 5$. Since there are not isolated fixed points in $Y$, Proposition \ref{nontre} gives again $t < 3$, that is $$s = 2, \ t = 1$$ necessarily. Again, this means that $q(p (C))$ has $2$ nodes. As divisors, we have $q_* (p_* (C)) = 2 D$ for some effective irreducible divisor $D \subset \mathbb{P}^1 \times \mathbb{P}^1$, and $$q^* (D) = p (C) = - 2 K_Y = q^* (2 F_1 + 2 F_2)$$ so that $$D \in |2 F_1 + 2 F_2|$$ But an irreducible curve in this linear system has algebraic genus $1$, so it can have at most one node, contradicting $s = 2$.\\
There is only one case left, that is the number $a$ in \ref{branch} equals $4$, so that $B = 4 F_1 + 2 F_2$. Again, we compute can compute $K_Y$ as $$K_Y = q^* (K_{\mathbb{P}^1 \times \mathbb{P}^1}) + \frac{1}{2} q^*(B) = q^*(- F_2)$$ Hence the canonical has square $$K_Y^2 = 0$$ and hence $Y$ is obtained from $X$ via the contraction of only $K_Y^2 - K_X^2 = 1$ $(-1)$-curve $E$ to a point $p \in Y$. But $F_2$ can move in $\mathbb{P}^1 \times \mathbb{P}^1$, and hence $|- K_Y| = |q^* (F_2)|$ can move in $Y$. But then one member of $- K_Y$ passes through $p$, which means that $|- p^* K_Y - E|$ is non empty. This contradicts the very Coble assumption $|- K_X| = \emptyset$. $\square$\\
\\
\begin{proposition}
For an unnodal Coble surface, with irreducible boundary $\{C\} = |- 2 K_X|$, cases $i)$ and $iii)$ of Theorem \ref{minimalpairs} cannot happen.
\end{proposition}
{\it Proof:} Assume that we are in case $i)$ of Theorem \ref{minimalpairs}: then there exists a commutative diagram:
\[\begin{tikzcd}
X \arrow{r}{i} \arrow[swap]{d}{p} & X \arrow{d}{p} 
\\ \mathbb{F}_n  \arrow{r}{i'}  \arrow[swap]{d}{f} & \mathbb{F}_n \arrow{d}{f}
\\  \mathbb{P}^1 \arrow{r}{\tau} & \mathbb{P}^1
\end{tikzcd}
\]
where $\mathbb{F}_n$ is a Hirzebruch surface, with its $\mathbb{P}^1$-bundle $f : \mathbb{F}_n \to \mathbb{P}^1$, $p$ is an equivariant morphism, and $\tau$ is a nontrivial involution on $\mathbb{P}^1$.\\
Again, we start excluding the case $n > 0$: note that $n = 1$ is impossible, since $(\mathbb{F}_1, i')$ is not a minimal pair. Hence $n \ge 2$, but then the negative curve $C_{- n} \subset \mathbb{F}_n$ has a strict transform in $X$ with self intersection lesser or equal than $-2$. Again, by Proposition \ref{negativecurves}, this strict tranform must coincide with the anti - bicanonical curve $C$, so that $p (C) = C_{-n}$. But this is impossible, since $p: X \to \mathbb{F}_n$ contracts $(-1)$-curves, and hence $C_{-n}$ should be singular, which is false.\\
This gives $n = 0$, so we have $$i' : \mathbb{P}^1 \times \mathbb{P}^1 \to \mathbb{P}^1 \times \mathbb{P}^1$$ By Theorem \ref{minimalpairs}, the involution $i'$ preserves the linear system given by the first ruling $|F_1|$, and it acts on it as the involution $\tau$ of the previous diagram. But then $i'$ must preserve also the second ruling $|F_2|$, so let denote by $\sigma$ the action of $i': \mathbb{P}^1 \to \mathbb{P}^1$ on the second coordinate of the product. We already saw during the proof of \ref{dueno} how to deal with the case $\sigma = \mathbb{1}_{\mathbb{P}^1}$, so we can assume $\sigma \ne \mathbb{1}_{\mathbb{P}^1}$. In a suitable choice of coordinates, we can write \begin{eqnarray} \label{quasihor} i' ([X_0, X_1], [Y_0, Y_1]) = ([X_0, - X_1], [Y_0, - Y_1])\end{eqnarray} It has four fixed points, namely $([1, 0], [1, 0]), ([1, 0], [0, 1]), ([0, 1], [1, 0]), ([0, 1], [0, 1])$. Moreover, since $p$ is birational, the image $p (C)$ is an irreducible member of the linear system \begin{eqnarray} \label{quasihordue} p (C) \in |{\cal O}_{\mathbb{P}^1 \times \mathbb{P}^1} (4, 4)|\end{eqnarray} and hence it must have $9$ nodes. This set of nodes is preserved by $i'$, hence at least one of them must coincide with one of the four fixed points. Assume for example that this node is $p = ([1, 0], [1, 0])$. Then the map $p : X \to \mathbb{P}^1 \times \mathbb{P}^1$ factors through the blow - up $X \to Bl_p (\mathbb{P}^1 \times \mathbb{P}^1) \to \mathbb{P}^1 \times \mathbb{P}^1$. The involution induced by $i'$ on this blow - up acts trivially on the exceptional divisor over $p$, and it preserves both the strict transforms of the two fibers $\{X_1 = 0\}, \{Y_1 = 0\}$ containing $p$. Hence we can contract these $(-1)$-curves, and this produces another minimal model:
\[\begin{tikzcd}
X \arrow{r}{i} \arrow[swap]{d}{p} & X \arrow{d}{p} 
\\  \mathbb{P}^2 \arrow{r}{I} & \mathbb{P}^2
\end{tikzcd}
\]
where $I$ is a projective involution. This is the unique left possibility, corresponding to case $iii)$ of Theorem \ref{minimalpairs}. In a suitable choice of coordinates, we can write $$I [X_0, X_1, X_2] = [- X_0, X_1, X_2]$$ and the fixed locus is the disjoint union $\Fix (I) = \{p_0\} \cup L_0$, where $p_0 = [1, 0, 0]$ and $L_0 = \{X_0 = 0\}$. There are three possibilities: $p_0$ could be a node of the $10$-nodal sextic $\overline C := p (C) \subset \mathbb{P}^2$, or a smooth point of $\overline C$, or it could not lie in $\overline C$ at all.\\
In all these situations, it is impossible that all nodes are fixed by $I$, otherwise at least $9$ of them (all but at most $p_0$) should lie on $L_0$. This would provide an intersection product $\overline C L_0 = 18$, which is way too large. Hence there exists at least one pair of nodes $p, q \in {\rm\, Sing\,} (\overline C)$ switched by $I$. As before, the map $p : X \to \mathbb{P}^2$ factors through the blow - up $X \to Bl_{p, q} \mathbb{P}^2  \to \mathbb{P}^2$ This blow up is equipped with a well - defined involution, induced by $I$, which preserves the $(-1)$-curve consisting on the strict transform of the line ${\overline {p, q}}$. The contraction of this curve leads us to a minimal model of $(X, i)$ consisting of $\mathbb{P}^1 \times \mathbb{P}^1$, equipped with the involution which exchanges the two factors. This is forbidden by Proposition \ref{nopunodue}. $\square$\\
\\
This concludes the proof of Theorem \ref{solobert}.
\begin{example}
Again, we note that the irreducibility hypothesis for the curve $C$ cannot be dropped. Indeed, consider a plane cubic $\overline C_3$ whose components are all rational, and whose singularities are all double points: the curve $\overline C_3$ could be the union of three non - concurrent lines, or a line and a smooth conic, or an irreducible cubic with a node, or a cusp. Pick a point $p_0 \in \mathbb{P}^2$ outside $\overline C_3$, and a line $L_0$ not containing $p_0$ nor any of the singular points of $\overline C_3$, and not tangent to any of its components. If $\overline C_3$ splits as a smooth conic $C_2$ plus a line, assume also that $L_0$ is not the polar to $C_2$ with center $p_0$. Let $I : \mathbb{P}^2 \to \mathbb{P}^2$ be the involution with fixed locus $p_0 \cup L_0$, and let $X$ be the blow up of all the singularities of the reducible sextic $C_3 + \overline C_3$. Then $X$ is a Coble surface, equipped with a non - minimal involution induced by a projective one, with $2, 4$ or $6$ boundary components.
\end{example}
Note also that the expressions \ref{quasihor} and \ref{quasihordue} closely remember the construction of Horikawa models for Enriques surfaces in \cite{Alexeev2023}. The difference is that the construction in the cited article starts with an $i'$ invariant divisor of type $|{\cal O}_{\mathbb{P}^1 \times \mathbb{P}^1} (4, 4)|$ which in general needs not to be irreducible or rational, differently from our case. Moreover, the construction in \cite{Alexeev2023} needs this divisor to not contain any of the $4$ fixed points, which is something we could not avoid in our proof.


\subsection{Families of Coble surfaces}
The goal of this section is to construct a suitable space to parametrize triples $(X, H, E)$, where $X$ is a Coble surface, $H \in \Pic (X)$ is a polarization, and $E \subset X$ is a $(-1)$-curve contracted by the morphism induced by $|H|$. We will proceed in the following way: let $$\mathbb{P}^{27} \simeq \mathbb{P}(H^0 ({\cal O}_{\mathbb{P}^2} (6)))$$ the space of plane curves of degree $6$. Consider the Severi variety $V \subset \mathbb{P}^{27}$ given by $$V := \{[F] {\rm\, s.\, t.\,} V(F) {\rm\, is\, reduced,\, irreducible,\, with\,} 10 {\rm\, nodes\,}\} \subset \mathbb{P}^{27}$$ It is a locally closed subset of $\mathbb{P}^{27}$ of dimension $17$.\\
Now we need to specify a choice for a $(-1)$-curve. Take $\tilde V \subset V \times \mathbb{P}^2$, $$\tilde V := \{(F, p) {\rm\, such\, that\,} p \in Sing(V(F))\}$$ The variety $\tilde V$ is a $10 : 1$ cover of $V$, and up to restricting $V$, we can assume it is unramified. Now we take the trivial $\mathbb{P}^2$-bundle $\tilde V \times \mathbb{P}^2$, and note that the fibre product $$\tilde V \times_V \tilde V = \{(F, p, q) {\rm\, such\, that\,} p, q \in Sing (V(F))\}$$ lives inside $\tilde V \times \mathbb{P}^2$. The subspace $\tilde V \times_V \tilde V$ has codimension $2$ inside $\tilde V \times \mathbb{P}^2$, and it consists of a disjoint union of $10$ copies of $\tilde V$. Note that $\tilde V \times_V \tilde V$ contains a distinguished component, namely the diagonal $\Delta_{\tilde V}$, $$\Delta_{\tilde V} = \{(F, p, p), p \in Sing (V(F))\} \subset \tilde V \times_V \tilde V$$ Finally, we construct the universal Coble surface $${\cal X} := Bl_{\tilde V \times_V \tilde V} \tilde V \times \mathbb{P}^2$$ 
Let $\overline {\cal C} \subset \tilde V \times \mathbb{P}^2$ the universal nodal sextic $$\overline {\cal C} := \{((F, p), x) {\rm\, such\, that\,} F(x) = 0\}$$ and let ${\cal C} \subset \cal{X}$ be its strict transform.\\
Let ${\cal E} \subset {\cal X}$ be the universal $(-1)$-curve, that is, the exceptional divisor associated to $\Delta_{\tilde V}$.\\
Let ${\cal H} \in \Pic({\cal X})$ be the line bundle given by the pull-back of ${\cal O}_{\mathbb{P}^2} (1)$ via the composite map $${\cal X} \to \tilde V \times \mathbb{P}^2 \to \mathbb{P}^2$$ For every $x = (F, p) \in \tilde V$, the fiber ${\cal X}_x$ is the Coble surface obtained by the blow up of the singularities of $F$. The Coble curve of ${\cal X}_x$ is just the normalization of $V(F)$, that is, the intersection ${\cal C}_x := {\cal C} \cap {\cal X}_x$. The divisor ${\cal E}$ cuts in ${\cal X}_x$ the exceptional divisor associated to $p$, and the restriction ${\cal H}|_{{\cal X}_x}$ is a polarization on ${\cal X}_x$ which contracts ${\cal E}_x$.\\
Informally, we can think of this space as the set of quadruples $(X, C, E, H)$, where $(X, H)$ is a polarized Coble surface, with $C$ its Coble curve, and $E$ is one of the $10$ divisors contracted by $H$. Since $C$ is uniquely determined by the relation $|- 2 K_X| = \{C\}$, it can be omitted. Hence we will talk about triples $(X, E, H)$ with the properties we said above.\\
Note that we could achieve the same construction first by considering $${\cal Y} := Bl_{\tilde V} V \times \mathbb{P}^2$$ to build the universal Coble surface over $V$, and then defining ${\cal X} = f^* {\cal Y}$ as the pull-back via the $10 : 1$ cover $f : \tilde V \to V$.
\begin{definition}\label{rda}
A rationally determined automorphism of Coble surfaces is an automorphism ${\cal G} : {\cal Y} \to {\cal Y}$ which fits in a commutative diagram:
\[\begin{tikzcd}
{\cal Y}  \arrow{r}{{\cal G}} \arrow[swap]{d}& {\cal Y} \arrow{dl}
\\ \tilde V  
\end{tikzcd}
\]
Note that ${\cal G}$ induces via pull-back an automorphism $f^* {\cal G} : {\cal X} \to {\cal X}$.\\
The coincidence locus $\Gamma({\cal G})$ of ${\cal G}$ is the set of pairs $$\Gamma ({\cal G}) := \{(X, E) \in \tilde V {\rm\, such\, that\,} f^* {\cal G}|_{E \cap C} = \mathbb{1}\}$$
\end{definition}

\begin{example}
Let $A, B, C \in H^0 ({\cal O}_{\mathbb{P}^1} (2))$ three linearly independent forms of degree $2$ in $2$ homogeneous variables $u, v$. Assume that none of them is a square of a linear form, and no two of them have share a common root. Let $$\Lambda = (\lambda_{i, j})_{i, j = 0}^2$$ a generic $3 \times 3$ invertible matrix with complex coefficients, and define \begin{eqnarray*}
G_0 &:=& \lambda_{0, 0} A + \lambda_{0, 1} B + \lambda_{0, 2} C
\\
G_1 &:=& \lambda_{1, 0} A + \lambda_{1, 1} B + \lambda_{1, 2} C
\\
G_2 &:=& \lambda_{2, 0} A + \lambda_{2, 1} B + \lambda_{2, 2} C
\end{eqnarray*}
Consider the three polynomials of degree $6$: $B C G_0, A C G_1, A B G_2$. If the matrix $\Lambda$ is generic, these polynomials are linearly independent, with no common roots, and they define a regular map $$\gamma: \mathbb{P}^1 \to \mathbb{P}^2$$ $$\gamma (u, v) := [B C G_0 (u, v), A C G_1 (u, v), A B G_2 (u, v)]$$ which is birational over a rational plane sextic with $10$ nodes. Its singularities are located at the points $[1, 0, 0], [0, 1, 0], [0, 0, 1]$ and other seven points $p_1, \ldots, p_7$. Consider the set of triples $(X, E, H)$, where:\\
$i)$ $X$ is the blow up of all the nodes;\\
$ii)$ $E \subset X$ is the strict transform of the line $Z = 0$;\\
$iii)$ $H \in \Pic(X)$ is the polarization defined by $Y Z, X Z, X Y$, that is the system of plane conics through $[1, 0, 0], [0, 1, 0], [0, 0, 1]$ Its exceptional curves are the divisors associated to $p_1, \ldots, p_7$, plus the transforms of the lines $X = 0, Y = 0, Z = 0$. In particular, $E$ is one of them.\\
Consider the Bertini involution ${\cal G}: X \to X$ with base points $[0, 0, 1], p_1, \ldots, p_7$. It is well defined for every $\Lambda$, so we are in the case described by Definition \ref{rda}. Its restriction to the curve acts as the deck involution associated to the $g_2^1$ spanned by $A, B$. Hence its pair of fixed points is given by the zeroes of $$\frac{\partial A}{\partial u} \frac{\partial B}{\partial v} - \frac{\partial A}{\partial v} \frac{\partial B}{\partial u} = 0$$ Meanwhile, the intersection $E \cap C$ is given by $G_2 = 0$. Hence the coincidence loci of ${\cal G}$ is the set of matrices $\Lambda$ such that $G_2, \frac{\partial A}{\partial u} \frac{\partial B}{\partial V} - \frac{\partial A}{\partial v} \frac{\partial B}{\partial u}$ are linearly independent in $H^0 ({\cal O}_{\mathbb{P}^1} (2))$. Note that this is a $2$-codimensional condition over the set of all matrices $\Lambda$.
\end{example}
Thanks to Theorem \ref{solobert}, we are able to state that this is the general case:
\begin{lemma}
The coincidence loci of a rationally defined involution has codimension $2$ inside $\tilde V$.
\end{lemma}

\newpage
\section{Appendix}
This Appendix contains some un - finished computations, which were made in attempt to classify involutions on unnodal Coble surfaces with irreducible Coble curve. These calculations were made before the proof of the Proposition \ref{nifc} and Theorem \ref{solobert} in Section $3$. The idea was to take any involution $i : X \to X$ such that $i|_C = \mathbb{1}_C$, an exceptional curve $E \subset X$, and look at the behaviour of the linear system $E + i (E)$. However, none of the proof of Section $3$ requires these computations, which until now remain suspended.\\
The previous lemma has a very nice consequence:
\begin{corollary}\label{diffmoc}
Suppose $X$ is a Coble surface, $C \subset X$ the Coble curve, $i: X \to X$ an involution acting identically on $C$, and $E \subset X$ a rational normal curve of self intersection $-1$. Then $i (E) \ne E$.
\end{corollary}
{\it Proof:} Let $X'$ be the surface obtained from $X$ via the blow-down of $E$ to a point $p \in X'$. Then $X'$ is still a smooth rational surface. If $i(E) = E$ the involution $i$ descends to an involution $i': X' \to X'$. Since $E$ is a $(-1)$-curve, we have $$E K_X = -1$$ so $$E C = -2 E K_X = 2$$ Thus the image of $C$ inside $X'$ is fixed by $i'$ but it has a node at $p$, which contradicts the previous Lemma. $\square$\\

In particular, the intersection product $E i(E)$ is a non negative number. Moreover, since $i$ fixes the two intersection points in $E \cap C$, this intersection counts at least these two points. This leads to the next definition.
\begin{definition}
We will denote by $\overline E$ the curve $$\overline E := i(E)$$ and by $k \ge 0$ the natural number given by $$E \overline E = k + 2$$ In other words, $k$ is the number of non trivial intersection points between $E$ and $\overline E$.
\end{definition}
Of course this number $k$ depends on the choice of the $(-1)$-curve $E$ and the involution $i$. We want to compute the dimension  of the linear system $|E + \overline E|$.\\
We look at the normalization of the divisor $E + \overline E$, $\nu: E \sqcup \overline E \to E + \overline E$, where the symbol $\sqcup$ denotes the disjoint union. It defines a short exact sequence $$0 \to {\cal O}_{E + \overline E}  \xrightarrow{\nu^*} {\cal O}_{E \sqcup \overline E} = {\cal O}_E \oplus {\cal O}_{\overline E} \xrightarrow{(f, g) \to f|_{E \cap \overline E} - g|_{E \cap \overline E}} {\cal O}_{E \cap \overline E} \to 0$$  Tensoring by ${\cal O}_X (E + \overline E)$, and taking cohomology, we find $$0 \to H^0 ({\cal O}_{E + \overline E} (E + \overline E)) \to H^0 ({\cal O}_E (E + \overline E)) \oplus H^0 ({\cal O}_{\overline E} (E + \overline E)) \to H^0 ({\cal O}_{E \cap \overline E} (E + \overline E)) \to \cdots$$ 
\begin{remark}
The map $$ev: H^0 ({\cal O}_E (E + \overline E)) \oplus H^0 ({\cal O}_{\overline E} (E + \overline E)) \to H^0 ({\cal O}_{E \cap \overline E} (E + \overline E))$$  $$(\sigma, \tau) \to \sigma |_{E \cap \overline E} - \tau|_{E \cap \overline E}$$ is surjective.\\
Indeed, we have $$E (E + \overline E) = \overline E (E + \overline E) = k + 1$$ so the domain vector space can be thought as the direct sum of homogeneous polynomials in two variables of degree $k + 1$ over $E, \overline E$ respectively. On the other hand, the codomain is a sum of $k + 2$ copies of $\complex$, one for each intersection point in $E \cap \overline E$.\\
The restriction of this evaluation map to each of the two summands is of course an isomorphism, because the only homogeneous polynomial of degree $k + 1$ in two variables admitting $k + 2$ zeroes is $0$. The surjectivity of $ev$ follows.
\end{remark}
This fact has several consequences.
\begin{corollary}
We have $h^0 ({\cal O}_{E + \overline E} (E + \overline E)) = k + 2$, and $h^0 ({\cal O}_X (E + \overline E)) = k + 3$.
\end{corollary}
{\it Proof:} The first equality follows immediately from the short exact sequence $$0 \to H^0 ({\cal O}_{E + \overline E} (E + \overline E)) \to H^0 ({\cal O}_E (E + \overline E)) \oplus H^0 ({\cal O}_{\overline E} (E + \overline E)) \to H^0 ({\cal O}_{E \cap \overline E} (E + \overline E)) \to 0$$ 
For the second equality, we look at the standard short exact sequence: $$0 \to {\cal O}_X \to {\cal O}_X (E + \overline E) \to {\cal O}_{E + \overline E} (E + \overline E) \to 0$$ The first cohomology group $H^1 ({\cal O}_X)$ vanishes since $X$ is rational, so the previous sequence remains exact at levels of global sections: $$0 \to H^0 ({\cal O}_X) \to H^0 ({\cal O}_X (E + \overline E)) \to H^0 ({\cal O}_{E + \overline E} (E + \overline E)) \to 0$$ In particular, $$h^0 ({\cal O}_X (E + \overline E)) = k + 3$$ $\square$
\begin{corollary}
The linear system ${\cal O}_X (E + \overline E)$ has not fixed components or base points.
\end{corollary}
{\it Proof:} The only possible fixed components of ${\cal O}_X (E + \overline E)$ can be $E$ or $\overline E$. But if one of these two curves is a fixed component, then the other one must move in a net, which is forbidden by $E^2 = \overline E^2 = -1$.\\
For the fixed points, the argument is similar: a fixed point of ${\cal O}_X (E + \overline E)$ must necessarily lie on the divisor $E + \overline E$. The restriction map $H^0 ({\cal O}_X (E + \overline E)) \to H^0 ({\cal O}_{E + \overline E} (E + \overline E))$ is surjective by the previous corollary, so it suffices to show that ${\cal O}_{E + \overline E} (E + \overline E)$ is base point - free.\\
Suppose $p \in E$, $\sigma \in H^0 ({\cal O}_E (E + \overline E))$ such that $\sigma (p) \ne 0$. By the previous remark, there exist a unique section $\tau \in H^0 ({\cal O}_{\overline E} (E + \overline E))$ which agrees with $\sigma$ in the intersection $E \cap \overline E$, so the couple $(\sigma, \tau)$ defines an element of $H^0 ({\cal O}_{E + \overline E} (E + \overline E))$ which is nonzero at $p$. $\square$ 

\begin{corollary} 
If $E \ne \overline E$, the linear system $|E + \overline E|$ defines a regular map $f : X \to \mathbb{P}^{k + 2}$.
\end{corollary}
\begin{definition}
We denote by $Y:= f(X) \subset \mathbb{P}^{k + 2}$ the image of $X$ under the regular function $f$. 
\end{definition}
The surface $Y$ satisfies $$(\Deg\, f) (\Deg\, Y) = D^2 = 2 k + 2$$ and, since it is non-degenerate, $$\Deg\, Y \ge k + 1$$ which leaves only two possibilities: \\
{\bf Case 1:} $f : X \to Y$ has degree two, and $Y \subset \mathbb{P}^{k + 2}$ is a surface of minimal degree, or\\
{\bf Case 2:} $f : X \to Y$ is a birational morphism, and $Y$ has degree $\Deg\, Y = 2 k + 2$.\\
\begin{remark}\label{noddoublecover}
We briefly underline the fact that, if $f : C \to \mathbb{P}^1$ is a double cover branched over a divisor $B = p + q + 2 r$, with $p, q, r$ distinct points in $\mathbb{P}^1$, then $C$ is a rational curve. Indeed, we can locally write $$C =\{(x, y) \in \complex^2 {\rm\,\, such\, that\,} y^2 = x^2 (x^2 - 1)\}$$ and $f$ corresponds to the projection on the $x$ axis. Then we can give the rational parametrization $$x = \frac{t^2 + 1}{t^2 - 1}, \quad y = 2t \frac{t^2 + 1}{(t^2 - 1)^2}$$ More geometrically, the unique singularity of $C$ is a node $\overline r$, with $f (\overline r) = r$. In the normalization $\tilde C \to C$ there are two points over $\overline r$, so the composite map $\tilde C \to C \to \mathbb{P}^1$ is ramified only over $p, q$.\\
This remains true even if $p = r$, so that the branch divisor $B$ in $\mathbb{P}^1$ has the shape $B = 3 p + q$. In this case, $C$ has locally the form $$C =\{(x, y) \in \complex^2 {\rm\,\, such\, that\,} y^2 = x^3 (x - 1)\}$$ which has a cusp in the origin. Such a curve admits the rational parametrization $$x = \frac{t^2}{t^2 - 1}, \quad y = \frac{t^3}{(t^2 - 1)^2}$$ Again, the geometric interpretation is that the normalization $\tilde C$ of $C$ has only one points lying over the cusp of $C$, so the composite map $\tilde C \to C \to \mathbb{P}^1$ is ramified over the two points $p, q$ without multiplicity.
\end{remark}

The previous theorem gives us a precious result:
\begin{remark}\label{contractedcurves}
Suppose $D$ is an irreducible curve contracted by the map $f : X \to \mathbb{P}^{k + 2}$ induced by the linear system $|E + \overline E|$. Since $H := E + \overline E$ is big and nef, and $H D = 0$, Hodge Index Theorem \cite{TieLuo1990} forces $$D^2 < 0$$ Thus, by Proposition \ref{negativecurves} we find $$D^2 \in \{-1, -2\}$$
If $D$ is a $(-1)$-curve contracted by $f$, then $$D^2 = -1, H D = 0$$ Since $i$ is an isomorphism and $H$ is $i$-invariant, we also have $$i(D)^2 = -1, H i(D) = 0$$ So we can apply Hodge Index Theorem to the divisor $D + i(D)$ to deduce that $$(D + i(D))^2 < 0$$ This  translates to $$D i(D) < 1$$ a contradiction to Corollary \ref{diffmoc}. Thus $D^2 = -1$ cannot happen.
\end{remark}
\begin{proposition}\label{almenotre}
The case $k = 0$ is impossible.
\end{proposition}
{\it Proof:} For $k = 0$ the linear system  induces a map $f : X \to \mathbb{P}^2$, with $f^* {\cal O}_{\mathbb{P}^2} (1) = E + \overline E$. The degree of $f$ can be computed by $$\Deg\, f = (E + \overline E)^2 = 2$$ By construction $$i^* (E + \overline E) = E + \overline E$$ hence there exists an projective involution $I$ of $\mathbb{P}^2$ which fits into a commutative diagram:
\[\begin{tikzcd}
X  \arrow{r}{f}  \arrow[swap]{d}{i} & \mathbb{P}^2 \arrow{d}{I} 
\\ X  \arrow{r}{f}  & \mathbb{P}^2 
\end{tikzcd}
\]
The involution $I$ acts trivially on the non degenerate curve $f(C)$, so $$I = \mathbb{1}_{\mathbb{P}^2}$$ So $i$ corresponds to the deck involution of the double cover $f$.\\
By a theorem of Stein, the map $f$ factors through a (possibly singular) surface $\overline X$ in two regular maps $$\tau : X \to \overline X,\quad \overline f : \overline X \to \mathbb{P}^2,\quad f = \overline f \circ \tau$$ where $\tau$ is a birational map with connected fibers, while $\overline f$ is a finite map of degree $\Deg\, \overline f = \Deg\, f = 2$, branched over a curve $B \subset \mathbb{P}^2$.\\
It is easy to compute $\Deg\, B$: if $l \subset \mathbb{P}^2$ is a generic line, the restriction $$f : f^{-1} (l) \to l$$ is a double cover from the smooth elliptic curve  $f^{-1} (l) \in |E + \overline E|$ to the rational normal curve $l$. This forces $l$ to contain exactly $4$ branch points, so that $$\Deg\, B = 4$$
By hypothesis $$i|_C = \mathbb{1}_C$$ so $\tau(C)$ is a component of the ramification locus of the double cover $\overline f$. Thus the restriction $f : C \to f(C)$ is a birational map, and $$\Deg\, f(C) = C (E + \overline E) = 4$$ This forces $$f(C) = B$$ by degree reasons.\\
Thus $B$ is a rational plane quartic curve, and hence $B$ must have singular points. But this forces also the branch locus $\tau (C) \subset \overline X$ to be singular. On the other hand, by remark \ref{contractedcurves}, $\tau$ can only contract some $(-2)$-curves of $X$, which are disjoint from $C$, so $\tau (C)$ should be smooth, a contradiction. $\square$\\

So let us assume $k > 0$, and suppose we are in {\bf Case 1} as above: the function $f : X \to \mathbb{P}^{k + 2}$ is a map of degree $2$ over a surface $Y := f(X) \subset \mathbb{P}^{k + 2}$ of minimal degree $k + 1$.
\begin{remark}
We remember the definition of Hirzebruch surfaces $\mathbb{F}_n, n \ge 0$ as follows: we consider the rank $2$ vector bundle ${\cal E}$ over $\mathbb{P}^1$, $${\cal E} := {\cal O}_{\mathbb{P}^1} \oplus {\cal O}_{\mathbb{P}^1} (- n)$$ and $\mathbb{F}_n := \mathbb{P}({\cal E})$. The surface $\mathbb{F}_n$ has a $\mathbb{P}^1$-fibration $$\pi : \mathbb{F}_n \to \mathbb{P}^1$$ For $n > 0$, this fibration admits a unique section $C_n$ with negative self intersection; more precisely, $$C_n^2 = - n$$ and $$\Pic (\mathbb{F}_n) = \integer C_n \oplus \integer F$$ where $F$ is the class of a fiber of $\pi$. The linear system $|C_n + n F|$ defines a map $g: \mathbb{F}_n \to \mathbb{P}^{n + 1}$. The image of $g$ is a cone over a Veronese model of $\mathbb{P}^1$ of degree $n$ in $\mathbb{P}^{n + 1}$. The map $g$ is an isomorphism outside of $C_n$, and it contracts $C_n$ to the vertex of the cone.
\end{remark}
Surfaces of minimal degree $r$ in $\mathbb{P}^{r + 1}$ are completely classified, see for example \cite{DP1885}, \cite{EH1987} and they are the following:\\
{\bf Case 1.1:} A smooth rational normal scroll, i.e., a surface obtained as $\mathbb{P} ({\cal E})$, with ${\cal E} = {\cal O}_{\mathbb{P}^1} (-a) \oplus {\cal O}_{\mathbb{P}^1} (-b), a, b > 0$, embedded as a smooth surface of degree $a + b$ in $\mathbb{P}^{a + b + 1}$ by the hyperplane linear system ${\cal O}_{\mathbb{P} ({\cal E})} (1)$. Note that such a surface is isomorphic to the Hirzebruch surface $\mathbb{F}_{|a - b|}$, but this is a ``better'' embedding, which avoids the cone singularity.\\
{\bf Case 1.2:} A cone representation of a Hirzebruch surface, as in the previous remark. Note that this is the degeneration of {\bf Case 1} when we let $a$ or $b$ to be $0$.\\
{\bf Case 1.3:} The Veronese embedding of $\mathbb{P}^2$ in $\mathbb{P}^5$.
\begin{proposition}\label{singim}
The {\bf Case 1.1} is impossible.
\end{proposition}
{\it Proof:} In {\bf Case 1.1} the surface $Y \subset \mathbb{P}^{k + 2}$ is isomorphic $$Y \simeq \mathbb{P} ({\cal O}_{\mathbb{P}^1} (- a) \oplus {\cal O}_{\mathbb{P}^1} (- b))$$ where $$a + b = k + 1$$ We denote by $F$ a fiber of the $\mathbb{P}^1$-fibration $Y \to \mathbb{P}^1$, and by $H \in \Pic (Y)$ the hyperplane section giving the embedding $Y \hookrightarrow \mathbb{P}^{k + 2}$. Then $$\Pic (Y) = \integer H \oplus \integer F$$ where the product law is given by $$H^2 = a + b = k + 1,\quad H F = 1,\quad F^2 = 0$$ Moreover, it is a straightforward computation that $$K_Y = - 2 H + (k - 1) F$$
If we decompose $$f_* C = c_1 H + c_2 F, \quad c_1, c_2 \in \integer, c_1 \ge 0$$ the equality $$C f^* H = C (E + \overline E) = 4$$ gives $$c_1 (k + 1) + c_2 = 4$$ On the other hand, the divisor $f_*C$ has the form $f_* C = d \overline C$, with $d \in \{1, 2\}$, and $\overline C$ a reduced irreducible divisor. Hence we can not allow $c_1 = 0, c_2 = 4$, because the linear system $|4 F|$ does not contain divisors of such type. Thus the inequality $$c_1 > 0$$ is sharp.\\ Let $\tilde F := f^* F \subset X$ be the pre-image of a line $F$. Note that $$\tilde F K_X = F f_* K_X = - \frac{1}{2} F f_*C = - \frac{1}{2} c_1$$ Hence, the curve $\tilde F$ has genus $$\tilde g := g(\tilde F)$$ such that $$2 \tilde g - 2 = \tilde F( \tilde F + K_X) = - \frac{1}{2} c_1$$ so that $c_1 = 4 - 4 \tilde g$. This forces $$\tilde g = 0$$ and $$f_* C = 4 H - 4 k F$$
Here we find a contradiction: the linear system $|4 H - 4 k F|$ does not contain any irreducible reduced divisor $\overline C$. On the countrary, the genus of such a curve $\overline C$ would be equal to: $$g(\overline C) = 1 + \frac{1}{2} \overline C(\overline C + K_Y) - m_{\rm sing}$$ where $m_{\rm sing} \ge 0$ is the contribution given by the singularities of $\overline C$. But this leads to $$g (\overline C) = 3 - 6 k - m_{\rm sing} < 0$$ which is impossible. The only remaining chance is that the restriction $f|_C$ is a double cover of an irreducible reduced divisor $\overline C \in |2 H - 2 k F|$, but again, its genus $g(\overline C)$ would be given by the same formula: $$g(\overline C) = 1 + \frac{1}{2} \overline C(\overline C + K_Y) -m_{\rm sing} = - k - m_{\rm sing} < 0$$ Again, we found a contradiction. $\square$\\

\begin{corollary}\label{fCnondeg}
If $k = 1$, then $f(C)$ is a non degenerate curve of degree $4$ in $\mathbb{P}^3$.
\end{corollary}
{\it Proof:} We already know that $$dim\, |H| = k + 2 = 3$$ so we have $f : X \to \mathbb{P}^3$. Suppose $f(C)$ is a plane curve: then the linear system $|H- C|$ is non empty, so there exist an effective divisor $D$, $$D \in |H - C|$$ Of course $H \ne C \in \Pic (X)$, so $D$ cannot be the zero divisor. Moreover, $$H D = H^2 - H C = (2 k + 2) - 4 = 0$$ so $f$ contracts all the irreducible components of $D$. By Remark \ref{contractedcurves} the only possibility is that all the irreducible components of $D$ are smooth rational $(-2)$-curves, which are necessarily disjoint from $C$. Thus $|H|$ would contain the disconnected section $$C + D \in |H|$$ which is forbidden by Bertini's Theorem. So $f (C)$ is non-degenerate. $\square$

\begin{corollary}
If $k = 1$, then $|H|$ induces a map of degree $2$ on a cone in $\mathbb{P}^3$.
\end{corollary}
{\it Proof:}  Since $H$ is an $i$-invariant linear system, there exists a projective involution $$I : \mathbb{P}^3 \to \mathbb{P}^3$$ which fits into the following commutative diagram:
\[\begin{tikzcd}
X  \arrow{r}{f}  \arrow[swap]{d}{i} & \mathbb{P}^3 \arrow{d}{I} 
\\ X  \arrow{r}{f}  & \mathbb{P}^3 
\end{tikzcd}
\]
In particular, $I$ acts as the identity over the curve $f (C)$, which is non degenerate by the previous corollary \ref{fCnondeg}. Thus  $$I = \mathbb{1}_{\mathbb{P}^3}$$ everywhere, and hence $i$ preserves the fibers of $f$, which forces $$\Deg\, f = 2$$ and $$\Deg\, f(X) = \frac{H^2}{\Deg\, f} = 2$$
So $f(X)$ is a quadric surface in $\mathbb{P}^3$, which must be a singular cone by Proposition \ref{singim}. $\square$\\

\begin{proposition}
The {\bf Case 1.3} is also impossible.
\end{proposition}
{\it Proof:} This case can only happen if $k = 3$, so that the linear system $|E + \overline E|$ is made up of curves of genus $4$. The image $Y$ of the associated map $f : X \to \mathbb{P}^5$ is the Veronese embedding of $\mathbb{P}^2$ through the linear system of plane quadrics. Let $v : \mathbb{P}^2 \simeq Y \subset \mathbb{P}^5$ be this embedding. A general hyperplane section $H_Y \subset Y$ is a smooth rational curve, covered $2 : 1$ by a curve of genus $4$. Thus $H_Y$ contains $10$ branch points. This means that the composite map $v^{-1} \circ f : X \to \mathbb{P}^2$ is a double cover branched over a quintic curve, and this can not happen. $\square$\\
\subsection{An application of Reider Theorem}
Assume $$E \cap \overline E =\{p_1, \ldots, p_n\}$$ with $n \ge 3$, and let label these points such that $$\{p_1, p_2\} = E \cap C = \overline E \cap C$$ This forces $i (p_1) = p_1, i(p_2) = p_2$. Assume also that $$i (p_3) = p_3$$ then for all $k \ge 4$ the Von Staudt Theorem gives the equality of cross ratios $$(p_1, p_2, p_3, p_k) = (p_1, p_2, p_3, i(p_k))$$ in $E$ (or in $\overline E$). This forces $$i(p_k) = p_k$$ for all the common points $p_k$'s, and hence $$\frac{E + \overline E}{i} \simeq \mathbb{P}^1$$ In other words, $E + \overline E$ is an hyper-elliptic curve. Consequently, the canonical linear system $K_{E + \overline E}$ does not separate the point $p \in E$ from the point $i (p) \in \overline E$. Note that the short exact sequence $$0 \to {\cal O}_X (K_X) \to {\cal O}_X (E + \overline E + K_X) \to {\cal O}_{E + \overline E} (E + \overline E + K_X) = {\cal O}_{E + \overline E} (K_{E + \overline E}) \to 0$$ induces an isomorphism \begin{eqnarray}\label{firstiso} H^0 (X, {E + \overline E + K_X}) \simeq H^0 (E + \overline E, K_{E + \overline E}) \end{eqnarray} For every $H \in |E + \overline E|$ there exists the same isomorphism \begin{eqnarray}\label{seciso} H^0 (X, E + \overline E) \simeq H^0 (H, K_H)\end{eqnarray}
In particular, a pair of conjugate point $p \in E, i(p) \in \overline E$ is not separated by the canonical system $|K_{E + \overline E}|$. Relation \ref{firstiso} implies that $p, i(p)$ are not separated by the linear system $|E + \overline E|$, and relation \ref{seciso} implies this pair is not separated by any canonical system $|K_H|$, for any $H \in |E + \overline E|$ containing both of them. This forces $H$ to be an hyperelliptic curve, for any $H \in |E + \overline E|$ containing a pair $p, i(p)$ for some $p \in E$.\\
Let's count how many these curves are: a curve $H$ containing $p, i(p)$ lives in $$H \in |{\cal I}_{p, i(p)} (E + \overline E)|$$ where ${\cal I}_{p, i(p)} \subset {\cal O}_X$ is the ideal sheaf of the pair of points. The passage through both of them imposes two linear conditions on the space $H^0 ({\cal O}_X (E + \overline E))$, so we find a $2$-codimensional subspace in $|E + \overline E|$. As $p$ moves in $E$, this subspace moves inside $|E + \overline E|$, sweeping a divisor inside $|E + \overline E|$.
\begin{theorem}
Suppose $H$ is a nef linear system on a smooth projective surface $X$, with $H^2 \ge 9$, and that there exist two points $p, q \in X$ which are not separated by the adjoint linear system $|H  + K_X|$. Then there exists an effective divisor $D$ containing $p, q$, which satisfies one of the following cases:\\
$\cdot H D =0, D^2 = -1$ or $-2$;\\
$\cdot H D = 1, D^2 = 0$ or $-1$;\\
$\cdot H D = 2, D^2 = 0$;\\
$\cdot H = 3 D, D^2 =1$.
\end{theorem}
We now remember that $X$ can contain at most countably many rigid curves $D$, which are curves with $h^0 ({\cal O}_X (D)) = 1$. Hence, without loss of generality we can assume that:
\begin{eqnarray}\label{assumption} 
p, q  {\rm\, do\, not\, belong\, to\, any\, rigid\, curve}
\end{eqnarray} 
We will use this particular choice of points $p, i (p)$ to exclude most of these cases, but before we go on, we need the following technical result.
\begin{proposition}\label{hmmaguno}
We cannot have an effective linear system $|M|$ without base components such that $H M = 1$.
\end{proposition}
{\it Proof:} Suppose the existence of such an $|M|$. Let $$f: X \to \mathbb{P}^{k + 2}$$ be the map induced by $|H|$.The equality $$H M = 1$$ means the generic element in $|M|$ is birational to (hence, isomorphic to) a line $l \in \mathbb{P}^{k + 2}$. Thus the generic member of $|M|$ is a smooth rational curve. The vanishing $H^1 ({\cal O}_X)$ gives a short exact sequence:
\begin{eqnarray} \label{shes}
0 \to H^0 ({\cal O}_X) \to H^0 ({\cal O}_X (M)) \to H^0 ({\cal O}_M (M)) \to 0
\end{eqnarray}
with $$ H^0 ({\cal O}_M (M)) \simeq H^0 ({\cal O}_{\mathbb{P}^1} (M^2))$$ Since $|M|$ has no base components, we have $$M^2 \ge 0$$ We want to prove that both the cases $M^2 = 0$ and $M^2 \ge 1$ are impossible.\\
Suppose $$M^2 = 0$$ then by \ref{shes} we have that $|M|$ is a base point-free pencil, so there exists a regular map $$\pi : X \to \mathbb{P}^1$$ whose fibers are elements of $|M|$. Since $H M = 1$ and $H = E + \overline E$, we must have $E M =1, \overline E M = 0$ or vice versa. Thus $\overline E$ is contained in a fiber of $\pi$, while $E$ is a section of $\pi$. But then $E, \overline E$ could meet in at most one point, contradicting $E \overline E \ge 2$.\\
Suppose now that $$M^2 \ge 1$$ By \ref{shes}, this corresponds to $$h^0 ({\cal O}_X (M)) \ge 3$$ Now choose a curve $M_0 \in |M|$ and a point $x_0 \in X$ such that its $f(x_0)$ does not belong to the line $l_0 := f(M_0)$. Since $h^0 ({\cal O}_X (M)) \ge 3$, for any $y \in M_0$ there exists a curve $M_y \in |M|$ containing both $y, x_0$. Hence the line $f(M_y)$ lies inside $f(X)$, for every $y \in M_0$. This is possible only if $f(X)$ is equal to the plane spanned by $l_0, f(x_0)$, which is equivalent to $E \overline E = 2$, and this is forbidden by \ref{almenotre}. $\square$\\

The last case would imply $H^2 = 9$, which contradicts the parity of $H^2$.\\

Suppose we are in case $1$: let decompose $$D = F + M$$ where $F = Bs (D)$ is the fixed part of $D$, while $M$ is the mobile part of $D$. The relation $H D = 0$ forces also $$H F = 0$$ $$H M = 0$$ so both $F, M$ are contracted by $|H|$. Hence $|M|$ cannot move, otherwise the image of $|H|$ would not be $2$-dimensional.

Suppose then that $p, p'$ belong to an effective divisor $D$ such that $H D = 1$ and $D^2 = 0$ or $-1$. The same decomposition $D = F + M$ now leaves two possibilities: $$H F = 1, H M = 0$$ or $$H F = 0, H M = 1$$ Since $M$ can move, it cannot be contracted by $|H|$ unless $M = 0$. Thus $D = F$ is a rigid divisor, and since $H D = 1$, the image of $D$ is a smooth line $l \subset \mathbb{P}^{k + 2}$, so $D$ has the form $D = D_1 + D_2$, where $D_1$ is a rational normal curve, isomorphically mapped onto $l$, while $D_2$ is contracted. Again, this implies that the irreducible components of $D_2$ are $(-2)$-curves, which cannot contain $p, i(p)$ by assumption. Of course $D_1$ cannot move, otherwise $D$ could move as well. Note that the short exact sequence $$0 \to {\cal O}_X \to {\cal O}_X (D_1) \to {\cal O}_{D_1} (D_1) \to 0$$ remains exact in cohomology, so $$H^0 ({\cal O}_{D_1} (D_1)) = 0$$ which is the same as $$D_1^2 < 0$$ But this only happen if $$D_1^2 = -1 {\rm\, or\,} -2$$ which again contradicts the initial assumption \ref{assumption}.\\
The opposite case  $$H F = 0, H M = 1$$ is forbidden by Proposition \ref{hmmaguno}.\\
The unique remaining case in that $$H D = 2, D^2 = 0$$ which leaves three possibilities: $$H F = 0, H M = 2$$ or $$H F = 1, H M = 1$$ or $$H F = 2, H M = 0$$ Assume we have the first case, so that $$H F = 0, H M = 2$$ We first show that the generic element in $|M|$ must be irreducible. If this does not happen, by Bertini's Theorem $M$ must be composite with a pencil $\hat M$, and the relation $H M = 2$ forces $$M = 2 \hat M$$ and $$H \hat M = 1$$ which is forbidden by \ref{hmmaguno}.\\
Now that we have the irreducibility of $|M|$, its general member is either mapped birationally (and hence isomorphically) on smooth plane curves of degree $2$, or it is a double cover of a line in $\mathbb{P}^{k + 2}$. In the first situation the generic $M$ is such that $f(M)$ spans a plane inside $\mathbb{P}^{k + 2}$, that is:  
\begin{eqnarray}\label{conichepiane}
h^0 ({\cal O}_X (H)) - h^0 ({\cal O}_X (H - M)) = 3
\end{eqnarray}
The second case happens if $f(M)$ spans a line in $\mathbb{P}^{k + 2}$, that is: 
\begin{eqnarray}\label{doppiaretta}
h^0 ({\cal O}_X (H)) - h^0 ({\cal O}_X (H - M)) = 2
\end{eqnarray} 
In both cases, we must have $$h^0 ({\cal O}_X (M)) = 2$$ Indeed, since $M$ can move the inequality $h^0 ({\cal O}_X (M)) \ge 2$ is guaranteed. Assume that $h^0 ({\cal O}_X (M)) \ge 3$. Then we can argue similarly to Proposition \ref{hmmaguno}: any two points $x_1, x_2 \in X$ can be joined by an element of $|M|$. Consequently, if \ref{conichepiane} (respectively, \ref{doppiaretta}) holds, then any two points $y_1, y_2 \in f(X)$ can be joined by a plane conic curve (respectively, a line) which is entirely contained in $f(X)$. This cases can happen respectively if $f(X)$ has degree $2$ or $1$, but by hypothesis we have $H^2 \ge 10$, which means $\Deg\, f(X) \ge 5$.\\
So until now we proved that $$h^0 ({\cal O}_X (M)) = 2$$ and the generic element in $|M|$ is irreducible.\\
Now we note that no element of $|M|$ can entirely contain $C$. If this happens, then $$M = C + R$$ for some effective divisor $R$. But multiplying by $H$ we find $$2 = 4 + H R$$ a contradiction, since $H R$ must be non-negative.\\
We now show that \ref{conichepiane} is impossible: as we stated, this corresponds to say that the generic element of $|M|$ is birational via $f$ to a smooth plane conic. In particular, the generic element in $|M|$ is a smooth rational curve. Hence, we find a short exact sequence $$0 \to H^0 ({\cal O}_X) \to H^0 ({\cal O}_X (M)) \to H^0 ({\cal O}_{\mathbb{P}^1} (M^2)) \to 0$$ and hence $$M^2 = 0$$ Thus $|M|$ is a base-point free pencil of rational curves, so its elements are fibers of a regular map $\pi : X \to \mathbb{P}^1$. The adjunction formula $$M^2 + M K_X = -2$$ becomes $M K_X = -2$ so $$M C = 4$$ Thus $C$ is a $4$-section of $\pi$. Consider again the map $f : X \to \mathbb{P}^{k + 2}$ induced by $|H|$. Since $f(C)$ has degree at most $4$, we can write $$f(C) \subset \Gamma$$ for some $4$-dimensional linear subspace $\mathbb{P}^4 \simeq \Gamma \subset \mathbb{P}^{k + 2}$. But now a generic $X_\lambda \in |M|$ intersects $C$ at $4$ points, so the plane conic $f(X_\lambda)$ touches $\Gamma$ in $4$ different points. This happens only if $f(X_\lambda) \subset \Gamma$ for the generic $\lambda \in \mathbb{P}^1$, which forces $f(X) \subset \Gamma$. This is a contradiction, because the initial hypothesis $H^2 \ge 10$ corresponds to $dim\, |H| \ge 6$.\\
So we are in the case where the generic element of $|M|$ covers a line in $\mathbb{P}^{k + 2}$ with degree $2$. Again, an inequality $h^0 ({\cal O}_X (M)) \ge 3$ would imply that every couple of points in $X$ can be joint by an element of $|M|$, so $f(X)$ would be a plane, which is absurd. Hence, $$h^0 ({\cal O}_X (M)) = 2$$ Putting this data in Riemann - Roch Theorem we get $$2 - h^1 ({\cal O}_X (M)) = 1 + \frac{1}{2} (M^2 - M K_X)$$ The effectiveness of $|- 2 K_X|$ gives $M K_X \le 0$. Moreover, for the same reason as before, no member of $|M|$ can contain $C$, so inequality $$M K_X < 0$$ must be sharp. This leads to $$1 - h^1 ({\cal O}_X (M)) = \frac{1}{2} (M^2 - M K_X) > 0$$ So we have $$h^1 ({\cal O}_X (M)) = 0$$ and $$M^2 - M K_X = 2$$ which separates in two cases $$M^2 = 0, M K_X = -2$$ or $$M^2 = 1, M K_X = -1$$ In the first case $|M|$ is a basepoint-free pencil of smooth rational curves, in the second it's a pencil of elliptic curves with one base point.\\
In the elliptic case, note that $${\cal O}_M (M + K_X) = {\cal O}_M (K_M) = {\cal O}_M$$ so the short exact sequence $$0 \to {\cal O}_X (K_X) \to {\cal O}_X (M + K_X) \to {\cal O}_M (M + K_X)$$ induces an isomorphism $$H^0 ({\cal O}_X (M + K_X)) \simeq H^0 ({\cal O}_M) \simeq \complex$$ Thus $M + K_X$ is a rigid effective divisor, and the equality $$H (M + K_X) = 2 - 2 = 0$$ states that it is contracted via $f$. But we know that such a divisor is supported on a $(-1)$-curves. Moreover, the self intersection $$(M + K_X)^2 = -2$$ gives $$M + K_X = F_1 + F_2$$ with $F_1, F_2$ disjoint $(-1)$-curves. But now $E, F_1, F_2$ can be extended to a set of $10$ disjoint $(-1)$-curves $$E_1, \ldots, E_8, F_1, F_2$$ with $$E = E_1$$. Since also $\overline E$ is disjoint from $F_1, F_2$, its class in $\Pic (X)$ takes the form $$\overline E = d L - k_1 E_1 \cdots - k_8 E_8$$ with $$k_1 = k + 2$$. Since $\overline E$ is a $(-1)$-curve, the coefficients have to satisfy the system $$ \left\{ \begin{array} {lll}
d^2 - \sum_i k_i^2 = -1\\\\
-3 d + \sum_i k_i = -1
\end{array}
\right.
$$
This corresponds to look for $(-1)$-curves in a Del Pezzo surface of degree $1$, but they are finite and well known, and satisfy $k_1 \le 3$. But then $H^2 = 2 E \overline E -2 \le 4$, contradicting Reider hypothesis.\\
Hence the unique possible answer is that $|M|$ is a base point free pencil of rational curves. But such a pencil must have $9$ singular fibers, and each of them consists of a linear chain of two $(-1)$-curves, attached at one point.\\
Remember that $$(E + \overline E) M = 2$$ and that leaves only \begin{eqnarray}\label{duesez}E M = \overline E M = 1\end{eqnarray} Indeed, if we assume $$E M = 2, \overline E M = 0$$ then $\overline E$ must be contained in a fiber of $|M|$, while $E$ is a bisection of $|M|$. This would force $$E \overline E \le 2$$ which contradicts the hypothesis.\\
By reason of the Picard number, $|M|$ must have $9$ singular fibers. Relation \ref{duesez} implies that each of these contains a $(-1)$-curve not touched by $E$. Let $F_1, \cdots, E_9$ be these $(-1)$-curves. The simoultaneous contraction of $E, F_1, \ldots, F_9$ transforms $X$ in $\mathbb{P}^2$, so we can use $L, E, F_1, \ldots, F_9$ as base for $\Pic (X)$. In this basis we must have $$|M| = |L - E|$$ since $M$ is orthogonal to $F_1, \ldots, F_9$. Moreover, we write $$\overline E = d L - (k + 2) E - a_1 F_1 \cdots - a_9 F_9$$ But now we use again \ref{duesez} to deduce that $$d = k + 3$$ Since $k + 2$ is the maximum possible multiplicity for an irreducible curve of degree $k + 3$, this forces $$a_1, \ldots, a_9 \in \{0, 1\}$$ Let $b$ the number of $a_i$'s which are equal to $1$. Since $$\overline E^2 = -1$$ we also have $$b = 2 k + 6$$ which combined with $b \le 9$ gives $$k = 0, 1$$ Both of these two values contradict the Reider assumption $H^2 \ge 9$. Note nonetheless that we producted two expressions for $\overline E$ and $H$, which up to permutations are $$\overline E = 3 L - 2 E - E_1 \cdots - E_6, H = 3 L - E - E_1 \cdots - E_6 {\rm\, if\,} k = 0$$ and $$\overline E = 4 L - 3 E - E_1 \cdots - E_8, H = 4 L - 2 E - E_1 \cdots - E_8 {\rm\, if\,} k = 1$$
\subsection{The tangent behaviour}
On the curve $C \subset X$ we have the short normal exact sequence $$0 \to T_C \to {T_X}|_C \to N_{C / X} \to 0$$ We identify $C \simeq \mathbb{P}^1$, so that the tangent bundle $T_C$ on $C$ can be identified with $$T_C = T_{\mathbb{P}^1} = {\cal O}_{\mathbb{P}^1} (2)$$ and clearly $$N_{C / X} = {\cal O}_C (C) \simeq {\cal O}_{\mathbb{P}^1} (C^2) = {\cal O}_{\mathbb{P}^1} (- 4)$$ Since $$Ext^1 ({\cal O}_{\mathbb{P}^1} (- 4), {\cal O}_{\mathbb{P}^1} (2)) = 0$$ we have $${T_X}|_C = {\cal O}_{\mathbb{P}^1} (2) \oplus  {\cal O}_{\mathbb{P}^1} (- 4)$$
Let ${\cal E}$ denote the rank $2$ vector bundle $${\cal E} = {TX}|_C \otimes {\cal O}_{\mathbb{P}^1} (-2) =  {\cal O}_{\mathbb{P}^1} \oplus  {\cal O}_{\mathbb{P}^1} (-6)$$ Then the space $P = \mathbb{P} ({\cal E})$ is a $\mathbb{P}^1$-fibration over $\mathbb{P}^1$, corresponding to the sixth Hirzebruch surface $P = \mathbb{F}_6$. The standard hyperplane bundle ${\cal O}_P (1)$ defines an isomorphism $$H^0 (P, {\cal O}_P (1)) \simeq H^0 (\mathbb{P}^1, {\cal E}^*) = H^0 (\mathbb{P}^1, {\cal O}) \oplus H^0 ({\cal O}_{\mathbb{P}^1} (6))$$ The last space has dimension $8$ over $\complex$, so the linear system $|{\cal O}_P (1)|$ defines a map $P \to \mathbb{P}^7$. The section\\
Suppose $\alpha : X \to X$ is an automorphism such that $\alpha|_C = \mathbb{1}_C$, and suppose also that $\alpha$ is an involution, $$\alpha^2 = \mathbb{1}$$ For each point $p \in C$, the differential operator $d_p \alpha$ acts on the tangent space $T_p X$, so the operator $d \alpha$ defines a linear endomorphism of the vector bundle $$d \alpha : TX|_C \to TX|_C$$ The line subbundle ${\cal O}_{\mathbb{P}^1} (2) \subset {\cal O}_{\mathbb{P}^1} (2) \oplus {\cal O}_{\mathbb{P}^1} (- 4)$ corresponds to the tangent space of $C$, so it is composed of $(+ 1)$-eigenvectors of  $d \alpha$. Since $\alpha^2 = \mathbb{1}_X$, the other eigenvector on the points of $C$ is $-1$. Let ${\cal L} \subset TX|_C$ denote the line subbundle generated by these $(-1)$-eigenvectors at all points of $C$. Then ${\cal L}$ corresponds to a complementar subspace for $TC$ inside $TX|_C$, so $${\cal L} \simeq {\cal O}_{\mathbb{P}^1} (-4)$$ 
We remember now that the Coble surface $X$ is nothing but a blow up of $\mathbb{P}^2$ at the nodes of a rational sextic curve $\overline C$. Let $$\pi : X \to \mathbb{P}^2$$ the blow-down map, $$$$Let $F \in H^0 ({\cal O}_{\mathbb{P}^2} (6))$ be the equation of $\overline C$. We consider the dual curve $${\overline C}^* \subset {\mathbb{P}^2}^*$$ which is the image of $\overline C$ under the map $$\overline C \to {\mathbb{P}^2}^*$$ $$p \to T_p \overline C$$ In projective coordinates, this map has the form $$x \to [\frac{\partial F}{\partial Z_0} (x), \frac{\partial F}{\partial Z_1} (x), \frac{\partial F}{\partial Z_2} (x)]$$ The partial derivatives $\frac{\partial F}{\partial Z_i}$ form a net of quintic forms with $10$ base points, if $F$ is generic enough. Thus, these form cut over $\overline C$ a divisor $D$ of degree $5 \Deg\, \overline C = 30$, with a base locus of degree $2$ at each node of $\overline C$. Thus the mobile part of $D$ has degree $30 - 2 \cdot 10 =10$, which means that $$\Deg\, {\overline C}^* = 10$$
Now we look at the following rational map: $$\psi : \mathbb{P} ({TX}|_C) \dasharrow {\mathbb{P}^2}^*$$
$$(x, [v]) \to \{{\rm\, the\, unique\, line\, through\,} \pi (x) {\rm\, tangent\, to\,} d\,\pi_x (v)\} $$ 
This map possesses $20$ indeterminacy points: if $E_1, \cdots, E_{10}$ are the exceptional curves of $\pi$, and $$\{x_i, y_i\} = C \cap E_i$$ then ${\rm\, ker\,} d_{x_i}\, \pi : T_{x_i} X\ to T_{\pi (x_i)} \mathbb{P}^2 \ne 0$. Indeed, $${\rm\, ker\,} d_{x_i} = T_{x_i} E_i$$ and the same holds for $y_i$ too. By construction, the map $\psi$ is undefined at the points $(x_i, T_{x_i} E_i), (y_i, T_{y_i} E_i)$. Away from these points, $\psi$ is well defined. Then we blow up $\mathbb{F}_6 = \mathbb{P} ({TX}|_C)$ in these $20$ points, and we get another surface $\tilde {\mathbb{F}_6}$, with a resolution of indeterminacies $$\tilde \psi: \tilde {\mathbb{F}_6} \to {\mathbb{P}^2}^*$$
\newpage
\section{Acknowledgements}
I would like to express my warmest gratitude to my Advisor, Professor Alessandro Verra. This work would have never been possible without his precious support. During these three years, he provided irreplaceable contributions with his comments, inputs, and ability to show links between totally apparently different fields of Algebraic Geometry. His deepest knowledge of the topics we treated were very inspirational to me. Especially when I felt blocked, his human presence and fantasy to find new perspectives provided an incredible support.

\newpage
\addcontentsline{toc}{section}{References}
\bibliographystyle{plain}

\end{document}